\documentclass[a4,11pt]{article} \usepackage{amssymb}
 \usepackage{amsthm}\usepackage{amsmath}\usepackage{latexsym}   

\newtheorem{teo}{Theorem}[section]
\newtheorem{oss}[teo]{Remark}
\newtheorem{Prop}[teo]{Proposition}

\newtheorem{lemma}[teo]{Lemma}

\newtheorem{Defi}[teo]{Definition}
\newtheorem{es}[teo]{Example}
\newtheorem{corollario}[teo]{Corollary}
\newtheorem{no}[teo]{Notation}

\newcommand{\cc}{_{^{_\HH}}}
\newcommand{\vv}{_{^{_\VV}}}
\newcommand{\vs}{_{^{_{\VS}}}}

\newcommand{\res}{\mathop{\hbox{\vrule height 7pt width .5pt depth 0pt
\vrule height .5pt width 6pt depth 0pt\,}}\nolimits}

\def \DH{\textsl{h}}

\newcommand{\LL}{\mathop{\hbox{\vrule height .5pt width 6pt depth
0pt \vrule height 7pt width .5pt depth 0pt\,}}\nolimits}
\newcommand{\rr}{_{^{_\mathit{R}}}}

\newcommand{\ngr}{_{^{_{\mathrm Gr}}}}
\newcommand{\ck}{_{^{_{\HH_k}}}}
\newcommand{\ci}{_{^{_{\HH_i}}}}
\newcommand{\ctr}{_{^{_{\HH_3}}}}
\newcommand{\cd}{_{^{_{\HH_2}}}}
\newcommand{\cu}{_{^{_{\HH_1}}}}

\newcommand{\ciss}{_{^{_{\HH_i\!S}}}}

\newcommand{\cidu}{_{^{_{\HH_i\UU}}}}

\def \cont{{\mathbf{C}}}
\def \cin{{\mathbf{C}^{\infty}}}

\def\dim {\mathrm{dim}}
\def\dc {d_{\bf cc}}

\def\ss{_{^{_{\HS}}}}
\def\ts{_{^{_{\TS}}}}

\def\x{x}
\def\g{\mathit{g}\cc}

\def\dg{\textit{grad}\cc}
\def\qq{\textit{grad}\ss}

\def \nis {\sigma^{n-2}\cc}
\def \mis {\sigma^{n-i}\cc}
\def \per {\sigma^{n-1}\cc}
\def \perh {\sigma^{2n}\cc}
\newcommand{\ce}{_{^{_E}}}

\def\HC{\mathcal{H}^Q_{\mathbf{cc}}}

\def\SC{{C^{\gg}}}

\def\UU{\mathcal{U}}

\def \nS{\nu\cc{S}}
\def \nn{\nu\cc}

\def \nt{\nu^t\cc}

\def \XH{\mathfrak{X}(\HH)}
\def \XX{\mathfrak{X}}

\def \MH{\overrightarrow{\mathcal{H}\cc}}

\def \MS{\mathcal{H}\cc}

\def \P{{\mathcal{P}}}

\def \PH{\P\cc}

\def \Om{\Omega}

\def \ee{\mathrm{e}}

\def \Omn{\sigma^n\rr}
\def \R{\mathbb{R}}
\def \Rn{\mathbb{R}^{\DN}}

\def \div{\mathit{div}}

\def \GG{\mathbb{G}}
\def \gg{\mathfrak{g}}

\def \pert{(\per)_t}

\def\divh{\div\cc}

\def\lh{\mathcal{D}\ss}

\def\lg{\overline{\mathcal{D}}\ss}

\def\grr{{{\mathtt{gr}}}}

\def\tsc{\nabla^{^{_{\TT\!{S}}}}}
\def\gs{\nabla^{_{\HS}}}
\def\gc{\nabla^{_{\HH}}}

\def\BB{\mathit{D}}

\def \Tor{{\textsc{T}}}

\def\UU{\mathcal{U}}

\def\UU{\mathcal{U}}

\def \nS{\nu_{_{\!\HH}}{S}}
\def \nn{\nu_{_{\!\HH}}}

\def \ee{\mathrm{e}}

\def \Om{\Omega}
\def \Rn{\mathbb{R}^{n}}
\def \R{\mathbb{R}}
\def \N{\mathbb{N}}

\def \cji {c_{j\,i}(x)}
\def \C { C(x):=[\cji]_{j,i},\,\, {j=1,\ldots,m \,,\, i=1,\ldots,n}}

\def \Qdim {Q:=\sum_{i=1}^{k}i\,\DH_i}

\def \X {X=(X_{1}, \ldots, X_{m_1})}

\def \X0 {X_{1}(0)\!=\!\partial_{x_{1}}, \ldots, X_{m_1}(0)\!=\!\partial_{x_{m_1}}}

\def \HG {\HH\GG}
\def \HS {\HH\!{S}}
\def \VS {\VV\!S}

\def \TG {\mathit{T}\GG}
\def \HH {\mathit{H}}

\def \VV {\mathit{V}}
\def \TT {\mathit{T}}
\def \TS {\mathit{T}\!S}

\def \BV{\mathit{B\!V}}

\def \BB{\mathcal{B}}

\def \grad{\textit{grad}}
\def \C0H{\mathbf{C}_{0}^{\infty}(U,\HG)}
\def \C00{\mathbf{C}_{0}^{\infty}(U)}
\def \C01{\mathbf{C}_{0}^{1}(U)}

\def \L1{d\,\mathcal{L}^n}

\def \H1{\mathcal{H}_{{\bf cc}}^{1}}

\def \Ar{\mathcal{H}^{n-1}_{\bf eu}}

\def \exp{\textsl{exp\,}}

\def \llog{\textsl{log\,}}
\def \esp{\textsl{exp\,}}
\def \exsp{{\textit{exp}\,}}

\def \Om{\Omega}
\def \Rn{\mathbb{R}^{n}}
\def \R{\mathbb{R}}
\def \N{\mathbb{N}}

\def \cji {c_{j\,i}(x)}
\def \C { C(x):=[\cji]_{j,i},\,\, {j=1,\ldots,m \,,\, i=1,\ldots,n}}

\def \GG{\mathbb{G}}
\def \gg{\mathfrak{g}}

\def \X {X=(X_{1}, \ldots, X_{m_1})}

\def \X0 {X_{1}(0)\!=\!\partial_{x_{1}}, \ldots, X_{m_1}(0)\!=\!\partial_{x_{m_1}}}

\def \HG {\mathit{H}}

\def \BVG {\HH\!\BV}

\def \BV{\mathit{B\!V}}

\def \BB{\mathcal{B}}

\def \C0H{\mathbf{C}_{0}^{\infty}(\Om,\HG)}
\def \C00{\mathbf{C}_{0}^{\infty}(\Om)}
\def \C01{\mathbf{C}_{0}^{1}(\Om)}

\def \exp{\textsl{exp\,}}
\def \llog{\textsl{log\,}}
\def \esp{\textsl{exp\,}}

\def\GG{\mathbb{G}}

\title{Isoperimetric, Sobolev and Poincar\'e inequalities on hypersurfaces in sub-Riemannian Carnot groups}
\author{Francescopaolo Montefalcone
\thanks{F. M. is partially supported by University of Bologna,
Italy, founds for selected research topics and by GNAMPA of INdAM,
Italy.} }

\begin{document}
\maketitle

\begin{abstract}
In this paper we shall study smooth submanifolds immersed in a
$k$-step Carnot group $\GG$ of homogeneous dimension $Q$. Among
other results,  we shall prove  an isoperimetric inequality for
the case of a $\cont^2$-smooth compact hypersurface $S$ with - or
without - boundary $\partial S$; $S$ and $\partial S$ are endowed
with their homogeneous measures $\per,\,\nis$, actually equivalent
to the intrinsic $(Q-1)$-dimensional and $(Q-2)$-dimensional
Hausdorff measures w.r.t. some homogeneous metric $\varrho$ on
$\GG$; see Section 5. This generalizes a classical inequality,
involving the mean curvature of the hypersurface, proven by
Michael and Simon, \cite{MS}, and Allard, \cite{Allard}. In
particular, from this result one may deduce some related
Sobolev-type inequalities; see Section 7. The strategy of the
proof is inspired by the classical one. In particular, we shall
begin by proving some linear isoperimetric inequalities. Once this
is proven, one can deduce a local monotonicity formula and then
conclude the proof by a covering argument. We stress however that
there are many differences, due to our different geometric
setting. Some of the tools which have been developed ad hoc in
this paper are, in order, a ``blow-up'' theorem, which also holds
for characteristic points, and a smooth Coarea Formula for the
$\HS$-gradient; see Section \ref{blow-up} and Section \ref{COAR}.
Other tools are the horizontal integration by parts formula and
the 1st variation of the $\HH$-perimeter $\per$ already developed
in \cite{Monteb, MonteSem}, and here generalized to hypersurfaces
having non-empty characteristic set. Some natural applications of
these results are in the study of minimal and constant horizontal
mean curvature hypersurfaces. Moreover we shall prove some purely
horizontal, local and global Poincar\'e-type inequalities as well
as some related facts and consequences; see Section 4 and Section
5.
\end{abstract}

\tableofcontents

\date{}

\normalsize

\section{Introduction}

In the last  decades considerable efforts have been made to extend
to the general setting of metric spaces the methods of Analysis
and Geometric Measure Theory. This philosophy,  in a sense already
contained in Federer's treatise \cite{FE}, has been pursued, among
other authors, by Ambrosio \cite{A2}, Ambrosio and Kirchheim
\cite{AK1, AK2}, Capogna, Danielli and Garofalo \cite{CDG},
Cheeger \cite{Che}, Cheeger and Kleiner \cite{Cheeger1}, David and
Semmes \cite{DaSe}, De Giorgi \cite{DG}, Gromov \cite{Gr1, Gr2},
Franchi, Gallot and Wheeden \cite{FGW}, Franchi and Lanconelli
\cite{FLanc},  Franchi, Serapioni and Serra Cassano \cite{FSSC3,
FSSC6}, Garofalo and Nhieu \cite{GN}, Heinonen and Koskela
\cite{HaKo},  Jerison \cite{Jer}, Korany and Riemann \cite{KR},
Pansu \cite{P1, P2}, Rumin \cite{Rum}, but the list is far from
being complete.

In this respect, {\it sub-Riemannian} or {\it
Carnot-Carath\'eodory} geometries have become  a subject of great
interest also because of their connections with many different
areas of Mathematics and Physics, such as PDE's, Calculus of
Variations, Control Theory, Mechanics, Theoretical Computer
Science. For references, comments and perspectives, we refer the
reader to Montgomery's book \cite{Montgomery} and the surveys by
Gromov, \cite{Gr2} and Vershik and Gershkovich, \cite{Ver}. We
also mention, specifically for sub-Riemannian geometry,
 \cite{Stric} and \cite{P4}. More recently, the
 so-called Visual Geometry has also received new impulses from this
 field; see \cite{SCM}, \cite{CMS},  \cite{CM}  and references therein.

The setting of the sub-Riemannian geometry is that of a smooth
manifold $N$, endowed with a smooth non-integrable distribution
$\HH\subset\TT N$ of $\DH$-planes, or {\it horizontal subbundle}
($\DH\leq\dim N$). Such a distribution is endowed with a positive
definite metric $g\cc$, defined only on the subbundle $\HH$. The
manifold $N$ is said to be a {\it Carnot-Carath\'eodory space} or
{\it CC-space}  when one introduces the so-called {\it CC-metric}
$\dc$ (see Definition \ref{dccar}). With respect to such a metric
the only paths on the manifold which have finite length are
tangent to the distribution $\HH$ and therefore called {\it
horizontal}. Roughly speaking, in connecting two points we are
only allowed to follow horizontal paths joining them.

Throughout this paper we shall extensively study hypersurfaces
immersed in Carnot groups which, for this reason, form the
underlying ambient space. A $k$-{\it{step Carnot group}}
$(\GG,\bullet)$ is an $n$-dimensional, connected, simply
connected, nilpotent and stratified Lie group (with respect to the
group multiplication $\bullet$) whose Lie algebra $\gg\cong\Rn$
satisfies:\[ {\mathfrak{g}}={\HH}_1\oplus...\oplus {\HH}_k,\quad
 [{\HH}_1,{\HH}_{i-1}]={\HH}_{i}\quad(i=2,...,k),\,\,\,
 {\HH}_{k+1}=\{0\}.\]Any Carnot group $\GG$ on $\R^n$ is endowed with a one-parameter
family of dilations, adapted to the stratification, that makes it
a {\it homogeneous group}, in the sense of Stein's definition; see
\cite{Stein}.

Carnot groups are of special interest for many reasons and, in
particular, because they constitute a wide class of examples of
sub-Riemannian geometries.

Note that, by a well-know result due to Mitchell \cite{Mi} (see
also Montgomery's book \cite{Montgomery}), the {\it
Gromov-Hausdorff tangent cone} at the regular points of a
sub-Riemannian manifold is a Carnot group. This motivates the
interest towards the study of Carnot groups, which play, for
sub-Riemannian geometries, an analogous role to that of
 Euclidean spaces in Riemannian geometry.

The initial development of Analysis in this setting was motivated
by some works published in the first eighties. Among others, we
cite the paper by Fefferman and Phong \cite{FePh} about the
so-called ``subelliptic estimates'' and that of Franchi and
Lanconelli \cite{FLanc}, where a H\"{o}lder regularity theorem was
proven for a class of degenerate elliptic operators in divergence
form. Meanwhile, the beginning of Geometric Measure Theory was
perhaps an intrinsic isoperimetric inequality proven by Pansu in
his Thesis \cite{P1}, for the case of the {\it Heisenberg group}
$\mathbb{H}^1$. For more results about isoperimetric inequalities
on Lie groups and Carnot-Carath\'eodory spaces, see also
\cite{Varo}, \cite{Gr1}, \cite{P4}, \cite{GN}, \cite{CDG},
\cite{FGW}, \cite{HaKo}.

For results on these topics, and for more detailed bibliographic
references, we shall refer the reader to \cite{A2}, \cite{AR},
\cite{ASCV}, \cite{CDG}, \cite{FSSC4, FSSC5, FSSC6, FSSC7, FSSC8},
\cite{DanGarN8, gar, DGN3}, \cite{G}, \cite{GN}, \cite{Mag, Mag2},
\cite{Montea, Monteb}, \cite{HP2}, \cite{MontArFer}.

We would also like to quote \cite{vari}, \cite{G},
\cite{Pauls}, \cite{RR},  \cite{CCM, CCM2},   \cite{Bar}, \cite{MoSC2}, for
some recent results about minimal and constant mean-curvature hypersurfaces
immersed in the Heisenberg group.

In this paper we shall try to give some contribution to the study
of both analytic and differential-geometric properties of
hypersurfaces  immersed in Carnot groups, endowed with the
so-called $\HH$-perimeter measure $\per$. To this aim we will
develop some technical tools and among the others, in Section
\ref{blow-up}, a ``blow-up'' result, which also holds for
characteristic points and, in Section \ref{COAR},  a Coarea
Formula for functions supported on the given hypersurface.
Furthermore, we shall extend some previous results of \cite{Monte,
Monteb, MonteSem}, such as the integration by parts formulae and
the 1st-variation formula of $\per$, to hypersurfaces having
non-empty characteristic set. These tools are then used to
investigate the validity of some linear, local and global
Poincar\'e inequalities; see Section \ref{poincinsect} and Section
\ref{wlinisoineq}.  In Section \ref{mike}, we will extend to
Carnot groups a classical isoperimetric inequality, proven by
Michael and Simon in \cite{MS}, for a general setting including
Riemannian geometries, and, independently,
 by Allard in \cite{Allard}, for varifolds.
Finally, in Section \ref{sobineqg}, we shall deduce the related
Sobolev-type inequality, following a classical pattern by
Federer-Fleming \cite{FedererFleming} and Mazja \cite{MAZ}.

Very recently, some similar results in this direction have also
been obtained by Danielli, Garofalo and Nhieu in \cite{DGN3},
where a monotonicity estimate for the $\HH$-perimeter has been
proven for graphical strips in the Heisenberg group
$\mathbb{H}^1$.

Now we would like to make a short comment about a classical
isoperimetric inequality for compact hypersurfaces immersed in the
Euclidean space $\Rn$. This result, that will be
generalized to our Carnot setting in Section \ref{mike}, reads as follows:\\

 \noindent{\bf Theorem
(Isoperimetric Inequality; Case $S^{n-1}\subset\Rn$)}\, {\it Let
$S\subset\Rn\,(n>2)$ be a $\cont^2$-smooth compact hypersurface
with - or without - boundary. Then
\[\{\sigma^{n-1}\rr(S)\}^{\frac{n-2}{n-1}}\leq C_I\Big\{\int_S|\mathcal{H}\rr|\,\sigma^{n-1}\rr+\sigma^{n-2}\rr(\partial
S)\Big\}\]where $C_I>0$ is a dimensional constant.}\\

 Here
$\mathcal{H}\rr$ denotes the mean curvature of $S$ and
$\sigma^{n-1}\rr, \,\sigma^{n-2}\rr$ are, respectively, the
$(n-1)$-dimensional and the $(n-2)$-dimensional Riemannian
measures.
 The first step in the proof of this theorem is a linear inequality. More
 precisely, one proves that
\[\sigma^{n-1}\rr(S)\leq R\Big\{\int_S|\mathcal{H}\rr|\,\sigma^{n-1}\rr+\sigma^{n-2}\rr(\partial
S)\Big\}\]where $R$ is the radius of a Euclidean ball containing
$S$. From this linear inequality and the Coarea Formula one gets
the so-called {\it monotonicity inequality}, which says that, at
every ``density-point''\footnote{By definition, $x\in S$ is a {\it
density-point} if $\lim_{t\searrow
0^+}\frac{\sigma^{n-1}(S_t)\rr}{t^{n-1}}=\omega_{n-1}$ where
$\omega_n$ denotes the Lebesgue measure of the unit ball in
$\R^{n-1}$.} $x\in S$, one has

\[-\frac{d}{dt}\frac{\sigma^{n-1}\rr(S_t)}{t^{n-1}}\leq
\frac{1}{t^{n-1}}\bigg\{\int_{S_t}|\mathcal{H}\rr|\,\sigma^{n-1}\rr
+ \sigma^{n-2}\rr(\partial S\cap B(x,t))\bigg\}
\]for $\mathcal{L}^1$-a.e. $t>0$, where $S_t=S\cap B(x, t)$.
By the monotonicity inequality, via a contradiction argument,  one
deduces a calculus lemma
 which, together with a well-known
covering result of Vitali-type, allows to achieve the proof.

The importance of the monotonicity estimate can also be understood
through one of its immediate consequences, that is the following
asymptotic  exponential estimate: \[\sigma^{n-1}\rr(S_t)\geq
\omega_{n-1}\, t^{n-1} e^{-\mathcal{H}_0 t}\]for $t\rightarrow
0^+$, where $x\in S$ is a density-point and $\mathcal{H}_0$ is any
constant such that $|\mathcal{H}\rr|\leq\mathcal{H}_0$. Note that
for minimal hypersurfaces this immediately implies  that
\[\sigma^{n-1}\rr(S_t)\geq \omega_{n-1}\, t^{n-1}\qquad(t\rightarrow
0^+).\]

As already said, a great part of this paper is concerned about the
generalization of the above results to hypersurfaces
immersed in Carnot groups.\\

 Henceforth, we shall survey, in detail, the content of each
 section.\\

Section \ref{0prelcar} is largely devoted to introduce the subject
of Carnot groups and the study of immersed hypersurfaces (and,
more generally, submanifolds) of Carnot groups. In particular, we
shall introduce and describe the main geometric structures useful
in this setting from many points of views, including basic facts
about stratified and homogeneous Lie groups, Riemannian and
sub-Riemannian geometries, intrinsic measures and connections. The
presentation given here is, as much as possible, self-contained.
Moreover, Section \ref{Intbypart} and Section \ref{prvar} describe
two important tools in the study of hypersurfaces endowed with the
so-called $\HH$-perimeter $\per$ (see Definition \ref{sh}), i.e.
the horizontal divergence theorem  w.r.t $\per$, its related
integration by parts formula and the 1st variation of $\per$. A
great part of this material can also be found in \cite{Monte,
Monteb, MonteSem}. However, there is some novelty in the
presentation given here. Indeed, some new corollaries are added
and all these results are extended to hypersurfaces having
possibly non-empty characteristic set. Here we would like to state
the horizontal integration by parts formula because of its
importance in this context.

Let $\GG$ be a  $k$-step Carnot group and let $S\subset\GG$ be a
$\cont^2$-smooth immersed hypersurface with boundary $\partial S$
oriented by its unit normal vector $\eta$. Moreover, let $\HS$
denote the horizontal tangent bundle of $S$, which is the
$(\DH-1)$-dimensional subbundle of $\TS$ induced by the horizontal
$\DH$-dimensional bundle $\HH\subset\TG$. Then, for every $X\in
\cont^1(S,\HH)$ one has
\[ \int_{S}\big\{ \div\ss X + \langle C\cc\nn,
X\rangle\big\}\per =-\int_{S}\langle X, \MH\rangle\,\per +
\int_{\partial S}\langle X,\eta\ss\rangle\,{\nis},\]where $\per$
and $\nis$ denote, respectively, the ``natural'' intrinsic
homogeneous measures on smooth $(n-1)$-dimensional and
$(n-2)$-dimensional submanifolds of $\GG$; see Section \ref{dere}
and Section \ref{Intbypart}. In the above formula $\nn$ is the
unit horizontal normal along $S$, $\eta\ss$ is the unit horizontal
normal along $\partial S$ (see Remark \ref{measonfr}), $\MH=\MS
\nn$ is the horizontal mean curvature vector of $S$ (see
Definition \ref{curvmed})  and $C\cc$ is a linear operator acting
on horizontal vectors, only depending on the structural constants
of the Lie algebra $\gg=\TT_0\GG$ and on the
 hypersurface $S$, through its - Riemannian - unit normal; see Definition \ref{notazia}. The
 horizontal linear
 operator $C\cc$ plays an important role in the study of hypersurfaces
 and, in particular, it is connected with the
 skew-symmetric part of the horizontal second fundamental form; see
 Remark \ref{mariodecandia}.

The above formula enable us to define the following important 2nd
order differential operator:
\begin{eqnarray*}\lh\varphi:=\Delta\ss\varphi +\langle C\cc\nn, \qq\varphi\rangle
\end{eqnarray*} for every $\varphi\in\cin(S)$, where, by definition,
$\Delta\ss\varphi:=\div\ss(\qq\varphi)$; see Section \ref{hs} and
Definition \ref{DEFIEIGEN}. We note that the operator $\lh$
represents a sub-Riemannian generalization of the usual
Laplace-Beltrami operator; see Proposition \ref{Properties} for
its main features.

Section \ref{blow-up}  contains some ``blow-up'' results for the
horizontal perimeter $\per$ on any smooth hypersurface
$S\subset\GG$, or, in other words, the study of the following
limit:\[\lim_{R\rightarrow 0^+}\frac{\per \{S\cap
B_{\varrho}(x,R)\}}{R^{Q-1}},\]where $B_{\varrho}(x,R)$ is the
$\varrho$-ball of center $x\in S$ and radius $R$ and $\varrho$ is
a smooth homogeneous distance on $\GG$. Notice that this limit
represents the density of $\per$ at $x\in S$. More precisely,
after reminding a well-known blow-up procedure, valid for
non-characteristic points of any (smooth enough) hypersurface $S$
(see, for instance, \cite{FSSC3, FSSC4, FSSC5}, \cite{balogh},
\cite{Mag, Mag2}, \cite{CM2}), we shall generalize it, under some
regularity assumptions on $S$, also to the case of characteristic
points of $S$; see Theorem \ref{BUP}. We remind that the
characteristic set $C_S$ of $S$ is the set of all points at which
the horizontal projection of the unit normal vanishes, i.e.
$C_S=\{x\in S: |\PH\nu|=0\}$. More precisely, let $x\in C_S\cap S$
and assume that, locally around $x$,  $S$ can be represented as a
$\cont^i$-smooth $X_\alpha$-graph, for some ``vertical'' direction
$X_\alpha\in\VV:=\HH^\perp$. For the sake of simplicity, let
$x=0\in\GG.$ In such a case one has $S\cap
B_\varrho(x,r)\subset\exp\{(\zeta_1,...,\zeta_{\alpha-1},
\psi(\zeta),\zeta_{\alpha+1},...,\zeta_n )\in\gg\},$ where $\psi$
is a $\cont^i$-function satisfying:
\[\frac{\partial^{\scriptsize(l)}
\psi}{\partial\zeta_{j_1}...\partial\zeta_{j_l}}(0)=0\qquad\mbox{whenever}\quad{\rm
ord}(j_1)+...+{\rm ord}(j_l)< i \] for all $l\in\{1,...,i\}$.
Then, we shall show that\footnote{The symbol $\sim$ means
``asymptotic''.}
\[\per(S\cap B_\varrho(x,R))\sim
\kappa_\varrho(C_S(x)) R^{Q-1}\qquad\mbox{for}\quad R\rightarrow
0^+\] where the constant $\kappa_\varrho(C_S(x))$ can be
explicitly computed by integrating $\per$ along a polynomial
hypersurface of ``anisotropic'' order $i={\rm ord}(\alpha)$; see
Theorem \ref{BUP}, Case (b).

In Section \ref{COAR} we shall state and discuss another important
tool, i.e. the Coarea Formula for the $\HS$-gradient, that is an
equivalent for smooth functions of the classical {\it
Fleming-Rischel} formula. More precisely, let $S\subset\GG$ be a
$\mathbf{C}^2$-smooth hypersurface and let
$\varphi:S\longrightarrow\R$ be a $\mathbf{C}^1$-smooth function.
Then, one has
\[\int_{S}|\qq\varphi(x)|\per(x)=\int_{\R}\nis\{\varphi^{-1}[s]\cap
S \}ds.\]We stress that $\per, \nis$  are actually equivalent to
the $(Q-1)$-dimensional and $(Q-2)$-dimensional Hausdorff measures
associated to a given homogeneous distance $\varrho$ on $\GG$,
where $Q\,(> n)$ is a positive integer representing the
homogeneous dimension of $\GG$; see Section \ref{prelcar} and
Section \ref{dere}.

In Section \ref{poincinsect} we shall prove the validity of some
global Poincar\'e-type inequalities on smooth compact
hypersurfaces with - or without - boundary, immersed into $\GG$.
Here we just stress the relationship among these inequalities and
some eigenvalues problems related to the 2nd order differential
operator $\lh$ previously introduced; see Definition \ref{epr}.

To be more precise, we shall first define some suitable
isoperimetric constants, in purely geometric terms, and then show
that these constants equal the infimum of their associated
``horizontal'' Rayleigh's quotients, i.e.\[\frac{\int_S|\qq
\psi|\per}{\int_S|\psi|\per}.\] For instance, if $\partial S\neq
\emptyset$, let us set
$${\rm Isop}(S):=\inf\frac{\nis(N)}{\per(S_1)},$$where
$N\cap \partial S=\emptyset$ and $S_1$ is the unique
$(n-2)$-dimensional submanifold of $S$ such that  $N=\partial
S_1$. Then, we will prove that $${\rm
Isop}(S)=\inf\frac{\int_S|\qq \psi|\per}{\int_S|\psi|\per},$$
where the infimum is taken over all $\cont^1$-smooth functions on
$S$ such that $\psi|_{\partial S}=0$. This constant is related to
the first non-zero eigenvalue $\lambda_1$ of the following
Dirichlet-type problem:
\begin{displaymath}
{\rm(P_2)}\,\,\,\left\{%
\begin{array}{ll}
    \quad\,\,\quad\lh\psi=\lambda\, \psi\, \qquad(\lambda\in\R,\,\,\psi\in\mathbf{C}^2(S)); \\
    \quad\quad\,\,\,\,\,  \psi|_{\partial S} =0; \\
\end{array}%
\right.\end{displaymath}see Definition \ref{epr}. Indeed, it turns
out that \[\lambda_1\geq \frac{({\rm Isop}(S))^2}{4};\] see
Corollary \ref{555}. Furthermore we shall prove some results about
another isoperimetric constant; see, for fore more details,
Theorem \ref{zdaz} and Corollary \ref{1zdaz}. The results stated
in Section \ref{poincinsect} are direct consequences of the Coarea
Formula for the $\HS$-gradient. The proofs follow the scheme of
the Riemannian case, for which we refer the reader to a celebrated
paper by Yau \cite{Yau}; see also \cite{Ch1, Ch3}.

In Section \ref{wlinisoineq} we will try to generalize to our
Carnot setting some classical differential-geometric results; see,
for instance,  \cite{BuZa}
 and references therein, but also \cite{Yano} and \cite{Ch2}.
 The starting point is an integral formula very similar to the Euclidean  Minkowsky
 Formula; see also Corollary \ref{IPPH0} for a precise statement.  More precisely, we shall
 show that\[(\DH-1)\,\per(S)\leq
R\Big\{\int_S\big(|\MS|+ |C\cc\nn|\big)\per+ \nis(\partial
S)\Big\}\]whenever $S\subset\GG$ is a smooth compact hypersurface
with boundary and $R$ denotes the radius of a $\varrho$-ball
circumscribed about $S$.

 From the above linear isoperimetric inequality, we shall deduce a number of simple geometric
 consequences. Among other results, Section \ref{wmf} contains a weak monotonicity
 inequality for the $\HH$-perimeter $\per$; see Proposition \ref{in}. In Section
 \ref{wme} we shall develop some weak Poincar\'e-type
 inequalities,
 depending on the local geometry of the hypersurface $S$; see Theorem \ref{0celafo} and Theorem \ref{1celafo}.
For instance, let $S\subset\GG$ be a  $\cont^2$-smooth
hypersurface with bounded horizontal mean curvature $\MS$. Then,
we shall prove that for every $x\in S$ there exists $R_0\leq {\rm
dist}_\varrho(x,\partial S)$ such that
\[
\Big(\int_{S_R}|\psi|^p\per\Big)^{\frac{1}{p}}\leq C_p\,
R\Big(\int_{S_R}|\qq\psi|^p\per\Big)^{\frac{1}{p}} \qquad
p\in[1,+\infty[
\]for all
$\psi\in\cont_0^1(S_R)$ and all $R\leq R_0$, where $S_R:=S\cap
B_\varrho(x,R)$.

 These
 results are obtained by means of elementary ``linear'' estimates starting from the
  horizontal integration by parts formula, together with a
 careful analysis of the role played by the characteristic set $C_S$ of $S$.

Section \ref{mike} contains perhaps the main result of this paper,
i.e. an isoperimetric inequality for compact hypersurfaces with -
or without - boundary, depending on the horizontal mean curvature
$\MS$ of the hypersurface. As  already said in our introduction,
this generalizes a classical and well-known result by Michael and
Simon \cite{MS} and Allard \cite{Allard}.

To state the claimed inequality, let us first introduce some
notation. Let $\GG$ be a $k$-step Carnot group and let
$S\subset\GG$ be a $\cont^2$-smooth compact hypersurface with - at
least - $\cont^1$-smooth boundary $\partial S$. Moreover, let us
set (see, for more details,  Remark \ref{indbun}):
\begin{itemize}\item $\varpi:=\frac{\P\vv\nu}{|\PH\nu|}=\sum_{i=2}^k\varpi\ci$, where $\nu$ is the
Riemannian unit normal vector along $S$. Moreover,  $\PH, \P\ci$
and $\P\vv$ denote, respectively, the projection operators onto
$\HH$, $\HH_i\,(i=2,...,k)$ and
$\VV:=\HH^\perp={\HH}_2\oplus...\oplus {\HH}_k$.

\item $\chi:=\frac{\P\vs\eta}{|\P\ss\eta|}$, where $\eta$ is the
Riemannian unit normal vector along $\partial S$ and $\P\ss,
\P\vs$ denote, respectively, the projection operators onto $\HS$
and $\VS$, which are the subbundles on $\TS$ induced by the
decomposition $\gg=\HH\oplus \VV$.

\item $\chi\ciss:=\P\ciss\chi$, where, as above, $\P\ciss$ denote
the projection operators onto the subbundles $\HH_iS$ of
$\VS\subset\TS\,\,(i=2,...,k)$.
\end{itemize}

Furthermore, we shall assume that $|x\cc|\leq \varrho(x)$ and that
$c_i\in\R_+\,\,(i=2,...,k)$ be positive constants satisfying:
\[|x\ci|\leq c_i
\varrho^i(x)\qquad(i=2,...,k),\]where $\varrho$ is  a given smooth
homogeneous distance on $\GG$ and where we have set
$\varrho(x)=\varrho(0,x)$.\\

Our main result reads as follows:\\

 \noindent{\bf
Theorem\,(Isoperimetric inequality)}\,\,{\it Let $\GG$ be a
$k$-step Carnot group and let $S\subset\GG$ be a $\cont^2$-smooth
compact hypersurface with $\cont^1$-smooth boundary $\partial S$.
Then\[\big\{\per(S)\big\}^{\frac{Q-2}{Q-1}}\leq
C_I\bigg\{\int_{S}|\MS|\bigg(1+\sum_{i=2}^k
i\,c_i\varrho_S^{i-1}|\varpi\ci|\bigg)\per+ \int_{\partial
S}\Big(1+\sum_{i=2}^k
i\,c_i\varrho_S^{i-1}|\chi\ciss|\Big)\nis\bigg\},
\]where $C_I$ is a constant independent of
$S$; see Theorem \ref{isopineq}.}\\

In the above formula we have set $\varrho_S=\frac{\rm{diam}_\varrho(S)}{2}$.\\

\noindent If we further assume that $\dim\,C_{\partial S}< n-2$,
then we shall see that there exists $C'_I>0$ such that
\[\big\{\per(S)\big\}^{\frac{Q-2}{Q-1}}\leq
C'_I\bigg\{\int_{S}|\MS|\Big(1+\sum_{i=2}^k
i\,c_i\varrho_S^{i-1}|\varpi\ci|\Big)\per + \nis(\partial
S)\bigg\};
\]see Corollary \ref{isopineqcaso2}.
The proof of these results is heavily inspired from the classical
one, for which we refer the reader to the book by Burago and
Zalgaller \cite{BuZa}. A similar strategy was also fundamental in
proving isoperimetric and Sobolev inequalities in abstract metric
setting such as weighted Riemannian manifolds and graphs; see
\cite{CGY}.

 However, we have to say that there are still many non-trivial
modifications to be done, due to the sub-Riemannian setting.
Roughly speaking, the starting point will be again a linear
inequality, yet somewhat different from those, purely horizontal,
proven in Section \ref{wlinisoineq}; see , for a more precise
statement, Section \ref{wlineq}, Proposition \ref{l1am}. This
linear-type inequality is then used to obtain a monotonicity
estimate for the $\HH$-perimeter $\per$; see Proposition
\ref{rmonin}. As in the Euclidean (and/or Riemannian) case, the
{\it monotonicity inequality} is a differential inequality,
expressing the local behavior of the derivative of the quotient
$$\frac{\per \{S\cap B_{\varrho}(x,t)\}}{t^{Q-1}}\qquad(t>0)$$ for
$t\searrow 0^+$, whenever $x\in S\setminus C_S$ is
non-characteristic; see Section \ref{wlineq}.  We will also
discuss the possibility of performing such arguments for
characteristic points, at least in some special cases; see, for
instance, Lemma \ref{condcar77}.

 Section \ref{isopineq1} is then
devoted to the proof  of the aforementioned isometric inequality.

 In Section \ref{asintper}  we shall discuss a
straightforward application of the monotonicity estimate. More
precisely, let $S\subset\GG$ be a $\cont^2$-smooth hypersurface
that, for simplicity, is taken to be without boundary.
Furthermore, let us assume that the horizontal mean curvature
$\MS$ be bounded. Then, for every $x\in{S}\setminus C_{S}$, we
shall show that
\[\per({S}_t)\geq \kappa_{\varrho}(\nn(x))\,t^{Q-1} e^{-t
\{\MS^0(1+O(t))\}}\]for $t\rightarrow 0^+$, where
$\kappa_{\varrho}(\nn(x))$ denotes the ``density'' of $\per$ at
$x\in {S}\setminus C_{S}$, also called {\it metric factor}. Again,
we shall tract the case in which $x\in C_S$, where $S$ is immersed
in the Heisenberg group ${\mathbb{H}^n}$.

 In Section \ref{sobineqg} we
shall discuss the equivalent Sobolev-type inequalities which can
be deduced by the previous isoperimetric inequality, following a
well-known and classical argument by Federer-Fleming
\cite{FedererFleming} and Mazja \cite{MAZ}; see Theorem
\ref{sobolev}. Some corollaries will be proven, and among others,
we shall show the following:\\

\noindent{\bf Corollary\,(Sobolev-type inequality)}\,\,{\it Let
$\GG$ be a $k$-step Carnot group. Let $S\subset\GG$ be a
$\cont^2$-smooth hypersurface without boundary and let $\MS$ be
globally bounded along $S$. Furthermore, let us assume that for
every smooth $(n-2)$-dimensional submanifold $N\subset S$ one has
$\dim\, C_{N}<n-2.$ Then there exists $C'_1>0$, only depending on
the geometry of $S$, such that
\[\bigg\{\int_S|\psi|^{\frac{Q-1}{Q-2}}\per\bigg\}^{\frac{Q-2}{Q-1}}\leq
C'_1\,\int_{S}\big\{|\psi|+ |\qq\psi|\big\}\per
\]
for all $\psi\in\cont_0^1(S)$.}\\

See Theorem \ref{gsobin} for a more detailed statement.

\section{Carnot
groups, submanifolds and measures}\label{0prelcar}
\subsection{Sub-Riemannian Geometry of Carnot groups}\label{prelcar}
In this section, we will introduce the definitions and the main
features concerning the sub-Riemannian geometry of Carnot groups.
References for this large subject can be found, for instance, in
\cite{CDG}, \cite{GE1}, \cite{GN}, \cite{Gr1, Gr2}, \cite{Mag},
\cite{Mi}, \cite{Montgomery}, \cite{P1, P2, P4}, \cite{Stric}. Let
$N$  be a ${\cin}$-smooth connected $n$-dimensional manifold and
let $\HH\subset \TT N$ be an $\DH$-dimensional smooth subbundle of
$\TT N$. For any $x\in N$, let $\TT^{k}_x$ denote the vector
subspace of $\TT_x N$ spanned by a local basis of smooth vector
fields $X_{1}(x),...,X_{\DH}(x)$ for $\HH$ around $x$, together
with all commutators of these vector fields of order $\leq k$. The
subbundle $\HH$ is called {\it generic} if, for all $x\in N$,
$\dim \TT^{k}_x$ is independent of the point $x$ and {\it
horizontal} if $\TT^{k}_x = \TT N$, for some $k\in \N$. The pair
$(N,\HH)$ is a {\it $k$-step CC-space} if is generic and
horizontal and if $k:=\inf\{r: \TT^{r}_x = \TT N \}$. In this case
\begin{equation}\label{filtration}0=\TT^{0}\subset
\HH=\TT^{\,1}\subset\TT^{2}\subset...\subset \TT^{k}=\TT
N\end{equation} is a strictly increasing filtration of {\it
subbundles} of constant dimensions $n_i:=\dim
\TT^{i}\,(i=1,...,{k}).$ Setting $(\HH_i)_x:=\TT^i_x\setminus
\TT^{i-1}_x,$ then $\grr(\TT_x N):=\oplus_{i=1}^k (\HH_k)_x$ is
its associated {\it graded Lie algebra}, at the point $x\in N$,
with Lie product induced by $[\cdot,\cdot]$. We set $\DH_i:=\dim
{\HH}_{i}=n_i-n_{i-1}\,(n_0=\DH_0=0)$ and, for simplicity,
$\DH:=\DH_1=\dim\HH$. The $k$-vector
$\overline{\DH}=(\DH,...,\DH_{k})$ is the {\it growth vector} of
$\HH$.

\begin{Defi} $\underline{X}=\{X_1,...,X_n\}$
is a {\bf graded frame} for $N$  if $\{{X}_{i_j}(x):
n_{j-1}<i_j\leq n_j\}$ is a basis for ${\HH_j}_x$ for all $x\in N$
and all $j\in\{1,...,k\}$.\end{Defi}

\begin{Defi}\label{dccar} A {\bf sub-Riemannian metric} $g\cc=\langle\cdot,\cdot\rangle\cc$ on $N$ is a
symmetric positive bilinear form on $\HH$. If $(N,\HH)$ is a
{CC}-space, the {\bf {CC}-distance} $\dc(x,y)$ between $x, y\in N$
is given by
$$\dc(x,y):=\inf \int\sqrt{\langle
 \dot{\gamma},\dot{\gamma}\rangle\cc} dt,$$
where the infimum is taken over all piecewise-smooth horizontal
paths $\gamma$ joining $x$ to $y$.
\end{Defi}

In fact, Chow's Theorem implies that $\dc$ is a true metric on $N$
and that any two points can be joined with - at least one -
horizontal path. The topology induced on $N$ by the {CC}-metric is
equivalent to the standard manifold topology; see \cite{Gr1},
\cite{Montgomery}.

The general setting introduced above is the starting point of
sub-Riemannian geometry. A nice and very large class of examples
of these geometries is represented by {\it Carnot groups} which,
for many reasons, play  in sub-Riemannian geometry  an analogous
role to that of Euclidean spaces in Riemannian geometry. Below we
will introduce their main features. For the geometry of Lie groups
we refer the reader to Helgason's book \cite{Helgason} and
Milnor's paper \cite{3}, while, specifically for  sub-Riemannian
geometry, to Gromov, \cite{Gr1}, Pansu, \cite{P1, P4}, and
Montgomery, \cite{Montgomery}.

A $k$-{\it{step Carnot group}}  $(\GG,\bullet)$ is an
$n$-dimensional, connected, simply connected, nilpotent and
stratified Lie group (with respect to the multiplication
$\bullet$) whose Lie algebra $\gg(\cong\Rn)$
satisfies:\begin{equation} {\mathfrak{g}}={\HH}_1\oplus...\oplus
{\HH}_k,\quad
 [{\HH}_1,{\HH}_{i-1}]={\HH}_{i}\quad(i=2,...,k),\,\,\,
 {\HH}_{k+1}=\{0\}.\end{equation}
We denote by $0$ the identity on $\GG$ and so $\gg\cong\TT_0\GG$.
The smooth subbundle
  ${\HH}_1$ of the tangent bundle $\TG$ is said to be
{\it horizontal}  and henceforth denoted by $\HH$. We set
${\VV}:={\HH}_2\oplus...\oplus {\HH}_k$ and call ${\VV}$ the {\it
vertical subbundle} of $\TG$. Moreover, we set
$\DH_i=\dim{{\HH}_i}$ $(i=1,...,k)$.  $\HH$ is generated by a
frame $\underline{X\cc}:=\{X_1,...,X_{\DH}\}$ of left-invariant
vector fields which can be completed to a global
 graded and left-invariant frame,
  say $\underline{X}:=\{X_i:
 i=1,...,n\}$. We set
 $n_l:=\DH+...+\DH_l$ ($n_0=\DH_0:=0,\DH=\DH_1,\,n_k=n$) and
 $${\HH}_l={\mathrm{span}}_\R\big\{X_i:  n_{l-1}< i \leq
 n_{l}\big\}.$$
The standard basis $\{\ee_i:i=1,...,n\}$ of $\Rn\cong\gg$ can be
relabelled to be graded or {\it adapted to the stratification}.
Any left-invariant vector field of $\underline{X}$ is given by
${X_i}(x)={L_x}_\ast\ee_i\,(i=1,...,n)$, where ${L_x}_\ast$
denotes the differential of the left-translation at $x$.

\begin{no}\label{1notazione}In the sequel, we shall set $I\cc:=\{1,...,h_1\}$,
$I\cd:=\{h_1+1,...,n_2(=h_1+h_2)\}$,...,
$I\ck:=\{n_{k-1}+1,...,n_k(=n)\}$, and $I\vv:=\{h_1+1,...,n\}$.
Moreover, we will use Latin letters $i, j, k,...,$ for indices
belonging to $I\cc$ and Greek letters
 $\alpha, \beta, \gamma,...,$ for indices belonging to $I\vv$. Unless otherwise  specified, capital Latin letters $I, J,
 K,...,$ may denote any generic index. Finally, we define the
  function $\mathrm{ord}:\{1,...,n\}\longrightarrow\{1,...,k\}$
 by $\mathrm{ord}(I):= i $ if, and only if, $n_{i-1}<I\leq n_{i}$  $\,(i=1,...,k)$.
\end{no}

If $x\in\GG$ and $X\in\gg$ we set
${\gamma^{_X}_{x}}(t):=\esp[tX](x)\, (t\in\R)$, i.e.
${\gamma^{_X}_{x}}$ is the integral curve of $X$ starting from $x$
and it is a {\it{1-parameter subgroup}} of $\GG$. The {\it{Lie
group exponential map}} is then defined by
$$\exp:\gg\longmapsto\GG,\quad\exp(X):=\esp[X](1).$$ It turns out
that $\exp$ is an analytic diffeomorphism between $\gg$ and $\GG$
whose inverse will be denoted by $\llog.$ Moreover we have
$${\gamma^{_X}_{x}}(t)=x\bullet\exp(t
X)\quad\forall\,\,t\in\R.$$From now on we shall fix on $\GG$ the
so-called {\it exponential coordinates of 1st kind} which are the
coordinates associated to the map $\llog$.

As for any nilpotent Lie group, the {\it Baker-Campbell-Hausdorff
formula} (see \cite{Corvin}) uniquely determines the group
multiplication $\bullet$ of $\GG$, from the ``structure'' of its
own Lie algebra $\gg$. In fact, one has
$$\exp(X)\bullet\exp(Y)=\exp(X\star Y)\,\,(X,\,Y \in\gg),$$ where
${\star}:\gg \times \gg\longrightarrow \gg$ is the
 {\it Baker-Campbell-Hausdorff product} defined by \begin{eqnarray}\label{CBHformula}X\star Y= X +
Y+ \frac{1}{2}[X,Y] + \frac{1}{12} [X,[X,Y]] -
 \frac{1}{12} [Y,[X,Y]] + \mbox{ brackets of length} \geq 3.\end{eqnarray}

Using exponential coordinates, \eqref{CBHformula} implies that the
group multiplication $\bullet$ of $\GG$ is polynomial and
explicitly computable; see \cite{Corvin}. Moreover,
$0=\exp(0,...,0)$ and the inverse of ${p}\in\GG$
 $(x=\exp(x_1,...,x_{n}))$ is ${x}^{-1}=\exp(-{x}_1,...,-{x}_{n})$.

Whenever $\HH$ is endowed with a metric
$g\cc=\langle\cdot,\cdot\rangle\cc$, we say that $\GG$
 has a {\it sub-Riemannian structure}.  It is always possible to define a left-invariant Riemannian metric
  $g =\langle\cdot,\cdot\rangle$ on $\GG$ such that $\underline{X}$ is {\it
orthonormal} and $g_{|\HH}=\g$. Notice that if we fix a Euclidean
metric on $\gg=\TT_0\GG$ (which makes $\{\ee_i: i=1,...,n\}$ an
orthnormal basis), this metric naturally extends to the whole
tangent bundle, by left-translation.

Since Chow's Theorem holds true for Carnot groups, the {\it
Carnot-Carath\'eodory distance} $\dc$ associated with $g\cc$ can
be defined. The pair $(\GG,\dc)$ turns out to be a complete metric
space on which every couple of points can be joined by - at least
- one $\dc$-geodesic.

Carnot groups are {\it homogeneous groups}, i.e. they are equipped
with a 1-parameter group of automorphisms
$\delta_t:\GG\longrightarrow\GG$ $(t\geq 0)$; see \cite{Stein}. By
using exponential coordinates, it turns out that $\delta_t x
:=\exp(\sum_{j,i_j}t^j\,x_{i_j}\ee_{i_j})$, where
$x=\exp(\sum_{j,i_j}x_{i_j}\ee_{i_j})\in\GG.$\footnote{Here,
$j\in\{1,...,k\}$ and $i_j\in I{\!_{^{_{{\HH}_j}}}}=\{
n_{j-1}+1,..., n_{j}\}$.} Furthermore, we shall set
$\widehat{\delta}_t(\llog x):=\sum_{j,i_j}t^j\,x_{i_j}\ee_{i_j}$
to denote the induced dilations on the Lie algebra $\gg$. The {\it
homogeneous dimension} of $\GG$ is then the integer $\Qdim$,
coinciding with the {\it Hausdorff dimension} of $(\GG,\dc)$ as a
metric space; see \cite{Mi}, \cite{Montgomery}, \cite{Gr1}.

\begin{Defi}A continuous distance
$\varrho:\GG\times\GG\longrightarrow\R_+\cup\{0\}$ is called {\bf
homogenous} if one has
\begin{itemize}\item[{\bf(i)}]$\varrho(x,y)=\varrho(z\bullet x,z\bullet
y)$ for every $x,\,y,\,z\in\GG$;
\item[{\bf(ii)}]$\varrho(\delta_tx,\delta_ty)=t\varrho(x,y)$ for
all $t\geq 0$. \end{itemize}\end{Defi}

The CC-distance $\dc$ is the fundamental example of a homogeneous
distance on any Carnot group $\GG$. Another example of a
homogeneous distance for Carnot groups has been recently defined
by Franchi, Serapioni and Serra Cassano and can be found in the
Appendix of \cite{FSSC5}. We stress that {\it on every Carnot
group there exists a smooth, subadditive and homogeneous norm};
see \cite{HeSi}. This means that there exists a function
$\|\cdot\|_\varrho:\GG\times\GG\longrightarrow\R_+\cup \{0\}$ such
that:\begin{itemize}\item[{\bf(i)}]$\|x\bullet
y\|_\varrho\leq\|x\|_\varrho+\|y\|_\varrho$;
\item[{\bf(ii)}]$\|\delta_tx\|_\varrho=t\|x\|_\varrho\quad (t\geq
0)$;\item[{\bf(iii)}]$\|x\|_\varrho=0\Leftrightarrow x=0$;
\item[{\bf(iv)}]$\|x\|_\varrho=\|x^{-1}\|_\varrho$;
\item[{\bf(v)}]$\|\cdot\|_\varrho$ is continuous;
\item[{\bf(vi)}]$\|\cdot\|_\varrho$ is smooth on $\GG\setminus 0$.
\end{itemize}

For instance, a homogeneous norm $\varrho$ which is smooth on
$\GG\setminus \{0\}$ can be defined by
$$\|x\|_\varrho:=(|x\cc|^{\lambda}+|x\cd|^{\lambda/2}+|x\ctr|^{\lambda/3}+...+|x\ck|^{\lambda/k}
)^{1/\lambda},$$where $\lambda$ is a positive number evenly
divisible by $i$, for $i=1,...,k$. Here $|x\ci|$ denotes the
Euclidean norm of the projection of $x$ onto the i-th layer
$\HH_i$ of the stratification of $\gg\,\,(i=1,...,k)$.

Having a Riemannian metric, we define the left-invariant co-frame
$\underline{\omega}:=\{\omega_I:I=1,...,n\}$ dual to
$\underline{X}$. In particular, the {\it left-invariant 1-forms}
\footnote{That is, $L_p ^{\ast}\omega_I=\omega_I$ for every
$p\in\GG.$} $\omega_i$ are uniquely determined by the condition:
$$\omega_I(X_J)=\langle X_I,X_J\rangle=\delta_I^J\qquad (I,
J=1,...,n)$$ where $\delta_I^J$ denotes the Kronecker delta. We
remind that the {\it structural constants} of the Lie algebra
$\gg$ associated with the left invariant frame $\underline{X}$ are
defined by
$$\SC^R_{IJ}:=\langle [X_I,X_J],
 X_R\rangle\qquad(I,J, R=1,...,n)$$
\noindent {and satisfy:
\begin{itemize}\item [{\bf (i)}]\, $\SC^R_{IJ} +\SC^R_{JI}=0,$\,\,(skew-symmetry) \item
[{\bf(ii)}]\, $\sum_{J=1}^{n} \SC^I_{JL}\SC^{J}_{RM} +
\SC^I_{JM}\SC^{J}_{LR} + \SC^I_{JR}\SC^{J}_{ML}=0$\,\,(Jacobi's
identity).\end{itemize} \noindent The stratification hypothesis on
the Lie algebra implies the following important property:
\begin{equation} X_i\in {\HH}_{l},\, X_j \in
{\HH}_{m}\Longrightarrow [X_i,X_j]\in {\HH}_{l+m}.\end{equation}
Therefore, if $i\in I{\!_{^{_{{\HH}_s}}}}$ and $j\in
I{\!_{^{_{{\HH}_r}}}}$, one has:
\begin{equation}\SC^m_{ij}\neq 0 \Longrightarrow m \in I{\!_{^{_{{\HH}_{s+r}}}}}.\end{equation}

\begin{Defi}\label{nota}Throughout this paper we shall make
use of the following notation:\begin{itemize}\item[{\bf(i)}]
$C^\alpha\cc:=[\SC^\alpha_{ij}]_{i,j\in
I\cc}\in\mathcal{M}_{h_1\times h_1}(\R)\,\,\,\qquad(\alpha\in
I\cd)$; \item[{\bf(ii)}] $
C^\alpha:=[\SC^\alpha_{IJ}]_{I,J=1,...,n}\in \mathcal{M}_{n\times
n}(\R)\qquad(\alpha\in I\vv).$\end{itemize}The linear operators
associated with these matrices will be denoted in the same way.
\end{Defi}

\begin{Defi}\label{twist}
The $i$-{\bf{th curvature}} of the distribution $\HH$
$(i=1,...,{k})$ is the antisymmetric, bilinear map, given by
$$\Om\ci:\HH\otimes{\HH}_{i}\longrightarrow {\HH}_{i+1},\qquad
\Om\ci(X\otimes Y):=[X,Y] \mod {\TT}^{i} \qquad\forall\,
X\in\HH,\,\forall\,  Y\in{\HH}_{i}.$$
\end{Defi}From the very definition of
$k$-step Carnot group, it follows that $\Om\ck(\cdot,\cdot)=0$.
Note the 1st curvature $\Om\cc(\cdot,\cdot):=\Om\cu(\cdot,\cdot)$
coincides with the curvature of a distribution in the sense of
\cite{GE1}, \cite{Gr2}, \cite{Montgomery}.

\begin{no} If $\, Y\in\TG$, we denote by
$Y=(Y_1,...,Y_k)$  its canonical decomposition w.r.t. the grading
of the tangent space, i.e. $Y=\sum_{i=1}^k\P\ci(Y)$, where $\P\ci$
denotes the orthogonal projection onto $\HH_i\,\,(i=1,...,k)$. We
also set $\Om\vv(X,Y):=\sum_{i=1}^{k-1}\Om\ci(X,Y_i)$ for every
$X\in \HH,\, Y\in\TG$.\end{no} The next lemma relates the above
notions.
\begin{lemma}\label{twist1}Let $X\in \HH$ and $Y, Z\in \TG$. Then
we have\begin{itemize}\item[{\bf(i)}]
$\langle\Om\cc(X,Y),Z\rangle=-\sum_{\alpha\in I\cd}\langle
Z,X_\alpha\rangle\langle C^\alpha\cc X, Y\rangle;$
\item[{\bf(ii)}] $\langle\Om\vv(X,Y),Z\rangle=-\sum_{\alpha\in
I\vv}\langle Z,X_\alpha\rangle\langle C^\alpha X,
Y\rangle.$\end{itemize}
\end{lemma}
\begin{proof}The proof is a straightforward  consequence of Definition \ref{twist} and Definition
\ref{nota}.

\end{proof}

\begin{Defi}\label{parzconn}
We shall denote by $\nabla$ the - unique - {\bf left-invariant
Levi-Civita connection} on $\GG$ associated with $g$. Moreover, if
$X, Y\in\cin(\GG,\HH)(:=\XH)$, we shall set $\gc_X Y:=\PH(\nabla_X
Y)$. We stress that $\gc$ is an example of partial connection,
called {\bf horizontal $\HH$-connection}; see \cite{Monteb} and
references therein.
\end{Defi}

\begin{oss}\label{flatness}From
Definition \ref{parzconn}, using the properties of structural
constants of any Levi-Civita connection, we get that the
horizontal connection $\gc$ is {\bf flat}, i.e.
$\gc_{X_i}X_j=0\,\, (i,j\in I\cc).$ Note that  $\gc$ is {\bf
compatible  with the sub-Riemannian metric} $g\cc$, i.e.
$$X\langle Y, Z \rangle\cc=\langle \gc_X Y, Z \rangle\cc
+ \langle  Y, \gc_X Z \rangle\cc\qquad \forall\,\,X, Y, Z\in
\XH.$$This follows immediately from the very definition of $\gc$
and the corresponding properties of the Levi-Civita connection
$\nabla$ on $\GG$. Furthermore, $\gc$ is {\bf torsion-free}, i.e.
$$\gc_X Y - \gc_Y X-\PH[X,Y]=0\qquad \forall\,\,X, Y\in \XH.$$For the global left-invariant frame
$\underline{X}=\{X_1,...,X_{n}\}$ it turns out that
\begin{equation}\label{c2}\nabla_{X_I} X_J =
\frac{1}{2}\sum_{R=1}^n( \SC_{IJ}^R  - \SC_{JR}^I + \SC_{RI}^J)
X_R\qquad (I,\,J=1,...,n).\end{equation}
\end{oss}

\begin{Defi}
If $\psi\in\cin({\GG})$ we define the {\bf horizontal gradient} of
$\psi$,  $\dg \psi$, as the - unique - horizontal vector field
such that
$$\langle\dg \psi,X \rangle\cc= d \psi (X) = X
\psi\quad \forall \,X\in \XH.$$  The {\bf horizontal divergence}
of $X\in\XH$, $\divh X$, is defined, at each point $x\in \GG$, by
$$\divh X(x):= \mathrm{Trace}\big(Y\longrightarrow \gc_{Y} X
\big)(x)\quad(Y\in \HH_x).$$
\end{Defi}

\label{}

We end this section with some examples.
\begin{es}[Heisenberg group $\mathbb{H}^n$] \label{epocase}Let $\mathfrak{h}_n:=\TT_0\mathbb{H}^n=\R^{2n + 1}$
denote the Lie algebra of the Heisenberg group $\mathbb{H}^n$ that
is the most important example of $2$-step Carnot group. Its Lie
algebra $\mathfrak{h}_n$ is defined by the
rules\begin{equation}\label{poe}[\ee_{i},\ee_{i+1}]=\ee_{2n +
1}\end{equation}where $i=2k+1$ and $k=0,...,n-1$, and all other
commutators are 0. One has $\mathfrak{h}_n=\HH\oplus \R\ee_{2n+1}$
where $\HH={\rm span}_{\R}\{\ee_i:i=1,...,2n\}.$ In particular,
the 2nd layer of the grading $\R\ee_{2n+1}$ is the center of the
Lie algebra $\mathfrak{h}_n$. These conditions determine the group
law $\bullet$ via the Baker-Campbell-Hausdorff formula. More
precisely, if $x=\exp(\sum_{i=1}^{2n+1}x_iX_i),\,
y=\exp(\sum_{i=1}^{2n+1}y_i X_i)\in \mathbb{H}^n$, it turns out
that\begin{center}$x\bullet y =\exp \Big(x_1 + y_1,...,x_{2n} +
y_{2n}, x_{2n+1} + y_{2n+1} + \frac{1}{2}\sum_{k=1}^{n} (x_{2k-1}
y_{2k}- x_{2k} y_{2k-1})\Big).$\end{center}\end{es}

\begin{es}[Engel group $\mathbb{E}^1$]\label{e} The Engel group is
a simple but significant example of a $3$-step Carnot group. Its
Lie algebra $\mathfrak{e}$ is $4$-dimensional and can be defined
by the rules:
$$[\ee_{1},\ee_{2}]=\ee_{3},\,\,[\ee_{1},\ee_{3}]=[\ee_{2},\ee_{3}]=\ee_{4}.$$The other commutators vanish. One has
$\mathfrak{e}=\HH\oplus\R \ee_3\oplus\R\ee_4$,
 where $\HH=\rm{span}_{\R}\{\ee_1,\ee_2\}$. The center of $\mathfrak{e}$ is ${\R}\ee_4.$
\end{es}

\subsection{Hypersurfaces, $\HH$-regular submanifolds and
measures}\label{dere}

In the sequel, $\mathcal{H}^m_{\bf cc}$ and $\mathcal{S}^m_{\bf
cc}$ will denote, respectively, the usual and the spherical
Hausdorff measures associated with the CC-distance\footnote{We
remind that:\begin{itemize}\item[{\rm(i)}]
$\mathcal{H}_{{\bf{cc}}}^{m}({S})=\lim_{\delta\to 0^+}\mathcal
{H}_{{\bf{cc}},\delta}^{m}({S})$  where, up to a constant
multiple,
$${\mathcal{H}}_{{\bf{cc}},\delta}^{m}({S})=
\inf\Big\{\sum_i\Big(\mathrm{diam}\cc ({C}_i)\Big)^{m}:\;{S}
\subset\bigcup_i {C}_i;\;\mathrm{diam}\cc ({C}_i)<\delta\Big\}$$
and the infimum is taken w.r.t. any non-empty family of closed
subsets $\{{C}_i\}_i\subset\GG$;\item[{\rm(ii)}]
$\mathcal{S}_{{\bf{cc}}}^{m}({S})=\lim_{\delta\to 0^+}\mathcal
{S}_{{\bf{cc}},\delta}^{m}({S})$  where, up to a constant
multiple,
$${\mathcal{S}}_{{\bf{cc}},\delta}^{m}({S})=
\inf\Big\{\sum_i\Big(\mathrm{diam}\cc ({B}_i)\Big)^{m}:\;{S}
\subset\bigcup_i {B}_i;\;\mathrm{diam}\cc ({B}_i)<\delta\Big\}$$
and the infimum is taken w.r.t. closed $d\cc$-balls
${B}_i$.\end{itemize}}.

The - left-invariant - {\it Riemannian volume form} on $\GG$ is
defined as $\sigma^n\rr:=\Lambda_{i=1}^n\omega_i\in
\Lambda^n(\TG).$

\begin{oss} By integrating $\sigma^n\rr$ we obtain a measure $d{vol}^n\rr$,
which is the {\bf Haar measure} of $\GG$; see \cite{FE}. Since the
determinant of ${L_{x}}_{\ast}$ is equal to 1 $(x\in\GG)$, this
measure equals that induced on $\GG$ by the push-forward of the
$n$-dimensional Lebesgue measure $\mathcal{L}^n$ on $\Rn\cong\gg$.
Moreover, up to a constant multiple, $d{vol}^n\rr$ equals the {\it
$Q$-dimensional Hausdorff measure} $\HC$ on $\GG$. This is because
they are both Haar measures for $\GG$ and therefore equal, up to a
constant; see \cite{Montgomery}. We shall assume this constant is
equal to $1$.
\end{oss}

In the study of hypersurfaces of Carnot groups we need the notion
of {\it characteristic point}.
\begin{Defi}\label{caratt}
Let $S\subset\GG$ be a $\mathbf{C}^r$-smooth $(r=1,...,\infty)$
hypersurface. Then we say that $x\in S$ is a {\bf characteristic
point} of $S$  if $\dim\,\HH_x = \dim (\HH_x \cap \TT_x S)$ or,
equivalently, if $\HH_x\subset\TT_x S$. The {\bf characteristic
set} of $S$ is denoted by $C_S$. One has$$ C_S:=\{x\in S :
\dim\,\HH_x = \dim (\HH_x \cap \TT_x S)\}.$$
\end{Defi}
Thus, a hypersurface $S\subset\GG$, oriented by its unit normal
vector $\nu$, is {\it non-characteristic} if, and only if, the
horizontal subbundle $\HH$ is {\it transversal} to $S$
($\HH\pitchfork \TS$). We have then
$$\HH_p\pitchfork\TT_x S\Longleftrightarrow \PH \nu_x\neq 0\Longleftrightarrow\exists
  X\in\XH:
 \langle X_x, \nu_p\rangle \neq 0\qquad(x\in S),$$ where
 $\PH:\TG\longrightarrow\HH$ denotes the orthogonal projection onto
 $\HH$.

\begin{oss}[Hausdorff measure of $C_S$; see \cite{Mag}]
If $S\subset\GG$ is a $\mathbf{C}^1$-smooth hypersurface, then the
intrinsic  $(Q-1)$-dimensional Hausdorff measure of the
characteristic set $C_S$  vanishes, i.e. $\mathcal{H}_{\bf
cc}^{Q-1}(C_S)=0.$\end{oss}

\begin{oss}[Riemannian measure on hypersurfaces]
Let $S\subset\GG$ be a $\mathbf{C}^1$-smooth hypersurface and let
$\nu$ denote the unit normal vector along $S$. The
$(n-1)$-dimensional Riemannian measure along $S$ can be defined
by\begin{equation}\label{misup}\sigma^{n-1}\rr\res
S:=(\nu\LL\sigma^{n}\rr)|_{S},\end{equation} where  $\LL$ denotes
the contraction, or interior product, of a differential
 form\footnote{The linear map $\LL:
\Lambda^k(\TG)\rightarrow\Lambda^{k-1}(\TG)$ is defined, for
$X\in\TG$ and $\omega^k\in\Lambda^k(\TG)$, by $(X \LL \omega^k)
(Y_1,...,Y_{k-1}):=\omega^k (X,Y_1,...,Y_{k-1})$; see
\cite{Helgason}, \cite{FE}.}.\end{oss}

\begin{Defi}\label{9}Let $U\subseteq\GG$ be open and $f\in L^1(U).$
We say that $f$ has $\HH${\bf-bounded variation} on $U$ if the
following holds:
$$|{\gc}f|\cc(U):=\sup\Big\{ \int_{U} f\,\div\cc Y\,\L1:
Y\in{\mathbf C}_0^1({U},\HH), \; |Y|\cc\leq 1 \Big\}<\infty.$$ Let
$\BVG(U)$ denote the vector space of {\bf bounded $\HH$-variation
functions in ${U}$}. By Riesz' Theorem it turns out that
$|{\gc}f|\cc$ is a Radon measure on ${U}$ and that there exists a
horizontal $|{\gc} f|\cc$-measurable vector field $\nu_f$ such
that $|\nu_{f}|\cc=1$ for $|{\gc}f|\cc$-a.e. $p\in U$ and that
$$\int_{U} f\,\div\cc Y\,\L1 = \int_{U} \langle
Y,\nu_f\rangle\cc\,d\,|{\gc} f|\cc \quad \forall\,\,Y\in {\mathbf
C}_0^1({U},\HH).$$ We say that a measurable set $E\subset\GG$ has
finite $\HH$-perimeter in $U$ if ${\chi}_E\in\BVG(U)$. The
$\HH${\bf -perimeter} of $E$ in $U$ is the Radon measure
$|\partial E|\cc(U):= |{\gc}{\chi}_E|\cc(U)$. We call {\bf
generalized unit $\HH$-normal along $\partial E$} the Radon
$\R^{\DH}$-measure $\nu\ce:=-\nu_{{{\chi}}_E}$; see \cite{A2},
\cite{CDG}, \cite{FSSC3, FSSC4, FSSC5}, \cite{GN}.\end{Defi}

Just as in \cite{Monte, Monteb, MonteSem} (see also \cite{vari},
\cite{HP2}, \cite{RR}), since we are studying smooth
hypersurfaces, instead of the previous ``weak'' definition of
$\HH$-perimeter measure,  we shall define an $(n-1)$-differential
form which, by integration, coincides with the $\HH$-perimeter
measure.

\begin{Defi}[$\per$-measure on hypersurfaces]\label{sh}
Let $S\subset\GG$ be a $\mathbf{C}^1$-smooth non-characteristic
hypersurface and let us denote by $\nu$ its unit normal vector. We
will call {\bf unit $\HH$-normal} along $S$, the normalized
projection of $\nu$ onto $\HH$, i.e.$$\nn:
=\frac{\PH\nu}{|\PH\nu|}.$$ We then define the $(n-1)$-dimensional
measure $\per$ along ${S}$ to be the measure associated with the
$(n-1)$-differential form $\per\in\Lambda^{n-1}(\TS)$ given by
contraction of the volume form $\sigma^n\rr$ of $\GG$ with the
horizontal unit normal $\nn$, i.e.\begin{equation}\per \res
S:=(\nn \LL \sigma^n\rr)|_S.\end{equation} If we allow $S$ to have
characteristic points we may trivially extend the definition of
$\per$ by setting $\per\res C_{S}= 0$. Note that $\per \res S =
|\PH \nu |\cdot\sigma^{n-1}\rr\, \res S$.\end{Defi}

From the definition we immediately get that
$$\per \res S=\sum_{i\in I\cc} {\nn^i}\,(X_i \LL \Omn)|_S =\sum_{i\in I\cc}
{\nn^i}\,\ast\omega_i|_S= \sum_{i\in I\cc}(-1)^{i+1}{\nn^i}\,
(\omega_1\wedge...\wedge\widehat{\,\omega_i\,}\wedge...\wedge\omega_n
)|_{ S},$$ where ${\nn^i} :=\langle\nn, X_i\rangle\,\,(i\in
I\cc)$. The symbol $\ast$ denotes the {\it Hodge star} operation;
see \cite{Helgason}, \cite{Hicks}.
\begin{oss} We would like to point out that
\begin{equation}\label{IPREGCDG}\per(S\cap U)=\int_{S\cap U}\sqrt{ \langle
X_1,\textsl{n}_{\bf e}\rangle_{\Rn}^2+...+\langle
X_{m_1},\textsl{n}_{\bf e}\rangle_{\Rn}^2}\,d\,\Ar,
\end{equation}
where ${\textsl{n}_{\bf e}}$ denotes the unit Euclidean  normal
along $S$, and that its unit $\HH$-normal is given
by$${\nu\cc}=\frac{(\langle {X_1},{\textsl{n}_{\bf e}
}\rangle_{\Rn}, ...,\langle {X_{\DH}},{\textsl{n}_{\bf e}
}\rangle_{\Rn})}{\sqrt{\langle {X_1},{\textsl{n}_{\bf e}
}\rangle_{\Rn}^2+...+\langle {X_{\DH}},{\textsl{n}_{\bf e}
}\rangle_{\Rn}^2}}.$$Here, the Euclidean normal $\textsl{n}_{\bf
e}$ along $S$ and the vector fields
 $X_i\,(i\in I\cc)$ of the horizontal left-invariant frame $\underline{X\cc}$,
are thought of as vectors in $\Rn$, endowed with its canonical
inner product $\langle\cdot,\cdot\rangle_{\Rn}$. We note that the
- Riemannian -  unit normal $\nu$ along $S$ may be represented
 w.r.t. the global left-invariant frame
$\underline{X}$ for $\GG$, in terms of the Euclidean normal
${\textsl{n}_{\bf e} }$. One has$$\nu(x)=
\frac{{L_{x}}_{\ast}\textsl{n}_{\bf
e}(x)}{|{L_{x}}_{\ast}\textsl{n}_{\bf e}(x)|}\qquad (x\in
S),$$where
${L_{x}}_{\ast}(\cdot)=[X_1(\cdot),...,X_n(\cdot)]\in\mathcal{M}_{n\times
n}(\R)$.
\end{oss}

The comparison between different notions of measures on
submanifolds, is a basic problem of GMT in the setting of
Carnot-Carath\'eodory spaces; see \cite{balogh}, \cite{BTW},
\cite{Mag, Mag2}, \cite{FSSC3, FSSC4, FSSC5}. In the case of
smooth hypersurfaces in Carnot groups, one may compare the
$\HH$-perimeter measure with the $(Q-1)$-dimensional Hausdorff
measure associated with either the CC-distance $\dc$ or with any
homogeneous distance $\varrho$. In general, thanks to a remarkable
density estimate for $\HH$-perimeter measure proved in \cite{A2},
one can prove the following:
\begin{teo}\label{magnanoi}If ${S}\subset\GG$ is a $\mathbf{C}^1$-smooth hypersurface  which is locally
the boundary of an open set $E$ having locally finite
$\HH$-perimeter, then
 \begin{equation}\label{hausd}
 |\partial E|\cc(\BB)=k_\varrho(\nn)\,\mathcal{S}_{\bf cc}^{Q-1}\res
 ({S}\cap\BB)\quad \forall\,\,\BB\in \mathcal{B}or(\GG),
 \end{equation}where $k_{\varrho}(\nn)$  depends on $\nn$ and on the fixed homogeneous metric
$\varrho$ on $\GG$
\end{teo}

This ``density'' is sometimes called {\it metric factor}; see
\cite{Mag}.  Because of regularity of $\partial E$, it turns out
that $$|\partial E|\cc(\BB)=\per\res (S\cap\BB)$$ for all $\BB\in
\mathcal{B}or(\GG)$. A proof of the previous theorem can be found
in \cite{Mag}. We shall extensively discuss the previous result
throughout Section \ref{blow-up}.

\begin{Defi}\label{carca}
If $\nn$ is the horizontal unit normal along $S$, at each regular
point $x\in S\setminus {\it C}_S$ one has that $\HH_x= (\nn)_x
\oplus \mathit{H}_x S$, where we have set $\mathit{H}_x
S:=\HH_x\cap\TT_x S; $ $\mathit{H}_x S$ is the horizontal tangent
space at $x$ along $S$. Moreover, we define in the obvious way the
associated subbundles $\HS (\subset \TS)$ and $\nn S$, called,
respectively, {\bf horizontal tangent bundle} and {\bf horizontal
normal  bundle} of $S$.
\end{Defi}

\begin{oss}[Induced stratification on $\TS$; see \cite{Gr1}]\label{indbun}The stratification of the Lie algebra $\gg$ induces a ``natural'' decomposition
of the tangent space of any smooth submanifold of $\GG$. Let us
analyze the case of a hypersurface $S\subset\GG$. More precisely,
let us intersect, at each point $x\in S$, the tangent spaces
$\TT_x S\subset\TT_x\GG$ with $\TT_x^i=\oplus_{j=1}^i(\HH_j)_x$.
We shall set $\TT^iS:=\TS\cap\TT^i\GG,$ $n'_i:=\dim\TT^iS$ and
$\HH_iS:=\TT^iS\setminus \TT^{i-1}S$. Note that $\HS=\HH_1S$. From
the previous construction it follows that
$\TS:=\oplus_{i=1}^k\HH_iS$ and that $\sum_{i=1}^kn'_i=n-1.$
Henceforth, we shall set $\VS:=\oplus_{i=2}^k\HH_iS$. One can show
that the Hausdorff dimension of any smooth hypersurface $S$,
w.r.t. the CC-distance $\dc$, is given by the number
$Q-1=\sum_{i=1}^k i\,n'_i$; see \cite{Gr1}, \cite{P4},
\cite{FSSC4, FSSC5}, \cite{Mag, Mag2, Mag3, Mag4}, \cite{HP2}. We
would like to stress that, if at each point $x\in S$, the
horizontal tangent bundle $\HS$ is generic and horizontal, then
the couple $(S, \HS)$ is a $k$-step CC-space; see Section
\ref{prelcar}. More generally, we stress that, at the
characteristic set of a smooth hypersurface $S\subset\GG$, the
dimension of the horizontal tangent bundle $\HS$ coincides with
that of $\HH$, i.e. $n'_1=\DH.$ This, in particular, implies that
only  a non-characteristic domain\footnote{Domain means open and
connected.} $\UU\subseteq S$ can be generic at each point.

\end{oss}

\begin{es}Let us consider the case of a smooth hypersurface $S\subset\mathbb{H}^n$ where  $\mathbb{H}^n$
denotes the $(2n+1)$-dimensional Heisenberg group. If $n=1$, then
the horizontal tangent bundle $\HS$ of $S$ cannot be a $2$-step
CC-space because $\HS$ is $1$-dimensional. Nevertheless, if $n>1$,
this is no longer true. Indeed, each non-characteristic domain
$\UU\subseteq S$ turns out to be generic and horizontal.
\end{es}

If we consider an immersed submanifold $S^{n-i}\subset\GG$ of
codimension $i>1$, what above can be generalized in the following
way:

\begin{Defi}\label{iuoi} We say that an  $i$-codimensional submanifold  $S^{n-i}$
of $\GG$ is {\bf geometrically $\HH$-regular} at $x\in S$ if there
exist linearly independent vectors $\nn^1,...,\nn^i\in\HH_x$
transversal to $S$ at $x$. Without loss of generality, these
vectors may also be assumed to be orthonormal at $p$. The
horizontal tangent space at $x$ is defined by
$$\HH_x S:=\HH_x\cap\TT_x S.$$ If this condition is independent of the point $x\in S$, we say that $S$ is
geometrically $\HH$-regular. In such a case we may define the
associated vector bundles $\HS(\subset \TS)$ and $\nn S$, called,
respectively, {\bf horizontal tangent bundle} and {\bf
 horizontal normal bundle}. One has
$$\HH_x:=\HH_x S\oplus\R\nn^1\oplus...\oplus\R\nn^i.$$
\end{Defi}

\begin{Defi}[Characteristic set of $S^{n-i}$]\label{carsetgen}
The {\bf characteristic set} $C_S$ of any  $\mathbf{C}^1$-smooth
 $i$-codimensional submanifold $S^{n-i}\subset\GG$ is defined by$$
C_S:=\{x\in S : \dim\,\HH_x -\dim (\HH_x \cap \TT_x S)\leq
i-1\}.$$
\end{Defi}

\begin{oss}[Hausdorff measure of $C_{S^{n-i}}$]\label{carsetgen}

The above definition of $C_S$ has been used in \cite{Mag2}, where
it was shown that the $(Q-i)$-dimensional Hausdorff measure
-w.r.t. $\dc$- of any $\mathbf{C}^1$-smooth submanifold
$S^{n-i}\subset\GG$ vanishes, i.e. $\mathcal{H}_{\bf
cc}^{Q-i}(C_S)=0.$\end{oss}\begin{Defi}[$\mis$-measure]\label{dens}
Let $S^{n-i}\subset\GG$ be an $i$-codimensional, geometrically
$\HH$-regular submanifold. Let $\nn^1,...,\nn^i\in\nn S$ and
assume they are everywhere orthonormal. We shall set
$$\nn:=\nn^1\wedge...\wedge\nn^i\in\Lambda_i(\TG),$$ and
define the $(n-i)$-dimensional measure $\mis$ along ${S}$ to be
the measure associated with the $(n-i)$-differential form
$\mis\in\Lambda^{n-i}(\TS)$ given by the interior product of the
volume form of $\GG$ with the
$i$-vector $\nn$, i.e.\footnote{For the general definition of the operation $\LL$ see \cite{FE}, Ch.1.} \begin{equation}\mis\res S:=(\nn \LL \sigma^n\rr)|_S.\\
\end{equation}
\end{Defi}
\begin{oss}\label{mourigno}The measure $\mis$ is
$(Q-i)$-homogeneous, w.r.t. Carnot dilations $\{\delta_t\}_{t>0}$,
i.e. $\delta_t^{\ast}\mis=t^{Q-i}\mis$. This easily follows from
the definitions. Moreover, it can be shown that the measure $\mis$
equals, up to a normalization constant, the $Q-i$-dimensional
Hausdorff
 measure associated to a
homogeneous distance $\varrho$ on $\GG$; see \cite{Mag3}.
\end{oss}

\subsection{Geometry of immersed
hypersurfaces }\label{hs}

We now introduce some notions, useful in the study of {\it
non-characteristic hypersurfaces}; see \cite{Monte, Monteb,
MonteSem} \cite{vari}, \cite{G,Pauls}, \cite{HP, HP2},
\cite{DanGarN8, gar}.

Let $S\subset\GG$ be a $\cont^2$-smooth hypersurface. We stress
that, if $\tsc$ is the connection induced  on $S$ from the
Levi-Civita connection $\nabla$ on $\GG$\footnote{Therefore,
$\tsc$ is the Levi-Civita connection on $S$; see \cite{Ch1}.},
then $\tsc$ induces a partial connection $\gs$, associated with
the subbundle $\HS\subset\TT{S}$, which is defined by\footnote{The
map $\P\ss:\TT{S}\longrightarrow\HS$ denotes the orthogonal
projection of $\TT{S}$ onto $\HS$.}
$$\gs_XY:=\P\ss(\tsc_XY)\quad\,(X,Y\in\HS).$$
 Starting from the orthogonal decomposition
 $\HH=\HS\oplus\nn S$
(see Definition \ref{carca}), we could also define
$\nabla^{_{\HS}}$ by making use of the classical definition of
``connection on submanifolds''; see  \cite{Ch1}. Actually, it
turns out that$$\gs_XY=\gc_X Y-\langle\gc_X
Y,\nn\rangle\cc\nn\quad\,(X, Y\in\HS).$$

\begin{Defi}
We call $\HS$-{\bf gradient} of $\psi\in \cin({S})$ the unique
horizontal tangent vector fields of $\HS$, $\qq\psi$, satisfying
$$\langle\qq\psi,X \rangle\ss= d \psi (X) = X
\psi\quad\forall\,\, X\in \HS.$$We denote by $\div\ss$ the
divergence operator on $\HS$, i.e. if $X\in\HS$ and $x\in {S}$,
then
$$\div\ss X (x) := \mathrm{Trace}\big(Y\longrightarrow
\gs_Y X \big)(x)\quad\,(Y\in \HH_xS).$$Finally, $\Delta_{_{\HS}}$
denotes the {\bf $\HS$-Laplacian}, i.e. the differential operator
given by
\begin{eqnarray} \label{Ciarr}
\Delta\ss\psi := \div\ss(\qq\psi)\quad(\psi\in
\cin({S})).\end{eqnarray}
\end{Defi}

\begin{Defi}\label{curvmed}
The sub-Riemannian {\bf horizontal \bf $\bf II^{nd}$ fundamental
form} of ${S}$ is the map ${\overline{B}\cc}: \HS\times\HS
\longrightarrow\nS$ given by
$$\overline{{B}}\cc(X,Y):=\langle\gc_X Y, \nn\rangle\cc
\,\nn\quad(X,\,Y\in\HS).$$ Moreover $\MH\in\nS$ is the {\bf
horizontal mean curvature vector of} $S$, defined as the trace of
${\overline{B}\cc}$, i.e. $\MH=\mathrm{Tr}{\overline{B}\cc}.$ The
{\bf horizontal - scalar - mean curvature} of $S$, denoted by
$\MS$, is defined by $\MS:=\langle\MH,\nn\rangle\cc$. We
set$${B\cc}(X,Y):=\langle\gc_X Y, \nn\rangle\cc
\quad(X,\,Y\in\HS).$$The {\bf torsion} $\Tor\ss$ of the partial
$\HS$-connection $\gs$ is defined
by\begin{eqnarray*}\Tor\ss(X,Y):=\gs_XY-\gs_YX-\PH[X,Y]\quad (X,
Y\in\HS).\end{eqnarray*}

\end{Defi}

The trace ${\rm Tr}$ is computed w.r.t. the 1st sub-Riemannian
fundamental form $g\ss=\langle\cdot,\cdot\rangle\ss$, which is the
restriction to $S$ of the metric $g\cc$, i.e.
$g\ss:=g\cc|_{\HS}=g|_{\HS}$. ${B\cc}(X,Y)$ is a {\it
$\cin(S)$-bilinear form in $X$ and $Y$.} We stress that the
torsion has been defined because, in general, $B\cc$ {\it is not
symmetric}; see \cite{Monteb}.

We now remind the notion of {\it adapted frame}; see
\cite{Monteb}. Roughly speaking, we shall ``adapt'' in the usual
Riemannian way (see, for example, \cite{Spiv}) an orthonormal
frame to the horizontal tangent space of a hypersurface. If
$U\subset\GG$ is open, we set $\UU:=U\cap S$ and, for a while, we
shall assume that $\UU$ be non-characteristic.

\begin{Defi}\label{movadafr}An {\bf adapted frame to $\UU$ on $U$} is an orthonormal frame
$\underline{\tau}:=\{\tau_1,...,\tau_n\}$ on $U$   such that:
\begin{eqnarray*}\mbox{\bf{(i)}}\,\,\tau_1|_\UU:=\nn;\,\,\qquad
\mbox{\bf{(ii)}}\,\,\HH_p\UU=\mathrm{span}\{(\tau_2)_p,...,(\tau_{{\DH}})_p\}\quad(p\in\UU);\,\,\qquad
\mbox{\bf{(iii)}}\,\,\tau_\alpha:=
X_\alpha.\end{eqnarray*}\end{Defi}

\begin{no}Remind that
 $I\cc=\{1,2,...,\DH\}$ and that $I\vv=\{\DH+1,...,n=\dim\,\GG\}$.
 In the sequel,
  we shall also set
 $I\ss:=\{2,3,...,\DH\}$ and we shall assume that
 $$\underline{\tau}=\{\underbrace{\tau_1}_{=\nn},
 \underbrace{\tau_2,...,\tau_{\DH}}_{\mbox{\tiny{o.n. basis of}}
 \,\HS},\underbrace{\tau_{\DH+1},...,\tau_n}_{\mbox{\tiny o.n. basis of}\,\VV}\}.$$For any
  adapted frame $\underline{\tau}$, we
shall denote by $\underline{\phi}:=(\phi_1,...,\phi_n)$, its dual
co-frame. This means that $\phi_I(\tau_J)=\delta_I^J\,\,\mbox{\rm
(Kronecker)}\,\,(I, J=1,...,n).$
\end{no}

Note that every adapted orthonormal frame to a hypersurface is a
graded frame.

 Throughout this paper, we shall frequently use the
following notation:
\begin{Defi}\label{notazia} Let $S\subset\GG$ be a $\cont^2$-smooth hypersurace
oriented by its - Riemannian - unit normal $\nu$. Then we shall set
\begin{itemize}\item[{\bf(i)}]$\varpi_\alpha:=\frac{\nu_\alpha}{|\PH\nu|}\qquad (\alpha\in
I\vv)$;\item[{\bf(ii)}]$\varpi:=\sum_{\alpha\in
I\vv}\varpi_\alpha\tau_\alpha$;\item[{\bf(iii)}]
$C\cc:=\sum_{\alpha\in {I\cd}} \varpi_\alpha\,C^\alpha\cc;$
\item[{\bf(iv)}] $C:=\sum_{\alpha\in {I\vv}}
\varpi_\alpha\,C^\alpha.$\end{itemize}Moreover, for any $\alpha\in
I\cd$, we shall set $C^\alpha\ss:=C^\alpha\cc|_{\HS}$ to stress
that the linear operator $C^\alpha\ss$ only acts on horizontal
tangent vectors, i.e. $(C^\alpha\ss)_{ij}:=\langle
C^\alpha\cc\tau_j,\tau_i\rangle$ for $i, j\in I\ss$. Accordingly,
$C\ss:=\sum_{\alpha\in I\cd}\varpi_\alpha C^\alpha\ss$.
\end{Defi}

\begin{oss}\label{mariodecandia}The horizontal second fundamental form ${{B}\cc}$ is  {\bf not symmetric}, in general.
Thus $B\cc$ can be regarded as the sum of two matrices, one
symmetric and the other skew-symmetric, i.e. $B\cc= S\cc + A\cc,$
where the skew-symmetric matrix
 $A\cc$ is explicitly given by $A\cc=\frac{1}{2}\,C\ss$; see
 \cite{Monteb}. In this regard, we would like to stress that, for every $X,\,Y\in\HS$,
 the following chain of identities holds true:
 \begin{eqnarray*}
 \Tor\ss(X,Y)&=&\overline{B}\cc(Y,X)-\overline{B}\cc(X,Y)
\\&=&\langle\PH[Y,X],\nn\rangle\nn
\\&=&\langle[X,Y],\varpi\rangle\nn\\&=&\langle\Omega\cc(X,Y),\varpi\rangle
\\&=&-\langle C\ss X, Y\rangle.\end{eqnarray*}
\end{oss}

\subsection{Integration by parts on hypersurfaces}\label{Intbypart}
In this section  we shall discuss horizontal integration by parts
formulas for smooth hypersurfaces, immersed in a $k$-step Carnot
group $\GG$, endowed with the $\HH$-perimeter measure $\per$; see, for instance,
\cite{Monte, Monteb, MonteSem}, but also \cite{DanGarN8, gar}.

Let $S\subset\GG$ be a $\cont^2$-smooth hypersurface and let
$\UU\subset S$ be a relatively compact open set with
$\cont^1$-smooth boundary (or, smooth enough for the application
of Stokes' Theorem). If $X\in\mathfrak{X}(S)$, by definition of
$\per$, using the {\it Riemannian Divergence Formula} (see
\cite{Spiv}), one gets
 \begin{eqnarray*}d(X\LL\per)|_\UU&=& d(|\PH \nu|\,X\,\LL \sigma ^{n-1})
 = \div\ts(|\PH \nu|\,X)\,\sigma\rr^{n-1}\\ &=& \Big(\div\ts X + \Big\langle X,
\frac{\grad\ts{|\PH\nu|}}{|\PH\nu|}\Big\rangle\Big)\,\per\res\UU,\end{eqnarray*}where
$\grad\ts$ and $\div\ts$ are, respectively, the tangential
{gradient} and the tangential {divergence} operators on
$\UU\subset S$. However this formula is not so ``explicit'' from a
sub-Riemannian point of view. The notion of adapted frame was
introduced to bypass this inconvenience by allowing us to make
explicit computations.

Let $\underline{\tau}$ be an adapted frame to $\UU\subset S$ on
the open set $U$ and let $\underline{\phi}:=\{\phi_1,...,\phi_n\}$
be its {\it dual co-frame}, obtained by means of the metric $g$.
It is immediate to see that the $\HH$-perimeter $\per$ on $\UU$ is
given by
\begin{eqnarray*}\per\res\UU&=&(\nn\LL\sigma^n)|_\UU=\ast\phi_1|_\UU
=(\phi_2\wedge...\wedge\phi_n)|_\UU
\\&=&(-1)^{\alpha+1}({\varpi_\alpha})^{-1}
\phi_1\wedge...\wedge\widehat{\phi_\alpha}
\wedge...\wedge\phi_n)|_{\UU}\quad(\alpha\in
I\vv),\end{eqnarray*}where the last identity makes sense only if
$\nu_\alpha\neq 0$\footnote{With respect to the adapted frame
$\underline{\tau}$, the Riemannian unit normal $\nu$ is given by
$\nu=\nu_1\tau_1 +\sum_{\alpha\in I\vv}\nu_\alpha\tau_\alpha$,
where $\tau_1:=\nn$, $\nu_1:=|\PH\nu|$ and $\tau_\alpha=X_\alpha$.
Remind also that $\varpi_\alpha:=\frac{\nu_\alpha}{\nu_1}$.}. By
direct computations based on the 1st Cartan's structure equation for
 $\underline{\phi}$, one obtains divergence-type formulas and
 some easy but useful corollaries; see \cite{Monte, Monteb, MonteSem}.

\begin{oss}[Measure on the boundary $\partial\UU$]\label{measonfr} Before stating these results we would like to
make a preliminary comment on the topological boundary
$\partial\UU$ of $\UU$. Firstly,  let us assume that $\partial\UU$ is
a $(n-2)$-dimensional manifold, oriented by its unit normal vector
$\eta$.
 Let us denote by $\sigma^{n-2}\rr$ the usual Riemannian measure on $\partial\UU$, which can
 be written out as $$\sigma^{n-2}\rr\res{\partial\UU}=(\eta\LL\sigma^{n-1}\rr)|_{\partial\UU}.$$This means that
 if $X\in\XX(\overline{\UU})$, then $$(X\LL\per)|_{\partial\UU}=\langle X, \eta\rangle
  |\PH\nu|\, \sigma^{n-2}\rr\res{\partial\UU}.$$ Let us further assume that  $\partial\UU$ be geometrically
 $\HH$-regular. This is equivalent
 to require that the projection onto $\HS$ of the unit  normal $\eta$ along
$\partial\UU$ be non-singular, i.e. $|\P\ss\eta|\neq 0$ for every
$p\in\partial\UU$.  We shall denote by $C_{\partial\UU}$ the characteristic
set of ${\partial\UU}$, which turns out to be given by
$$C_{\partial\UU}=\{p\in{\partial\UU}: |\P\ss\eta|=0\}.$$From
Definition \ref{dens} it follows that
 $${\nis}\res{\partial\UU}=
 \Big(\frac{\P\ss\eta}{|\P\ss\eta|}\LL\per\Big)\Big|_{\partial\UU},$$
or, equivalently, that ${\nis}\res{\partial\UU}=
|\PH\nu|\cdot|\P\ss\eta|\sigma^{n-2}\rr\res{\partial\UU}.$ Setting
$\eta\ss:=\frac{\P\ss\eta}{|\P\ss\eta|}$, we will call $\eta\ss$
the {\bf unit  horizontal normal } along $\partial\UU$. We
get then
$$(X\LL\per)|_{\partial\UU}=\langle X, \eta\ss\rangle
  {\nis}\res{\partial\UU}\qquad \forall\,\,X\in\cont^1(S,\HS).$$
 \end{oss}

The proof of the next results can be found in \cite{Monteb}.
\begin{teo}[Horizontal Divergence Theorem ]\label{GD} Let $\GG$
be a $k$-step Carnot group. Let $S\subset\GG$ be an immersed
hypersurface and $\UU\subset S\setminus C_S$ be a
non-characteristic relatively compact open set. Assume that
$\partial\UU$ is  a smooth, $(n-2)$-dimensional manifold oriented
by its unit normal vector $\eta$. Then, for every $X\in
\cont^1(S,\HS)$ one has
\[ \int_{\UU}\big\{ \div\ss X + \langle C\cc\nn,
X\rangle\big\}\per = \int_{\partial\UU}\langle
X,\eta\ss\rangle\,{\nis}.\]
\end{teo}

\begin{corollario}[Horizontal integration by parts]\label{IPPH}Under the hypotheses Theorem \ref{GD},
for every $X\in \cont^1(S,\HH)$ one has \[ \int_{\UU}\big\{ \div\ss
X + \langle C\cc\nn, X\rangle\big\}\per =-\int_{\UU}\langle X,
\MH\rangle\,\per + \int_{\partial\UU}\langle
X,\eta\ss\rangle\,{\nis}.\]
\end{corollario}

\begin{proof}It follows by Theorem \ref{GD} and Definition
\ref{curvmed}.\end{proof}

A simple consequence of Corollary \ref{IPPH} is the following:

\begin{corollario}[Integral Minkowsky-type formula]\label{IPPH0}Under the hypotheses of Theorem
\ref{GD}, let  $\x\cc:=\sum_{i\in I\cc}x_i X_i$ be the
{``horizontal position vector''} and let $\g$ denote its
component along the $\HH$-normal $\nn$, i.e. $\g:=\langle
\x\cc,\nn\rangle$.\footnote{This function could be called
``horizontal support function'' of $\x\cc$.} Then
\[ \int_{\UU}\big\{ (\DH-1)+ \g\MS +  \langle C\cc\nn,
\x\ss\rangle\big\}\per = \int_{\partial\UU}\langle
\x\cc,\eta\ss\rangle\,\sigma^{n-2}\ss.\]
\end{corollario}
\begin{proof} Apply Corollary \ref{IPPH0} to the position vector field $\x\cc$.\end{proof}

\begin{oss}\label{cpoints}Let $S\subset\GG$ be a compact $\cont^2$-smooth hypersurface with boundary and let $\UU_\epsilon\,(\epsilon>0)$ be a family of open subsets
of $S$ with  piecewise $\cont^2$-smooth boundary such that
\footnote{As shown in \cite{HP2} (see Lemma 5.11) such a family
exists and can be constructed by covering the characteristic set
$C_S$ (that is a closed subset of $S$ having - at most -
$(n-2)$-dimensional Riemannian Hausdorff measure) with Riemannian
balls of ``small'' radii.
 We would like to stress that
(iii) can easily be proved by using Stokes' Theorem.
Indeed one
has\begin{eqnarray}\label{sti}d\big(|\PH\nu|\sigma^{n-2}\rr\big)\big|_{\UU_\epsilon}&=&\big(d|\PH\nu|\wedge\sigma^{n-2}\rr
+|\PH\nu|\wedge
d\sigma^{n-2}\rr\big)\big|_{\UU_\epsilon}\\\nonumber&=&\Big(\frac{\partial|\PH\nu|}{\partial\widetilde{\eta}}+
|\PH\nu|\div\ts\widetilde{\eta}\Big)\sigma^{n-1}\rr\res\UU_\epsilon,
\end{eqnarray}where $\widetilde{\eta}$ is any smooth enough vector
field on $\UU_\epsilon$ extending the unit normal vector $\eta$
along $\partial\UU_\epsilon$ which is, by construction, piecewise
$\cont^2$-smooth. Since $|\PH\nu|$ is a bounded Lipschitz function
and $\widetilde{\eta}$ is smooth, the function
$\big(\frac{\partial|\PH\nu|}{\partial\widetilde{\eta}}+
|\PH\nu|\div\ts\widetilde{\eta}\big)$ is bounded and (iii) follows
by integrating both sides of \eqref{sti} along $\UU_\epsilon$,
using Stokes' Theorem and (ii).}:\begin{itemize}\item[{\rm(i)}]
$C_S\subset\UU_\epsilon$ for every $\epsilon>0$; \item[{\rm(ii)}]
$\sigma^{n-1}\rr(\UU_\epsilon)\longrightarrow 0$ for
$\epsilon\rightarrow
0^+$;\item[{\rm(iii)}]$\int_{\UU_\epsilon}|\PH\nu|\sigma^{n-2}\rr\longrightarrow
0$ for $\epsilon\rightarrow 0^+$.\end{itemize}Incidentally, we
stress that ${\rm(iii)}$ immediately implies that
$\nis(\partial\UU_\epsilon)\longrightarrow 0$ for
$\epsilon\rightarrow 0^+$. By using this family, one can extend
the previous formulae to the case of hypersurfaces having a
non-empty characteristic set. Indeed, let us apply  Theorem
\ref{GD} to the non-characteristic set $S\setminus \UU_\epsilon$,
for some (small enough) $\epsilon>0$. It turns out that
\begin{eqnarray*} \int_{S\setminus\UU_\epsilon}\big\{\div\ss X +
\langle C\cc\nn, X\rangle\big\}\per = \int_{\partial S}\langle
X,\eta\ss\rangle\,{\nis}-\int_{\partial\UU_\epsilon}\langle
X,\eta\ss\rangle\,{\nis}.\end{eqnarray*}Since  $C_S$ is a
0-measure set w.r.t. the $\per$-measure, by letting
$\epsilon\rightarrow 0^+$, we easily get that $$\lim_{\epsilon\rightarrow 0^+}\int_{S\setminus\UU_\epsilon}\big\{\div\ss X +
\langle C\cc\nn, X\rangle\big\}\per=\int_S\big( \div\ss X + \langle C\cc\nn,
X\rangle\big)\per.$$ Furthermore, by using ${\rm(iii)}$, one gets
that the third integral vanishes. {\bf Therefore,  all the results
previously stated can be applied also to the case of
hypersurfaces having non-empty characteristic set}.
\end{oss}

In the sequel we shall first introduce a natural 2nd-order
tangential operator which acts on functions defined along any
smooth non-characteristic hypersurface of the group.
\begin{Defi}\label{DEFIEIGEN} Let $S\subset\GG$ be a smooth non-characteristic hypersurface. Then we denote by
$\lh$ the differential operator defined by
\begin{eqnarray*}\lh\varphi:=\Delta\ss\varphi +\langle C\cc\nn, \qq\varphi\rangle
\qquad\mbox{for
every}\,\,\varphi\in\cin(S).\end{eqnarray*}Moreover, we shall
denote by $\lg$ the differential operator
$\lg:\XX(\HS)\longrightarrow\R$ given by
\begin{eqnarray*}\lg(X):=\div\ss X + \langle C\cc\nn,
X\rangle\qquad\mbox{for every}\,\,X\in\XX(\HS).\end{eqnarray*}Note
that $\lh \varphi=\lg(\qq \varphi).$
\end{Defi}We stress that the previous results and, in particular, Remark \ref{cpoints}, allow us to immediately deduce
the next:

\begin{Prop}[Some properties of $\lh$ and $\lg$]\label{Properties}
Let $S\subset\mathbb\GG$ be a $\cont^2$-smooth hyper-surface with
- or without - boundary $\partial S$. With the previous notation we
have:\begin{itemize} \item[{\bf(i)}] $\int_{S}\lh \varphi\,\per=0$
for every $\varphi\in\mathbf{C}_0^\infty(S)$; \item[
{\bf(ii)}]$\int_{S}\psi\,\lh
\varphi\,\per=\int_{S}\varphi\,\lh\psi \,\per$  for every
$\varphi,\,\psi\in\mathbf{C}_0^\infty(S)$;\item[
{\bf(iii)}]$\lg(\varphi X)=\varphi\lg X + \langle\qq \varphi,
X\rangle$   for every $X\in\XX(\HS),\, \varphi\in\cin(S)$;
\item[{\bf(iv)}]$\int_{S}\lh \varphi\,\per=\int_{\partial
S}{\partial\varphi}/{\partial\eta\ss}\,\nis$ for every
$\varphi\in\mathbf{C}^\infty(S)$; \item[
{\bf(v)}]$\int_{S}\{\psi\,\lh \varphi-\varphi\,\lh\psi\}
\,\per=\int_{\partial
S}\{\psi{\partial\varphi}/{\partial\eta\ss}-\varphi{\partial\psi}/{\partial\eta\ss}\}\,\nis$
for every $\varphi,\,\psi\in\cin(S)$;\item[
{\bf(vi)}]$\int_{S}\psi\lh
\varphi\,\per+\int_{S}\langle\qq\varphi,\qq\psi\rangle\,\per=
\int_{\partial S}\psi{\partial\varphi}/{\partial\eta\ss}\,\nis$
for every $\varphi, \psi\in\cin(S)$.
\end{itemize}
\end{Prop}\begin{proof}All these properties easily follow from the
results previously stated in this section and, in particular,
Theorem \ref{GD} and Remark \ref{cpoints}. Note also that (ii)
holds true even if only one of the two functions is compactly supported on
$S$.\end{proof}

\begin{Defi}[The eigenvalue problems for the operator  $\lh$]\label{epr} Let $S\subset\GG$ be a compact hypersurface
without boundary. Then we look for $\cont^2$-smooth solutions of
the problem:
\begin{displaymath}
{\rm(P_1)}\,\,\,\left\{%
\begin{array}{ll}
    \quad\,\quad\lh\psi=\lambda \, \psi\, \qquad(\lambda\in\R,\,\,\psi\in\mathbf{C}^2(S)); \\
    \,\,\,  \int_S\psi\,\per  =0. \\
\end{array}%
\right.\end{displaymath}If $\partial S\neq {\emptyset}$ we look
for $\cont^2$-smooth solutions of the following problems:
\begin{displaymath}
{\rm(P_2)}\,\,\,\left\{%
\begin{array}{ll}
    \quad\,\,\quad\lh\psi=\lambda \, \psi\, \qquad(\lambda\in\R,\,\,\psi\in\mathbf{C}^2(S)); \\
    \quad\quad\,\,\,\,\,  \psi|_{\partial S} =0; \\
\end{array}%
\right.\end{displaymath}

\begin{displaymath}
{\rm(P_3)}\,\,\,\left\{%
\begin{array}{ll}
    \quad\,\,\,\,\,\,\lh\psi=\lambda \,\psi\, \qquad(\lambda\in\R,\,\,\psi\in\mathbf{C}^2(S)); \\
    \quad \,\frac{\partial\psi}{\partial\eta\ss}\big|_{\partial S} =0,
\end{array}%
\right.\end{displaymath} where we explicitly note that
$\frac{\partial\psi}{\partial\eta\ss}=\langle\qq\psi,\eta\ss\rangle$.

\end{Defi}

The problems (P$_1$), (P$_2$) and (P$_3$) generalize to our
context the classical {\it closed, Dirichlet and Neumann
eigenvalue problems} for the Laplace-Beltrami operator on
Riemannian manifolds.

\subsection{1st  variation of $\per$}\label{prvar}In this section, we recall and explain the 1st variation formula
of $\per$. The presentation given here closely follows  that given
in \cite{Monteb}, but provides a further analysis of the case of
hypersurfaces having non-empty characteristic set; see also
\cite{DanGarN8, gar}, \cite{Monte}, \cite{vari}, \cite{RR},
\cite{HP, HP2}. As references, for the Riemannian case, we mention
Spivak's book \cite{Spiv} and also the paper by Hermann \cite{2}.

Let $\GG$ be a $k$-step  Carnot group and let $S\subset\GG$ be a
$\cont^2$-smooth hypersurface oriented by its unit normal vector
$\nu$. Moreover, let $\UU\subset S\setminus C_S$ be a {\it
non-characteristic} relatively compact open set and let us assume
that $\partial\UU$ is a $(n-2)$-dimensional $\cont^1$-smooth
submanifold oriented by its outward unit normal vector $\eta$.

\begin{Defi} \label{leibniz}Let $\imath:\UU\rightarrow\GG$ denote the inclusion of $\UU$ in $\GG$
and let
 $\vartheta: (-\epsilon,\epsilon)\times \UU
\rightarrow \GG$ be a smooth map. Then $\vartheta$ is a  {\bf
smooth variation} of $\imath$ if:
\begin{itemize}
\item[{\bf(i)}] every
$\vartheta_t:=\vartheta(t,\cdot):\UU\rightarrow\GG$ is an
immersion;\item[{\bf(ii)}] $\vartheta_0=\imath$.
\end{itemize}
Moreover, we say that the variation $\vartheta$ {\bf keeps the
boundary $\partial\UU$ fixed} if:\begin{itemize}
\item[{\bf(iii)}]$\vartheta_t|_{\partial \UU}=\imath|_{\partial
\mathcal{U}}$ for every $t\in (-\epsilon,\epsilon)$.
\end{itemize}
The {\bf variation vector} of $\vartheta$, is defined by
$W:=\frac{\partial \vartheta}{\partial
t}\big|_{t=0}=\vartheta_{\ast}\frac{\partial}{\partial
t}\big|_{t=0}.$
\end{Defi}
For any $t\in(-\epsilon,\epsilon)$ let ${\nu}^t$ be the unit
normal vector along $\UU_{t}:=\vartheta_t(\UU)$ and let
$(\sigma^{n-1}\rr)_t$ be the Riemannian measure on $\UU_t$. Let us
define the differential $(n-1)$-form
 $\pert$ along $\UU_{t}$, by
$$\pert|_{\UU_{\, t}}= (\nt \LL \Omn)|_{\UU_{\, t}}
\in\Lambda^{n-1}(\TT\UU_{\,t})\qquad t\in(-\epsilon,
 \epsilon)$$ where
$\nt:=\frac{\PH {\nu}^t }{|\PH {\nu}^t|}.$ Moreover, let us set
$\Gamma(t):=\vartheta_t^\ast\pert\in\Lambda^{n-1} (\TT\UU),\,\,
t\in(-\epsilon,\epsilon).$ The 1st variation $I_\UU(W,\per)$ of
$\per$ is given
by\begin{equation}\label{nome}I_\UU(W,\per)=\frac{d}{dt}\Big(\int_{\UU}\Gamma(t)\Big)\Big|_{t=0}=
\int_{\UU}\dot{\Gamma}(0).\end{equation}
\begin{teo}[1st variation of $\per$] \label{1vg} Under the previous
hypotheses we have\begin{equation}\label{fv}I_\UU(W,\per) = -
\int_{\UU}\MS \Big\langle W,\frac{\nu}{|\P\cc\nu|}
\Big\rangle\,\per  +  \int_{\partial\UU}
  \langle W, \eta \rangle\, |\PH{\nu}|\,\sigma^{n-2}\rr.\end{equation}
\end{teo}
For a proof see, for instance, \cite{Monteb}. It is
 clear that, if we allow the variation vector to be
horizontal, then (\ref{fv}) becomes more ``intrinsic''.
\begin{teo}[Horizontal 1st variation of $\per$]Under the previous hypotheses,
let us assume that the variation vector $W$ of $\vartheta$ be
horizontal. Then
\begin{equation}\label{fv3}I_\UU(W,\per) =
- \int_{\UU}\langle \MH, W \rangle\,\per +
\int_{\partial\UU}\langle
W,\eta\ss\rangle\,{\nis}.\end{equation}\end{teo}

\begin{proof} Use Theorem \ref{1vg} and Remark \ref{measonfr}.
\end{proof}

Therefore, in the case of horizontal variations, remembering
Corollary \ref{IPPH}, we get
$$I_\UU(W,\per)=\int_{\UU}\big\{\div\ss W + \langle
C\cc\nn, W\rangle\big\}\per.$$

We stress that the horizontal 1st variation formula \eqref{fv3} is
the sum of two terms, the first of whose only depends on the
horizontal normal component of $W$, while the second one, only
depends on its horizontal tangential component.

The previous formulae provide the 1st variation of $\per$ on
regular non-characteristic subsets of $S$ containing ${\rm spt}
W$. In the following remark we explain how one can extend the
previous results  to include the case in which the hypersurface
has a possibly non-empty characteristic set $C_S$.  A similar
remark in the case of the Heisenberg group $\mathbb{H}^1$ was done
in a recent work by Ritor\'e and Rosales \cite{RR}; see also
\cite{HP2}.

\begin{oss}[1st variation: case $C_S\neq{\emptyset}$]\label{cpoints1}Let $S\subset\GG$ be a $\mathbf{C}^2$-smooth  hypersurface
and let $W\in\cont^1(S,\TG)$ be the variation vector field.
Moreover, let us assume that $\MS\in L^1(S)$. First, we note that
the function $|\P\cc\nu|$ is Lipschitz continuous  and that
$|\P\cc\nu|$ vanishes along $C_S$. Now let
$\UU_\epsilon\,(\epsilon>0)$ be the family of open subsets of $S$
defined in Remark \ref{cpoints}. For every $\epsilon>0$ one
computes
\begin{equation}\label{seer}I_S(W, \per)=I_{S\setminus\UU_\epsilon}(W, \per) + I_{\UU_\epsilon}(W, \per).\end{equation}
The first term is then given by the previous Theorem \ref{1vg}:
\[I_{S\setminus\UU_\epsilon}(W, \per) = -
\int_{S\setminus\UU_\epsilon}\MS \Big\langle
W,\frac{\nu}{|\P\cc\nu|} \Big\rangle\,\per  +  \int_{\partial S}
  \langle W, \eta \rangle\, |\PH{\nu}|\,\sigma^{n-2}\rr + \int_{\partial\UU_\epsilon}
  \langle W, \eta^- \rangle\, |\PH{\nu}|\,\sigma^{n-2}\rr,\]where
  $\eta^-$ denotes the outward unit normal of $S\setminus\UU_\epsilon$ along
  $\partial\UU_\epsilon$. The second term in \eqref{seer} can be computed as follows:
\[I_{\UU_\epsilon}(W, \per)=\frac{d}{dt}\Big(\int_{\UU_\epsilon}\pert\Big)\Big|_{t=0}
=\int_{\UU_\epsilon}\frac{d}{dt}\pert\big|_{t=0}.\]Note
that\begin{equation}\label{serer}\frac{d}{dt}\pert\big|_{t=0}=
\frac{d}{dt}|\P\cc\nu^t|\big|_{t=0}\sigma^{n-1}\rr + |\P\cc\nu|
\frac{d}{dt}(\sigma^{n-1}\rr)_t\big|_{t=0}.\end{equation} Clearly,
the first addend is bounded, since the function $|\P\cc\nu|$ is
Lipschitz, while the second one, up to the bounded function
$|\P\cc\nu|$, is just the $(n-1)$-form
$\frac{d}{dt}(\sigma^{n-1}\rr)_t|_{t=0}$, which expresses the
``infinitesimal'' 1st variation formula of the - Riemannian -
$(n-1)$-dimensional measure; see \cite{2}, \cite{Spiv}. Actually,
the previous formula can also be rewritten in terms of Lie
derivatives; see \cite{Monte, Monteb, MonteSem}. More precisely,
it turns out that
$$\frac{d}{dt}\pert\big|_{t=0}=\imath_{\UU_\epsilon}^\ast\mathcal{L}_W\pert.$$From
this formula, Cartan's identity and a simple computation, it
follows that

\[I_{\UU_\epsilon}(W, \per)=\underbrace{\int_{\UU_\epsilon}\Big\{\frac{\partial|\P\cc\nu|}{\partial\nu}-|\P\cc\nu|\mathcal{H}\rr\Big\}\Big\langle W,\frac{\nu}{|\P\cc\nu|}
\Big\rangle\per}_{I^{Int.}_{\UU_\epsilon}(W,
\per)}+\underbrace{\int_{\partial\UU_\epsilon}\langle W,{\eta^+}
\rangle|\P\cc\nu|\sigma^{n-2}\rr}_{I^{Bound.}_{\UU_\epsilon}(W,
\per)},\]where
  $\eta^+$ denotes the outward unit normal of $\UU_\epsilon$ along
  $\partial\UU_\epsilon$. Since
$\MS\in L^1(S)$ and since $\mathcal{H}\rr$ is locally
bounded\footnote{Since $S$ is $\cont^2$, the Riemannian mean
curvature $\mathcal{H}\rr$ is continuous along $S$.}, by using
${\rm(ii)}$ of Remark \ref{cpoints}, we easily get
$I^{Int.}_{\UU_\epsilon}(W, \per)\longrightarrow 0$ for
$\epsilon\rightarrow 0^+$. Moreover, since\footnote{We stress that
$\partial\UU_\epsilon$ is the common boundary of $\UU_\epsilon$
and $S\setminus\UU_\epsilon$.} $\eta^+=-\eta^-$ along
$\partial\UU_\epsilon$, it follows that

\[I_{S}(W, \per) = -
\int_{S\setminus\UU_\epsilon}\MS \Big\langle
W,\frac{\nu}{|\P\cc\nu|} \Big\rangle\,\per
+I^{Int.}_{\UU_\epsilon}(W, \per) +\int_{\partial S}
  \langle W, \eta \rangle\, |\PH{\nu}|\,\sigma^{n-2}\rr\]for every
  $\epsilon>0$. Therefore, by letting $\epsilon\rightarrow 0^+$, we
  finally obtain
\begin{equation}\label{1wwvarcp}I_S(W, \per) =
- \int_{S}\MS\Big\langle W,\frac{\nu}{|\P\cc\nu|}
\Big\rangle\,\per  +  \int_{\partial S}
  \langle W,
  \eta\rangle\,{|\PH\nu|\,\sigma^{n-2}\rr},\end{equation}which coincides
  formally with \eqref{fv}; compare with \cite{RR}, \cite{HP2}.
\end{oss}

The previous Remark \ref{cpoints1} enables us to state the
following:
\begin{corollario}[1st variation of $\per$] \label{22vg} Let $S\subset\GG$ be a $\cont^2$-smooth hypersurface
having possibly non-empty characteristic set $C_S$. Then, the
1st variation formula \eqref{fv} holds true.\\
\end{corollario}

\begin{oss}
Let us consider the case of a $\cont^2$-smooth surface
$S\subset\mathbb{H}^1$, where $\mathbb{H}^1$ denotes the 1st
Heisenberg group on $\R^3$. Moreover, let us choose, as a vector
variation $W$ for $S$, the following  ``singular'' vector field
$W:=\frac{w\,\nn}{|\PH\nu|},\,\, w\in\cont_0^2(S).$ Actually, $W$
degenerates at $C_S$, since $|\PH\nu|=0$ along $C_S$. So let us
introduce the following partition of  $S$, $S=(S\setminus
\UU_\epsilon)\cup \UU_\epsilon,$ where $\UU_\epsilon$ denotes an
element of the family of open subsets of $S$ defined in Remark
\ref{cpoints}. By using formula \eqref{fv} for the
non-characteristic set $S\setminus\UU_\epsilon$, the very
definition of $W$ and the fact that $W|_{\partial S}=0$, we get
that
\begin{eqnarray*}I_{S\setminus\UU\epsilon}(W, \per) &=& -
\int_{S\setminus\UU_\epsilon}\MS w\,\sigma^{2}\rr + \int_{\partial
\UU_\epsilon}
  \langle W,
  \eta^-\rangle\,{|\PH\nu|\,\sigma^{1}\rr}\\&=&-
\int_{S\setminus\UU_\epsilon}\MS w\,\sigma^{2}\rr +\int_{\partial
\UU_\epsilon}
 w\,\langle \nn,
  \eta^-\rangle\,\sigma^{1}\rr\\&=&
\int_{S\setminus\UU_\epsilon}\{\div\ts(w \nn^{\ts})-\MS w\}
\sigma^{2}\rr,\end{eqnarray*}where
  $\eta^-$ is the outward unit normal of $S\setminus\UU_\epsilon$ along
  $\partial\UU_\epsilon$ and $\nn^{\ts}$ denotes
   the tangential component of $\nn$ along $\TT S$. Note that the last identity is just the
  Riemannian Divergence Theorem. If $\MS\in L^1(S,
  \sigma^{n-1}\rr)$, by Dominate Convergence Theorem one finally obtains
\begin{eqnarray*}I_{S}(W, \per)&=&\lim_{\epsilon\rightarrow
0^+}I_{S\setminus\UU\epsilon}(W, \per) \\&=&
\lim_{\epsilon\rightarrow
0^+}\int_{S\setminus\UU_\epsilon}\{\div\ts(w \nn^{\ts})-\MS w\}
\sigma^{2}\rr\\&=&-\int_S\MS w\,\sigma^{2}\rr +
\lim_{\epsilon\rightarrow
0^+}\int_{S\setminus\UU_\epsilon}\div\ts(w
\nn^{\ts})\sigma^{2}\rr;\end{eqnarray*}compare with Lemma 4.3 in
\cite{RR}.\end{oss}

\section{Blow-up of the horizontal perimeter $\per$}\label{blow-up} Let
$S\subset\GG$ be a smooth\footnote{See later on, for a more
``precise'' requirement. Indeed,  as will be clear, the regularity
assumptions will be fundamental for the rest of this section.}
hypersurface. In this section we shall discuss the behavior of the
horizontal perimeter $\per$ near any point $x\in{\rm Int}(S)$.
More precisely, we shall analyze the following
limit:\begin{equation}\label{limit}\kappa_\varrho(\nn(x)):=\lim_{R\rightarrow
0^+}\frac{\per \{S\cap
B_{\varrho}(x,R)\}}{R^{Q-1}},\end{equation}where
$B_{\varrho}(x,R)$ is the - homogeneous - $\varrho$-ball of center
$x$ and radius $R$. Here, the point $x\in{\rm Int} S$ is {\it not
necessarily non-characteristic}. This is the main difference with
some previous results in literature; compare, for instance, with
the results obtained by Magnani in \cite{Mag, Mag2}; see also
\cite{balogh}, \cite{FSSC3, FSSC5}, \cite{CM2}.

 However we shall briefly
discuss, to the sake of completeness, the non-characteristic
case, to better understand the differences. So let us state the
following:

\begin{teo}\label{BUP}Let $\GG$ be  a $k$-step Carnot group.

\begin{itemize}

\item [{\bf Case (a)}]\,\,Let $S$ be a $\cont^1$-smooth
hypersurface and  $x\in S\setminus C_S$;
 then \begin{eqnarray}\label{BUP1}\per(S\cap B_\varrho(x,R))\sim
\kappa_\varrho(\nn(x)) R^{Q-1}\qquad\mbox{for}\quad R\rightarrow
0^+,\end{eqnarray}where the constant $\kappa_\varrho(\nn(x))$ is
that of Theorem \ref{magnanoi} and is given by
$$\kappa_\varrho(\nn(x))=\per(\mathcal{I}(\nn(x))\cap B_\varrho(x,1)),$$where $\mathcal{I}(\nn(x))$
 denotes the vertical hyperplane
orthogonal to $\nn(x)$.\footnote{Notice that $\mathcal{I}(\nn(x))$
corresponds to an ideal of the Lie algebra $\gg$.}.

\item [{\bf Case (b)}]\,\, Let $x\in C_S$ and let us assume that,
locally around $x$, there exists $\alpha\in I\vv$, ${\rm
ord}(\alpha)=i$, such that $S$ can be represented as the
exponential image of a $\cont^i$-smooth $X_\alpha$-graph. For sake
of simplicity and without loss of generality, let $x=0\in\GG.$ In
such a case one has$$S\cap
B_\varrho(x,r)\subset\exp\big\{\big(\zeta_1,...,\zeta_{\alpha-1},
\psi(\zeta),\zeta_{\alpha+1},...,\zeta_n \big)\, \big|\,
\zeta:=(\zeta_1,...,\zeta_{\alpha-1},
0,\zeta_{\alpha+1},...,\zeta_n )\in \ee_\alpha^\perp\big\},$$
where $\psi:\ee_\alpha^{\perp}\cong\R^{n-1}\rightarrow\R$ is a
$\cont^i$ function. If $\psi$ satisfies
\begin{equation}\label{0dercond}\frac{\partial^{\scriptsize(l)}
\psi}{\partial\zeta_{j_1}...\partial\zeta_{j_l}}(0)=0\qquad\mbox{whenever}\quad{\rm
ord}(j_1)+...+{\rm ord}(j_l)< i
\end{equation} for
every $l\in\{1,...,i\}$, then it follows that
\begin{eqnarray}\label{BUPcarcase}\per(S\cap B_\varrho(x,R))\sim
\kappa_\varrho(C_S(x)) R^{Q-1}\qquad\mbox{for}\quad R\rightarrow
0^+\end{eqnarray}where $\kappa_\varrho(C_S(x))$ can be computed by
integrating $\per$ along a polynomial hypersurface of anisotropic
order $i={\rm ord}(\alpha)$ depending on the Taylor's expansion up
to order $j \leq i$ of $\psi$ at $0\in\R^{n-1}$. More precisely,
it turns out that$$\kappa_\varrho(C_S(x))=\per(\Psi_\infty\cap
B_\varrho(x,1)),$$where the  limit-set  $\Psi_\infty$ is given by
\begin{equation*}\Psi_\infty=
\big\{\big(\zeta_1,...,\zeta_{\alpha-1},\widetilde{\psi}
(\zeta),\zeta_{\alpha+1},...,\zeta_n\big)\big|
\zeta\in\ee_\alpha^\perp\big\}\end{equation*}and
$\widetilde{\psi}$ is the homogeneous polynomial function of
degree $l$ and anisotropic order $i={\rm ord}(\alpha)$ defined by
\begin{eqnarray}\nonumber\widetilde{\psi}(\zeta)=\sum_{\stackrel{{j_1}}{\scriptsize{\rm
ord}(j_1)=i}}\frac{\partial\psi}{\partial\zeta_{j_1}}(0)\zeta_{j_1}+
\sum_{\stackrel{{j_1, j_2}}{\scriptsize{{\rm ord}(j_1)+{\rm
ord}(j_2)=i}}}\frac{\partial^{\scriptsize(2)}\psi}{\partial\zeta_{j_1}\partial\zeta_{j_2}}(0)\zeta_{j_1}\zeta_{j_2}\\\label{jhhjas}+
\sum_{\stackrel{{j_1,...,j_l}}{\scriptsize{{\rm ord}(j_1)+...+{\rm
ord}(j_l)=i}}}\frac{\partial^{\scriptsize(l)}\psi}{\partial\zeta_{j_1}...
\partial\zeta_{j_l}}(0)\zeta_{j_1}\cdot...\cdot\zeta_{j_l}.\end{eqnarray}
Finally, if \eqref{0dercond} does not hold, then
\begin{eqnarray}\label{Bse}\per(S\cap B_\varrho(x,R))\longrightarrow
0\qquad\mbox{for}\quad R\rightarrow
0^+.\end{eqnarray}\end{itemize}
\end{teo}

\begin{proof}

Let us preliminarily note that the limit above can be calculated
at the identity $0\in\GG$, by using a suitable left translation of $S$.
More precisely, by left-invariance of $\per$, one computes
$${\per \{S\cap B_{\varrho}(x,R)\}}={\per \{x^{-1}\bullet(S\cap
B_{\varrho}(x,R))\}}={\per \{(x^{-1}\bullet S)\cap
B_{\varrho}(0,R)\}}\quad(x\in{\rm Int}(S)),$$where the second
equality easily follows from the additivity of the group
multiplication
$\bullet$.\\
 Now let us set $S_R(x):=S\cap B_{\varrho}(x,R)$ and
$$S_R:=x^{-1}\bullet S_R(x)=(x^{-1}\bullet S)\cap
B_{\varrho}(0,R).$$Furthermore remind that, by the homogeneity of
the metric $\varrho$ and by invariance of $\per$ under (positive)
Carnot dilations\footnote{More precisely, one has
$\delta_t^\ast\per=t^{Q-1}\per$,  $\delta_t\, (t\in \R_+)$; see
Section \ref{prelcar}.} it follows that
$$\per\{S_R(x)\}=\per\{\delta_R(\delta_{1/R}S)\cap B_\varrho(x,R)\}=R^{Q-1}\per\{(\delta_{1/R}S)\cap B_\varrho(x,1)\}
\qquad( R\geq 0).$$ With these preliminaries, the blow-up
procedure can easily be done. One has$$\frac{\per\{
S_R(x)\}}{R^{Q-1}}=\per\{(\delta_{1/R}S)\cap
B_\varrho(x,1)\}=\per\{(\delta_{1/R}S)_1(x)\}.$$Hence the limit we
have to compute is the following:$$\lim_{R\rightarrow
0^+}\per\{(\delta_{1/R}S)_1(x)\}.$$By left-translating at
$0\in\GG$ the set $(\delta_{1/R}S)_1(x)$, one gets
$$x^{-1}\bullet\{(\delta_{1/R}S)_1(x)\}=\delta_{1/R}\{(x^{-1}\bullet
S)_1\},$$and, by left-invariance of $\per$,
 one has $$\per\{(\delta_{1/R}S)_1(x)\}=\per\{(\delta_{1/R} (x^{-1}\bullet
S))_1\},$$for all $R\geq 0.$ Now, let us begin the analysis of the non-characteristic case.\\
\\ \noindent{\bf Case (a). Blow-up at non-characteristic points.} {\it Let us assume that $S\subset\GG$
is a $\cont^1$-smooth hypersurface and that $x\in S$ is
non-characteristic}. In such a case, at $x$, the hypersurface $S$
is oriented by the horizontal unit normal vector $\nn(x)$, i.e.
$\nn(x)$ is transversal to $S$ at $x$. Thus, at least locally
around $x$, we may think of $S$ as the (exponential image of the)
graph of a $\cont^1$-function w.r.t. the horizontal direction
$\nn(x)$. Moreover, at the level of the Lie algebra
$\gg\cong\TT_0\GG$, one may perform an orthonormal change of
coordinates in such a way that
$$\ee_1=X_1(0)=(L_{x^{-1}})_\ast\nn(x).$$ By the classical Implicit Function Theorem, we may write
$S_r=x^{-1}\bullet S_r(x)$, for some (small enough) $r>0$, as the
exponential image in $\GG$ of a $\cont^1$-smooth graph\footnote{Actually,
since the argument is local, $\psi$ may be defined just on a
suitable neighborhood of $0\in \ee_1^\perp\cong \R^{n-1}$.}
$$\Psi=\{(\psi(\xi),\xi)| \xi\in\R^{n-1}\}\subset\gg,$$ where
$\psi:\ee_1^{\perp}\cong\R^{n-1}\longrightarrow\R$ is a
$\cont^1$-function satisfying:

\begin{itemize}\item[${\bf(a_1)}$]$\psi(0)=0$;
\item[${\bf(a_2)}$]$\grad_{\R^{n-1}} \psi (0)=0.$
\end{itemize}By this way we get that $S_r=\exp\Psi\cap
B_\varrho(0,r),$ for all (small enough) $r>0$. Clearly, this
remark can be used to compute
\begin{equation}\label{dens1}\lim_{R\rightarrow 0^+}\per\{(\delta_{1/R} (x^{-1}\bullet
S))_1\}.\end{equation} So let us us fix a positive $r$ satisfying
the previous hypotheses and let $0\leq R\leq r$. We have
\begin{eqnarray}\nonumber(\delta_{1/R} (x^{-1}\bullet
S))_1&=&\delta_{1/R} (x^{-1}\bullet S_r(x))\cap
B_\varrho(0,1)\\\nonumber&=&\delta_{1/R}(S_r)\cap
B_\varrho(0,1)\\\nonumber&=&\delta_{1/R}(\exp\Psi\cap
B_\varrho(0,r))\cap
B_\varrho(0,1)\\\nonumber&=&\exp(\delta_{1/R}\Psi\cap
B_\varrho(0,r/R))\cap
B_\varrho(0,1)\\\label{dens2}&=&\exp(\widehat{\delta}_{1/R}\Psi)\cap
B_\varrho(0,1),\end{eqnarray}where $\widehat{\delta}_{t}$ are the
Carnot dilations induced on the Lie algebra $\gg$, i.e.
$\delta_t=\exp\circ\widehat{\delta}_{t}\, (t\in\R_+)$, or
equivalently, $\widehat{\delta}_{t}=\llog\circ\delta_t\,
(t\in\R_+)$. Furthermore we have
$$\widehat{\delta}_{1/R}\Psi=\widehat{\delta}_{1/R}\{(\psi(\xi),\xi)| \xi\in\R^{n-1}\}
=\Big\{\Big(\frac{\psi(\xi)}{R},(\widehat{\delta}_{1/R})\big|_{\ee_1^{\perp}}(\xi)\Big)\big|
\xi\in\R^{n-1}\Big\}.$$By using the change of variables
$\zeta:=(\widehat{\delta}_{1/R})|_{\ee_1^{\perp}}(\xi)$, we
therefore get that$$\widehat{\delta}_{1/R}\Psi=
\Big\{\Big(\frac{\psi\big((\widehat{\delta}_{R})|_{\ee_1^{\perp}}(\zeta)\big)}{R},\zeta\Big)\big|
\zeta\in\R^{n-1}\Big\}.$$Remind that $\psi\in \cont^1(U_0)$ where
$U_0$ is a suitable open neighborhood of $0\in\R^{n-1}$. From
$(a_1)$ and $(a_2)$ we obtain
$$\psi(\xi)=\psi(0) + \langle\grad_{\R^{n-1}}\psi(0), \xi \rangle_{\R^{n-1}} + {\rm o}(\|\xi\|_{\R^{n-1}})
={\rm o}(\|\xi\|_{\R^{n-1}}),$$for $\xi\rightarrow 0\in \R^{n-1}$.
Finally, since
$\|(\widehat{\delta}_{R})|_{\ee_1^{\perp}}(\zeta)\|_{\R^{n-1}}\longrightarrow
0$ as long as $R\rightarrow 0^+$, it follows that
\begin{equation}\label{dens3}\Psi_\infty=\lim_{R\rightarrow
0^+}\widehat{\delta}_{1/R}\Psi=\exp(\ee_1^\perp)=\mathcal{I}(X_1(0)).\end{equation}
We stress that $\mathcal{I}(X_1(0))$ is the vertical hyperplane
through the identity $0\in\GG$ and orthogonal to the horizontal
direction $X_1(0)$. Hence, the limit \eqref{dens1} can be
computed by using \eqref{dens2} and \eqref{dens3}. More precisely,
we have

\begin{equation}\label{dens4}\lim_{R\rightarrow 0^+}\frac{\per \{S\cap
B_{\varrho}(x, R)\}}{R^{Q-1}}=\lim_{R\rightarrow
0^+}\per\{(\delta_{1/R} (x^{-1}\bullet
S))_1\}=\per(\mathcal{I}(X_1(0))\cap B_\varrho(0,1))\end{equation}
By the previous change of variable, it is also clear that

$$\per(\mathcal{I}(X_1(0))\cap B_\varrho(0,1))=\per(\mathcal{I}(\nn(x))\cap B_\varrho(x,1)).$$
As it can easily be seen, the horizontal perimeter $\per$
restricted to any vertical hyperplane coincides with the Euclidean
Hausdorff measure $\Ar$ on the hyperplane $\mathcal{I}(\nn(x))$,
i.e.
$$\per\res\mathcal{B}\cap\mathcal{I}(\nn(x))=\Ar\res\mathcal{B}\cap\mathcal{I}(\nn(x))\qquad \forall\,\,\mathcal{B}
\in\mathcal{B}or(\mathcal{I}(\nn(x))).$$Therefore$$\lim_{R\rightarrow
0^+}\frac{\per \{S\cap B_{\varrho}(x,
R)\}}{R^{Q-1}}=\kappa_\varrho(\nn(x)),$$which was to be
proven.\[\]

Now we continue the study of the characteristic case.\\
\\ \noindent{\bf Case (b). Blow-up at the characteristic set.}
{\it We are now assuming that $S\subset\GG$ is a $\cont^i$-smooth
hypersurface $(i\geq 2)$ and that $x\in C_S$, i.e. is a
characteristic point}. In such a case the hypersurface $S$ is
oriented, at the point $x$, by some vertical vector. Hence, at
least locally around $x$, we may think of $S$ as the (exponential
image of the) graph of a $\cont^i$-function w.r.t. some given
vertical direction  $X_\alpha\,(\alpha\in I\vv)$ which is
transversal to $S$ at $x$, i.e. $\langle X_\alpha(x),
\nu(x)\rangle\neq 0$, where $\nu$ is the Riemannian unit normal
vector to $S$ at  $x$. Here $X_\alpha$ is a vertical
left-invariant vector field of the fixed left-invariant frame
$\underline{X}=\{X_1,...,X_n\}$ on $\GG$ and $\alpha\in I\vv$ is
any index of ``vertical'' type; see Section \ref{prelcar},
Notation \ref{1notazione}. {\it Furthermore,  we are assuming
that}
$$\mathrm{ord}(\alpha):= i\qquad
  (i\in\{2,..,k\}).$$

To the sake of simplicity, as in the non-characteristic case, we
left-translate the hypersurface in such a way that $x$ will
coincide with the identity $0\in\GG$. To this end, it suffices to
replace $S$ by $x^{-1}\bullet S$.

At the level of the Lie algebra $\gg\cong\TT_0\GG$, we introduce
the hyperplane $\ee_\alpha^\perp$ through the origin
$0\in\R^{n}\cong\gg$ and orthogonal to $\ee_\alpha=X_{\alpha}(0)$.
In the sequel, the exponential image in $\GG$ of
$\ee_\alpha^\perp$ will be denoted by $\mathcal{I}(X_\alpha(0))$,
i.e. $\mathcal{I}(X_\alpha(0))=\exp(\ee_\alpha^\perp)$. We stress
that $\ee_\alpha^\perp$ is the ``natural'' domain of a graph along
the direction $\ee_\alpha$. By the classical Implicit Function
Theorem, for some (small enough) $r>0$, we may write
$S_r=x^{-1}\bullet S_r(x)$ as the exponential image in $\GG$ of
the $\cont^i$-graph
$$\Psi=\Big\{\big(\xi_1,...,\xi_{\alpha-1}\underbrace{, \psi(\xi),}_{\scriptsize{\alpha-th\, place}}\xi_{\alpha+1},...,\xi_n \big)\,
\big|\, \xi:=(\xi_1,...,\xi_{\alpha-1}, 0,\xi_{\alpha+1},...,\xi_n
)\in \ee_\alpha^\perp\cong \R^{n-1}\Big\}\subset\gg,$$ where
$\psi:\ee_\alpha^{\perp}\cong\R^{n-1}\longrightarrow\R$ is a
$\cont^i$-function satisfying:

\begin{itemize}\item[${\bf(b_1)}$]$\psi(0)=0$;
\item[${\bf(b_2)}$]${\partial\psi}/{\partial\xi_j} (0)=0$ for
every $j=1,...,\DH\,(=\dim\HH)$.
\end{itemize}Thus we get that $S_r=\exp\Psi\cap
B_\varrho(0,r),$ for every (small enough) $r>0$. Now we  apply the
previous considerations to compute
$$\lim_{R\rightarrow 0^+}\per\{(\delta_{1/R} (x^{-1}\bullet
S))_1\}.$$By arguing exactly as in the non-characteristic case, we
get that $$(\delta_{1/R} (x^{-1}\bullet
S))_1=\exp(\widehat{\delta}_{1/R}\Psi)\cap B_\varrho(0,1),$$where
$\widehat{\delta}_{t}$ are the Carnot dilations induced on the Lie
algebra $\gg$.
Furthermore\begin{eqnarray*}\widehat{\delta}_{1/R}\Psi&=&
\widehat{\delta}_{1/R}\big\{\big(\xi_1,...,\xi_{\alpha-1},
\psi(\xi),\xi_{\alpha+1},...,\xi_n \big)\, |\, \xi\in
\ee_\alpha^\perp\big\}\\
&=&\Big\{\Big(\frac{{\xi_1}}{R},...,\frac{\xi_{\alpha-1}}{R^{{\rm
ord}(\alpha-1)}},
\frac{\psi(\xi)}{R^i},\frac{\xi_{\alpha+1}}{R^{{\rm
ord}(\alpha+1)}},...,\frac{\xi_n}{R^k} \Big)\,\big|\, \xi\in
\ee_\alpha^\perp\Big\}.\end{eqnarray*}By setting
$$\zeta:=(\widehat{\delta}_{1/R})\big|_{\ee_\alpha^{\perp}}
\xi=\Big(\frac{{\xi_1}}{R},...,\frac{\xi_{\alpha-1}}{R^{{\rm
ord}(\alpha-1)}}, 0,\frac{\xi_{\alpha+1}}{R^{{\rm
ord}(\alpha+1)}},...,\frac{\xi_n}{R^k} \Big),$$where
$\zeta=(\zeta_1,...,\zeta_{\alpha-1},0,\zeta_{\alpha+1},...,\zeta_n)\in\ee_\alpha^\perp$,
we therefore get
$$\widehat{\delta}_{1/R}\Psi=
\Big\{\Big(\zeta_1,...,\zeta_{\alpha-1},
\frac{\psi\big(\widehat{\delta}_{R}(\zeta)|_{\ee_\alpha^{\perp}}\big)}{R^i},\zeta_{\alpha+1},...,\zeta_n\Big)\big|
\zeta\in\ee_\alpha^\perp\Big\}.$$Remind that $\psi\in
\cont^i(U_0)$ where $U_0$ is an open neighborhood of
$0\in\ee_\alpha^{\perp}\cong\R^{n-1}$. Moreover, it is immediate
to see that
$$\widehat{\delta}_{R}(\zeta)|_{\ee_\alpha^{\perp}}\longrightarrow
0$$as long as $R\rightarrow 0^+$. So we have to evaluate the limit
\begin{equation}\label{lim}\widetilde{\psi}(\zeta):=\lim_{R\rightarrow
0^+}\frac{\psi\big(\widehat{\delta}_{R}(\zeta)|_{\ee_\alpha^{\perp}}\big)}{R^i}.\end{equation}The
first remark is that, if this limit equals $+\infty$, one
immediately obtains
$$\lim_{R\rightarrow
0^+}\frac{\per\{ S_R(x)\}}{R^{Q-1}}=\lim_{R\rightarrow
0^+}\per\{\exp(\widehat{\delta}_{1/R}\Psi)\cap
B_\varrho(0,1)\}=0,$$ since $\widehat{\delta}_{1/R}\Psi\cap
B_\varrho(0,1)\rightarrow{\emptyset}$, for $R\rightarrow 0^+$.

Now we make use of a Taylor's expansion of $\psi$ up to the $l$-th
- Euclidean - order ($l\leq i$) together with $(b_1)$ and $(b_2)$.
We thus get

\begin{eqnarray*}\psi\big(\widehat{\delta}_{R}(\zeta)|_{\ee_\alpha^{\perp}}\big)&=
&\psi(0) +\sum_{j_1}R^{{\rm
ord}(j_1)}\frac{\partial\psi}{\partial\zeta_{j_1}}(0)\zeta_{j_1}
+\sum_{j_1, j_2}R^{{\rm ord}(j_1)+{\rm
ord}(j_2)}\frac{\partial^{\scriptsize(2)}\psi}{\partial\zeta_{j_1}\partial\zeta_{j_2}}(0)\zeta_{j_1}\zeta_{j_2}\\&&+...+
\sum_{j_1,..., j_i}R^{{\rm ord}(j_1)+...+{\rm
ord}(j_i)}\frac{\partial^{\scriptsize(i)}\psi}{\partial\zeta_{j_1}...
\partial\zeta_{j_i}}(0)\zeta_{j_1}\cdot...\cdot\zeta_{j_i}+{\rm
o}\big(\|\widehat{\delta}_{R}(\zeta)|_{\ee_\alpha^{\perp}}\|^i_{\R^{n-1}}\big)\\&=&
\sum_{j_1}R^{{\rm
ord}(j_1)}\frac{\partial\psi}{\partial\zeta_{j_1}}(0)\zeta_{j_1}
+\sum_{j_1, j_2}R^{{\rm ord}(j_1)+{\rm
ord}(j_2)}\frac{\partial^{\scriptsize(2)}\psi}{\partial\zeta_{j_1}\partial\zeta_{j_2}}(0)\zeta_{j_1}\zeta_{j_2}\\&&+...+
\sum_{j_1,..., j_i}R^{{\rm ord}(j_1)+...+{\rm
ord}(j_l)}\frac{\partial^{\scriptsize(l)}\psi}{\partial\zeta_{j_1}...
\partial\zeta_{j_i}}(0)\zeta_{j_1}\cdot...\cdot\zeta_{j_l}+{\rm
o}\big(R^{{\rm ord}(j_1)+...+{\rm ord}(j_l)}\big),\end{eqnarray*}
where $j_l\in\{1,...,(\alpha-1),(\alpha+1),...,n\}$. Thus

\begin{eqnarray*}\frac{\psi\big(\widehat{\delta}_{R}(\zeta)|_{\ee_\alpha^{\perp}}\big)}{R^i}&=
& \sum_{j_1}R^{{\rm
ord}(j_1)-i}\frac{\partial\psi}{\partial\zeta_{j_1}}(0)\zeta_{j_1}
+\sum_{j_1, j_2}R^{{\rm ord}(j_1)+{\rm
ord}(j_2)-i}\frac{\partial^{\scriptsize(2)}\psi}{\partial\zeta_{j_1}\partial\zeta_{j_2}}(0)
\zeta_{j_1}\zeta_{j_2}\\&&+...+ \sum_{j_1,..., j_l}R^{{\rm
ord}(j_1)+...+{\rm
ord}(j_l)-i}\frac{\partial^{\scriptsize(l)}\psi}{\partial\zeta_{j_1}...
\partial\zeta_{j_l}}(0)\zeta_{j_1}\cdot...\cdot\zeta_{j_l}+{\rm
o}\big(R^{{\rm ord}(j_1)+...+{\rm
ord}(j_l)-i}\big).\end{eqnarray*}By hypothesis
\begin{equation}\label{dercond}\frac{\partial^{\scriptsize(l)}
\psi}{\partial\zeta_{j_1}...\partial\zeta_{j_l}}(0)=0\qquad\mbox{whenever}\qquad{\rm
ord}(j_1)+...+{\rm ord}(j_l)<i\qquad(1\leq l\leq i).
\end{equation}This implies that the limit
\eqref{lim} exists. Therefore
\begin{equation*}\Psi_\infty=\lim_{R\rightarrow
0^+}\widehat{\delta}_{1/R}\Psi=
\big\{\big(\zeta_1,...,\zeta_{\alpha-1},\widetilde{\psi}
(\zeta),\zeta_{\alpha+1},...,\zeta_n\big)\big|
\zeta\in\ee_\alpha^\perp\big\},\end{equation*}where
$\widetilde{\psi}$ is the homogeneous polynomial function of
degree $l$ and - anisotropic - order $i={\rm ord}(\alpha)$ given
by
\begin{eqnarray}\nonumber\widetilde{\psi}(\zeta)=\sum_{\stackrel{{j_1}}{\scriptsize{\rm
ord}(j_1)=i}}\frac{\partial\psi}{\partial\zeta_{j_1}}(0)\zeta_{j_1}+
\sum_{\stackrel{{j_1, j_2}}{\scriptsize{{\rm ord}(j_1)+{\rm
ord}(j_2)=i}}}\frac{\partial^{\scriptsize(2)}\psi}{\partial\zeta_{j_1}\partial\zeta_{j_2}}(0)\zeta_{j_1}\zeta_{j_2}\\
\label{jhhjas}+ \sum_{\stackrel{{j_1,...,j_l}}{\scriptsize{{\rm
ord}(j_1)+...+{\rm
ord}(j_l)=i}}}\frac{\partial^{\scriptsize(l)}\psi}{\partial\zeta_{j_1}...
\partial\zeta_{j_l}}(0)\zeta_{j_1}\cdot...\cdot\zeta_{j_l}.\end{eqnarray}The
thesis easily follows.

\end{proof}

\section{A technical lemma: Coarea Formula for the
$\HS$-gradient}\label{COAR} Let $S\subset\GG$ be a $\mathbf{C}^2$
hypersurface and $\varphi:S\longrightarrow\R$ be a
$\mathbf{C}^1$-smooth function. As it is well-known, $x\in S$ is a
{\it regular point} of $\varphi$ when
$|\grad\,\varphi|\neq 0$ and a {\it critical point} otherwise. In
other words, $x\in S$ is a regular point if and only if the
differential $d\varphi|_x:T_xS\longrightarrow T_{\varphi(x)}\R$ is
surjective. Furthermore, $s\in\R$ is a {\it regular value} of
$\varphi$ if and only if every point of $\varphi^{-1}[s]$ is  a
regular point of $\varphi$ (and by convention $s\in \R$ is a
regular value if $\varphi^{-1}[s]=\emptyset$) and is a {\it
critical value} if it is not a regular value.

\begin{oss}\label{prea}Clearly, any $(n-2)$-dimensional $\cont^1$-smooth submanifold $\Sigma$ of $S$
can be thought, at least locally, as the level-set of a
$\cont^1$-function defined on $S$. Furthermore, from a purely
``geometric'' point of view, the characteristic set of
$\Sigma\subset S$ is the kernel of the projection onto $\HS$ of
the unit normal vector $\eta$ along $\Sigma$, i.e.
$C_\Sigma=\{x\in \Sigma\big| \P\ss\eta(x)=0\}$.

Let $\varphi\in\cont^1(S)$ and, for any regular value $s\in\R$ of
$\varphi$, let us consider the level set of $\varphi$, i.e.
$\varphi^{-1}[s]$. It is clear that $\varphi^{-1}[s]$ is a
$(n-2)$-dimensional $\cont^1$-smooth closed submanifold of $S$.
Note that $\varphi^{-1}[s]$ is geometrically $\HH$-regular if, and
only if, $|\qq\varphi|\neq 0$ everywhere along $\varphi^{-1}[s]$;
see Definition \ref{iuoi}. As an application of Theorem 2.16 in
\cite{Mag2}, the characteristic set $C_{\varphi^{-1}[s]}$ of
$\varphi^{-1}[s]$ is a 0-measure set w.r.t. the intrinsic
Hausdorff measure, i.e. $\mathcal{H}_{\bf
cc}^{Q-2}(C_{\varphi^{-1}[s]})=0$; see also Remark
\ref{carsetgen}. Hence $\nis(C_{\varphi^{-1}[s]})=0$. It is clear
that $C_{\varphi^{-1}[s]}$ is a closed subset of
${\varphi^{-1}[s]}$ - under the relative topology - and that
$$C_{\varphi^{-1}[s]}=\big\{x\in\varphi^{-1}[s]\big|
|\qq\varphi(x)|=0\big\}.$$
\end{oss}

As it is well-known, the classical {\it Sard's Theorem} (see, for
instance, \cite{NAR}) states  that {\it the set of critical values
of $\varphi$ is a 0-measure set} (w.r.t. the Lebesgue measure on
$\R$) and so, for $\mathcal{L}^1$-a.e. $s\in \R$, the set
$\varphi^{-1}[s]\subset S$ is a $\cont^1$-smooth - closed -
submanifold of $S$. For any fixed regular value $s$ of $\varphi$,
we shall consider the open subset
$(\varphi^{-1}[s])^\ast:=\varphi^{-1}[s]\setminus
C_{\varphi^{-1}[s]}$ of $\varphi^{-1}[s]$. At each point of
$(\varphi^{-1}[s])^\ast$ the horizontal unit normal vector is
well-defined and it can be written - up to the choice of an
orientation - as $\eta\ss=\frac{\qq\varphi}{|\qq\varphi|}$.

\begin{teo}\label{TCOAR}Let $S\subset\GG$ be a $\mathbf{C}^2$-smooth hypersurface and
let $\varphi\in\cont^1(S)$. Then
\begin{equation}\label{1coar}\int_{S}|\qq\varphi(x)|\per(x)=\int_{\R}\nis\{\varphi^{-1}[s]\cap
S \}ds\end{equation}and
\begin{equation}\label{2coar}\int_{S}\psi(x)|\qq\varphi(x)|\per(x)=\int_{\R}ds \int_{\varphi^{-1}[s]\cap S}\psi(y)\nis(y)\end{equation}
for every $\psi\in L^1(S,\per)$.\end{teo}

\begin{proof}We begin by assuming that $S$ be non-characteristic.
So let us fix any regular value $s$ of $\varphi$ and consider the
set $(\varphi^{-1}[s])^\ast\subset S$. At each point of
$(\varphi^{-1}[s])^\ast$ we set
$\eta\ss=\frac{\qq\varphi}{|\qq\varphi|}$. Thus we may choose an
oriented orthonormal basis $\tau_3,...,\tau_h$ of the orthogonal
complement of $\eta\ss$ in $\HS$, i.e.
$\eta\ss^{\perp\ss}:=\rm{span}_{\R}\{\tau_3,...,\tau_h\}$, where
the notation is that used for an adapted orthonormal frame
$\underline{\tau}=\{\tau_1, \tau_2,...,\tau_n\}$ to $S$. More
precisely, for such a frame, at each point of
$(\varphi^{-1}[s])^\ast,$  we have  $\tau_1=\nn, \tau_2=\eta\ss$
and $\eta\ss^{\perp\ss}=\rm{span}_{\R}\{\tau_3,...,\tau_h\}$. Note
that $\HH=\nn S\oplus\HS$ and that $\HS=\eta\ss \oplus
\eta\ss^{\perp\ss}$. We set
$$S^\ast:=\bigcup_{s \in \{\mbox{\tiny Regular Values of\,} \varphi\}}(\varphi^{-1}[s])^\ast.$$
Now let us consider the following differential forms on $S$:
$$\phi_2,\quad \nis=\phi_3\wedge...\wedge\phi_n|_S,$$where
$\phi_I$ is the 1-form dual of $\tau_I$, i.e.
$\phi_I^\ast(\tau_J)=\delta_I^J$ for every $I, J=2,...,n$
($\delta_I^J$ denotes the Kronecker delta). It is clear that
$\phi_2=\eta\ss^\ast=\frac{d\ss\varphi}{|\qq\varphi|}$ and that
$\per=\phi_2\wedge\nis$. The natural volume form on $\R$ can be
represented by the differential 1-form $ds$. Therefore, one gets
$\varphi^\ast ds=d\varphi=(d\ss\varphi + d\vv\varphi)|_S$. Note
that,  at each point of $S$, $d\vv\varphi|_S$ is a linear
combination of the
 1-forms $\phi_\alpha|_S\,(\alpha\in I\vv)$. Thus
 $\varphi^\ast
 ds=(|\qq\varphi|\phi_2+d\vv\varphi)|_S$ and it turns out that
 $$\varphi^\ast
 ds\wedge\nis|_{\varphi^{-1}[s]}=|\qq\varphi|\per|_S.$$ The last one can be regarded as
  an ``infinitesimal'' Coarea Formula for any non-characteristic
  hypersurface.  Now we consider the general case in which $S$
  can have a non-empty characteristic set $C_S$. The
  above arguments can be repeated verbatim, after replacing $S$
 with $S\setminus C_S$.

  Here is worth to note that
  both members of this formula  vanish in the following three cases: (i) {\it $x\in C_S$, i.e. $x$ is a characteristic
   point of $S$;} (ii) {\it $s\in \R$
  is a critical value of $\varphi$}; (iii) {\it $s\in \R$ is a regular value of $\varphi$
  but  $y\in\varphi^{-1}[s]\cap
  C_{\varphi^{-1}[s]}$, i.e. $y$ is a characteristic point of
  $\varphi^{-1}[s]$}. Therefore, in proving the theorem, we may first replace $S$ by $S\setminus C_S$ and then
  consider
  the set
  $(S\setminus C_S)^\ast$. In fact, the following identities
  hold true:
  \begin{eqnarray*}\int_{S}\psi(x)|\qq\varphi(x)|\per(x)
  &=&\int_{S\setminus C_S}\psi(x)|\qq\varphi(x)|\per(x)\\&=&
  \int_{(S\setminus C_S)^\ast}\psi(x)|\qq\varphi(x)|\per(x)\end{eqnarray*}
  and
$$\int_{\R}dt
\int_{\varphi^{-1}[s]\cap S}\psi(y) \nis(y)=\int_{\R}ds
\int_{\varphi^{-1}[s]\cap (S\setminus C_S)^\ast}\psi(y) \nis(y).$$
The classical {\it Lemma on fiber integration}\footnote{The lemma
on fiber integration can be stated as follows:\begin{lemma}Let
$f:M^{n+m}\longrightarrow N^n$ be a submersion of $\cont^1$-smooth
manifolds, $\alpha$ an $m$-form on $M$ and $\beta$ an $n$-form on
$N$. Then for every $g\in\cont_0(M)$, one has
$$\int_N\Big(\int_{f^{-1}[y]}g\alpha\Big)\beta(y)=\int g \alpha\wedge f^\ast\beta.$$
\end{lemma}}, an equivalent of Fubini's Theorem for
submersions, yields \eqref{1coar} and \eqref{2coar} in the case
where $\psi$ is continuous and compactly supported on $S$. Finally,
the general case  follows by a standard approximation argument.
\end{proof}

\begin{oss} We would like to point out that the previous result can also be
deduced by using the Riemannian Coarea Formula as we shall see
below. Indeed, let $S\subset\GG$ be a $\mathbf{C}^2$-smooth
hypersurface and $\varphi:S\longrightarrow\R$ be a piecewise
$\cont^1$-smooth function. Then
\begin{equation}\label{Riecoar}\int_{S}\phi(x)|
\grad\ts\varphi(x)|\sigma^{n-1}\rr(x)=\int_{\R}ds
\int_{\varphi^{-1}[s]\cap
S}\phi(y)\sigma^{n-2}\rr(y)\end{equation} for every $\psi\in
L^1(S,\sigma^{n-1}\rr)$. This is the Riemannian Coarea Formula;
see
 \cite{BuZa}, \cite{Ch2, Ch3}, \cite{FE}.  Now, by choosing
 $$\phi=\psi\frac{|\grad\ss\varphi|}{|\grad\ts\varphi|}|\P\cc\nu|$$for some $\psi\in
 L^1(S,\per)$, we get that
\begin{eqnarray*}\int_{S}\phi|
\grad\ts\varphi|\sigma^{n-1}\rr=\int_{S}\psi\frac{|\grad\ss\varphi|}{|\grad\ts\varphi|}
\underbrace{|\P\cc\nu|\sigma^{n-1}\rr}_{=\per}=\int_{S}|\qq\varphi|\per.
\end{eqnarray*}Along $\varphi^{-1}[s]$ one has
$\eta=\frac{\grad\ts\varphi}{|\grad\ts\varphi|}$ and this implies
$|\P\ss\eta|=\frac{|\grad\ss\varphi|}{|\grad\ts\varphi|}$.
Therefore
\begin{eqnarray*}\int_{\R}ds
\int_{\varphi^{-1}[s]\cap S}\phi(y)\sigma^{n-2}\rr&=&\int_{\R}ds
\int_{\varphi^{-1}[s]\cap
S}\psi\frac{|\grad\ss\varphi|}{|\grad\ts\varphi|}|\P\cc\nu|\sigma^{n-2}\rr\\&=&\int_{\varphi^{-1}[s]\cap
S}\psi\underbrace{|\P\ss\eta||\P\cc\nu|\sigma^{n-2}\rr}_{=\nis}\\&=&\int_{\varphi^{-1}[s]\cap
S}\psi\nis.\end{eqnarray*}
\end{oss}

\section{Poincar\'e inequalities on compact hypersurfaces}\label{poincinsect}
Let $S\subset\GG$ be a $\cont^1$-smooth compact hypersurface with
- or without - boundary. As in the Riemannian setting (see
\cite{Cheeger}) one may give the following: \begin{Defi}The {\it
isoperimetric constant} ${\rm Isop}(S)$ of $S$ is defined as
follows:\begin{itemize} \item if $\partial S=\emptyset$ we set
$${\rm Isop}(S):=\inf\frac{\nis(N)}{\min\{\per(S_1),\per(S_2)\}},$$where
the infimum is taken over all smooth $(n-2)$-dimensional
submanifolds $N$ of $S$ which divide $S$ into two hyperurfaces
$S_1, S_2$ with common boundary $N=\partial S_1=\partial S_2$;
\item if $\partial S\neq \emptyset$ we set
$${\rm Isop}(S):=\inf\frac{\nis(N)}{\per(S_1)},$$where
$N\cap \partial S=\emptyset$ and $S_1$ is the unique
$(n-2)$-dimensional submanifold of $S$ such that  $N=\partial
S_1$.

\end{itemize}

\end{Defi}

\begin{oss}\label{xaz}It is important to say that, in general, the isoperimetric
constant ${\rm Isop}(S)$ might vanish. Examples of this fact can
be constructed, for instance, in the case of the 1st Heisenberg
group $\mathbb{H}^1$ on $\R^3$; see Example \ref{hjk} below.
Nevertheless, if we make some further assumptions on the
underlying geometry, this phenomenon cannot occur. More precisely,
let us assume that:
\begin{itemize}\item[${\bf (H)}$]For every smooth
$(n-2)$-dimensional submanifold $N\subset S$, one has \[\dim\,
C_{N}<n-2.\]
\end{itemize}

\end{oss}

\begin{oss}\label{pcar}Let $S\subset\GG$ be $\cont^2$-smooth and
non-characteristic. Moreover let us assume that the horizontal
tangent bundle $\HS$ be generic and horizontal.  For example, this
is the case of any $\cont^2$-smooth non-characteristic
hypersurface in the Heisenberg group $\mathbb{H}^n$ for $n>1$. In
other words, we are assuming that $(S, \HS)$ is a CC-space; see
Section \ref{prelcar} and Remark \ref{indbun}. Then we claim that
${\bf (H)}$ is automatically satisfied. Indeed, let $\UU\subset S$
be any relatively compact open set with $\cont^1$-smooth boundary.
Note that $\partial \UU$ can be regarded as a smooth hypersurface
in $S$, oriented by its unit normal $\eta$. Then, as an
application of Frobenius' Theorem, one infers that
$C_{\partial\UU}$ cannot have a non-empty interior and this
immediately implies that ${\bf (H)}$ holds true\footnote{More
generally, let $(N^n, \HH)$ be  an $n$-dimensional CC-space and let
$M^{n-1}\subset N$ be a $\cont^2$-smooth immersed hypersurface.
Moreover, let us define the characteristic set $C_M$ of $M$
exactly as in the case of Carnot groups, i.e. $C_M:=\{x\in M:
|\P\cc\nu\rr(x)|=0\}$. Then, as an application of Frobenius'
Theorem, one infers that $C_M$ cannot have a non-empty interior.}.

\end{oss}

\begin{es}\label{hjk}Let $S$ be a compact surface $S\subset\mathbb{H}^1$ without
boundary. Moreover, let us assume that there exists a regular,
simple, closed, horizontal curve $\gamma:[0,
1]\longrightarrow\mathbb{H}^1$ such that ${\rm Im}\,\gamma\subset
S$. Then, it follows immediately that ${\rm Isop}(S)=0$.
\end{es}

As a consequence of the Coarea Formula \eqref{1coar} we may
generalize to our Carnot setting some classical results about
isoperimetric constants and global Poincar\'e inequalities which
can be found in the books \cite{Ch1, Ch2, Ch3}; see also
\cite{Cheeger}, \cite{Yau}.

\begin{teo}\label{zaz}Let $S\subset\GG$ be a $\cont^1$-smooth compact hypersurface. \begin{itemize} \item[{\bf (a)}]
If $\partial S=\emptyset$, then
$${\rm Isop}(S)=\inf\frac{\int_S|\qq \psi|\per}{\int_S|\psi|\per},$$where
the infimum is taken over all $\cont^1$-smooth functions on $S$
such that $\int_S\psi\per=0$. \item[{\bf (b)}] If $\partial S\neq
\emptyset$, then $${\rm Isop}(S)=\inf\frac{\int_S|\qq
\psi|\per}{\int_S|\psi|\per},$$ where the infimum is taken over
all $\cont^1$-smooth functions on $S$ such that $\psi|_{\partial
S}=0$.
\end{itemize}
\end{teo}
\begin{proof}The proof repeats almost verbatim the arguments of Theorem 1 in
\cite{Yau}. We just prove the theorem for $\partial S=\emptyset$
since the other case is analogous. First, let us prove the
inequality

$${\rm Isop}(S)\leq\inf\frac{\int_S|\qq
\psi|\per}{\int_S|\psi|\per}$$where $\psi\in\cont^1(S)$ and
$\int_S\psi\per=0$. To prove this inequality let us consider the
auxiliary functions $\psi^{+}=\max\{0, \psi\},\, \psi^{-}=\max\{0,
-\psi\}$. By applying the Coarea Formula \eqref{1coar}, the very
definition of ${\rm Isop}(S)$
 and {\it Cavalieri's principle} \footnote{\label{CavPrin}The following lemma, also known as {\it
Cavalieri's principle}, is an easy consequence of Fubini's
Theorem:
\begin{lemma}Let X be an abstract space, $\mu$ a measure on $X$,
$\alpha>0$, $\varphi\geq 0$ and $A_t=\{x\in X: \varphi>t\}$. Then
$$\int_0^{+\infty}t^{\alpha-1}\mu(A_t)\,dt=\frac{1}{\alpha}\int_{A_0}\varphi^\alpha\,d\mu.$$\end{lemma}}
one gets$$\int_{S}|\qq\psi^{\pm}|\per=\int_0^{+\infty}\nis\{x\in
S: \psi^\pm=t\}\,dt\geq{\rm Isop}(S)\int_S|\psi^\pm|\per.$$Now we
shall prove the reversed inequality. So let us assume that
$\per(S_1)\leq\per(S_2)$ and let $\epsilon>0$. We now define  the
 function

\begin{eqnarray}\label{ODER1}\psi_{\epsilon}(x)|_{S_1}:=
\left\{\begin{array}{ll} \frac{\varrho(x, N)}{\epsilon}
\,\,\,\mbox{if }\,\, \varrho(x, N)\leq \epsilon
\\\\
\,\,\,\, 1\qquad\mbox{if } \, \varrho(x, N)> \epsilon.\end{array}
\right.\qquad\psi_{\epsilon}(x)|_{S_2}:= \left\{\begin{array}{ll}
-\alpha\frac{\varrho(x, N)}{\epsilon} \,\,\,\mbox{if }\,\,
\varrho(x, N)\leq \epsilon
\\\\
\,\,\,\,\,   -\alpha  \,\,\qquad\mbox{if } \,\,\varrho(x, N)>
\epsilon\end{array} \right.\end{eqnarray}where the constant
$\alpha$ depends on $\epsilon$ and is chosen in such a way that
$\int_S\psi_\epsilon\per=0.$ Obviously
$$\lim_{\epsilon\rightarrow
0}\alpha=\frac{\per(S_1)}{\per(S_2)}.$$ Since
\begin{eqnarray*}\int_{S}|\qq\psi_\epsilon|\per&=&\frac{1+\alpha}
{\epsilon}\int_{N_\epsilon:=\{x\in S:\varrho(x,
N)\leq\epsilon\}}|\qq\varrho(x,N)|\per\\&=&\frac{1+\alpha}
{\epsilon}\int_{0}^{\epsilon}\nis\{x\in N_\epsilon: \varrho(x,
N)=t\}\,dt,\end{eqnarray*}one gets $$\lim_{\epsilon\rightarrow
0}\int_{S}|\qq\psi_\epsilon|\per=(1+\alpha)\nis(N).$$Moreover
$\lim_{\epsilon\rightarrow
0}\int_{S}|\psi_\epsilon|\per=\per(S_1)+\alpha\per(S_2)$. Putting
all together we finally get that $$\lim_{\epsilon\rightarrow
0}\frac{\int_S|\qq
\psi_\epsilon|\per}{\int_S|\psi_\epsilon|\per}\leq {\rm Isop}(S).
$$

\end{proof}

\begin{corollario}\label{555}Let $\lambda_1$ be the first non-zero eigenvalue of either the closed eigenvalue problem
 or the Dirichlet  eigenvalue problem; see Definition \ref{DEFIEIGEN}.
  Then $\lambda_1\geq \frac{({\rm Isop}(S))^2}{4} $. \end{corollario}

\begin{proof}From Theorem \ref{zaz},  by  Cauchy-Schwartz Inequality, one gets$$
\int_S|\psi|^2\per\leq \frac{2}{{\rm Isop}(S)}\|\psi\|_{L^2(S;
\per)}\|\qq\psi\|_{L^2(S; \per)},$$i.e. $\|\psi\|^2_{L^2(S;
\per)}\leq \frac{4}{({\rm Isop}(S))^2}\|\qq\psi\|^2_{L^2(S;
\per)}$. By using (vi) of Proposition \ref{Properties} with
$\psi=\varphi$ and the hypotheses one gets
$$\lambda_1\int_S|\psi|^2\per=\int_S\psi\,\lh\psi\,\per=\int_S|\qq\psi|^2\per.$$
\end{proof}

Now let us generalize to our Carnot setting another isoperimetric
constant, which is well-known  for compact Riemannian manifolds;
see, for instance, \cite{Yau}.
\begin{Defi} The {\it isoperimetric constant}
${\rm Isop}_0(S)$ of a $\cont^1$-smooth  compact hypersurface -
with boundary - $S\subset\GG$ is given by
$${\rm Isop}_0(S):=\inf\bigg\{\frac{\nis(\partial S_1\cap \partial S_2)}{\min\{\per(S_1),\per(S_2)\}}\bigg\},$$where
the infimum is taken over all decompositions $S=S_1\cup S_2$ such
that $\per (S_1\cap S_2)=0$.

\end{Defi}

\begin{teo}\label{zdaz}Let $S\subset\GG$ be a $\cont^1$-smooth compact hypersurface with boundary.
Then one has
$${\rm Isop}_0(S)=\inf\bigg\{\frac{\int_S|\qq \psi|\per}{\inf_{\beta\in \R}\int_S|\psi-\beta|\per}\bigg\},$$where
the $\inf$ is taken over all $\cont^1$-functions defined on $S$.
\end{teo}
\begin{proof}The proof is analogous to that of Theorem 6 in
\cite{Yau}. First, let us prove the inequality

$${\rm Isop}(S)\leq\inf\frac{\int_S|\qq
\psi|\per}{\int_S|\psi|\per}.$$To this purpuse let us define the
functions $\psi^{+}:=\max\{0, \psi-k\},\, \psi^{-}:=-\min\{0,
\psi-k\}$, where $k\in\R$ is any constant such that:

\begin{eqnarray*}\per\{x\in S: \psi^{+}>0\}&\leq& \frac{1}{2}\per(S),\\
\per\{x\in S: \psi^{-}>0\}&\leq
&\frac{1}{2}\per(S).\end{eqnarray*} Again, by applying Coarea
Formula \eqref{1coar}, the very definition of ${\rm Isop}_0(S)$
and Cavalieri's principle one infers
that$$\int_{S}|\qq\psi^{\pm}|\per=\int_0^{+\infty}\nis\{x\in S:
\psi^\pm=t\}\,dt\geq{\rm Isop}(S)\int_S|\psi^\pm|\per.$$Now we
prove the other inequality. Assuming $\per(S_1)\leq\per(S_2)$ and
$\epsilon>0$, we define  the
 function

\begin{eqnarray}\label{ODER2}\psi_{\epsilon}(x)|_{S_1}:=1\qquad\psi_{\epsilon}(x)|_{S_2}:=
 \left\{\begin{array}{ll}
1-\frac{\varrho(x,\partial S_1\cap
\partial S_2)}{\epsilon}\,\,\,\mbox{if }\,\, \varrho(x,\partial S_1\cap
\partial S_2))\leq
\epsilon
\\\\
\,\,\,\qquad   0 \,\qquad\qquad\mbox{if } \,\,\varrho(x,\partial
S_1\cap
\partial S_2))>
\epsilon.\end{array} \right.\end{eqnarray}We also stress that one
can find a constant $k(\epsilon)$ satisfying
$$\int_{S}|\psi_\epsilon-k(\epsilon)|\per=\inf_{\beta\in\R}\int_S|\psi_\epsilon-\beta|\per$$and
such that $k(\epsilon)\longrightarrow 0$ for $\epsilon\rightarrow
0^+$. But then $$\lim_{\epsilon\rightarrow
0}\bigg\{\frac{\int_S|\qq \psi_\epsilon|\per}{\inf_{\beta\in
\R}\int_S|\psi_\epsilon-\beta|\per}\bigg\}\leq\frac{\nis(\partial
S_1\cap
\partial S_2)}{\min\{\per(S_1),\per(S_2)\}}.$$

\end{proof}

\begin{corollario} \label{1zdaz}Let $S\subset\GG$ be a $\cont^1$-smooth compact hypersurface. Then
\begin{eqnarray}\label{gay}\int_S|\psi-k|^2\per\leq \frac{4}{({\rm
Isop}_0(S))^2}\int_S|\qq\psi|^2\per
\end{eqnarray}for every $\psi\in\cont^1(S)$ and every $k\in\R$
such that\begin{eqnarray*}\per\{x\in S: \psi\geq k\}&\geq& \frac{1}{2}\per(S),\\
\per\{x\in S: \psi\leq k\}&\geq
&\frac{1}{2}\per(S).\end{eqnarray*}Furthermore, if
$\psi\in\cont^1(S)$ and $\int_S\psi\per=0$,
then\begin{eqnarray}\label{ga0}\int_S|\psi|^2\per\leq
\frac{4}{({\rm Isop}_0(S))^2}\int_S|\qq\psi|^2\per.
\end{eqnarray}
\end{corollario}

\begin{proof}One obviously has
$\int_S(\psi^+\cdot\psi^-)\,\per=0$, where the functions
$\psi^\pm$ are defined as in the proof of the previous Theorem
\ref{zdaz}. Moreover, by using once more the Coarea Formula, we
get

\begin{eqnarray*}\int_S|\psi-k|^2\per&=&\int_S|\psi^+ +\psi^-|^2\per\\&\leq&
\int_S|\psi^+|^2\per+\int_S|\psi^-|^2\per\\& \leq&\frac{1}{{\rm
Isop}_0(S)}\Big(\int_S|\qq(\psi^+)^2|\per+\int_S|\qq(\psi^-)^2|\per\Big)\\&\leq&
\frac{2}{{\rm Isop}_0(S)}\int_S(\psi^+
+\psi^-)|\qq\psi|\per\\&\leq&\frac{2}{{\rm
Isop}_0(S)}\|\psi^++\psi^-\|_{L^2(S; \per)}\|\qq\psi\|_{L^2(S;
\per)}.\end{eqnarray*}This proves inequality \eqref{gay}. To prove
\eqref{ga0} we note that the hypothesis $\int_S\psi\per=0$
actually implies that
\[\int_S\psi^2\per=\inf_{k\in\R}\int_S(\psi-k)^2\per.\] This,
together with \eqref{gay}, implies the claim.

\end{proof}

\section{Some Linear Isoperimetric inequalities}\label{wlinisoineq}
To the sake of simplicity, let us begin by considering the {\it
non-characteristic case}. So let $S\subset\GG$ be a
$\cont^2$-smooth hypersurface, with or without boundary, and let
$\UU\subset S$ be a non-characteristic relatively compact open set
with boundary $\partial\UU$, smooth enough for the validity of the
Riemannian Divergence Theorem. First of all, we shall apply
Corollary \ref{IPPH0} which claims that
\begin{equation}\label{jkjlkloko} \int_{\UU}\big\{(\DH-1)+ \g\MS +  \langle C\cc\nn,
\x\ss\rangle\big\}\per = \int_{\partial\UU}\langle
\x\cc,\eta\ss\rangle\,\nis.\end{equation} Hence

\begin{eqnarray}\label{bla}(\DH-1)\per(\UU)\leq \int_\UU\big\{|\g||\MS| +  |\langle C\cc\nn,
\x\ss\rangle|\big\}\per + \int_{\partial\UU}|\langle
\x\cc,\eta\ss\rangle|\nis,\end{eqnarray}where we remind that
$\g=\langle \x\cc,\nn\rangle$ is the {\it horizontal support
function} for $\UU$.

\begin{oss}\label{iponhomnor} Let $\varrho:\GG\times \GG\longrightarrow\R_+$
be any  homogeneous distance. In this section we shall assume that
the following conditions hold true:
\begin{itemize}\item[{\bf(i)}]$\varrho$ is - at least - piecewise $\cont^1$-smooth;\item[{\bf(ii)}]$|\grad\cc\varrho|\leq 1$ at each regular point of $\varrho$;
\item[{\bf(iii)}]$\frac{|x\cc|}{\varrho(x,0)}\leq 1$.
\end{itemize}
\end{oss}

Henceforth, we shall set $\varrho(x):=\varrho(x,0)=\|x\|_\varrho$.

\begin{es}\label{Kor}It can be proved that the  CC-distance $\dc$ satisfies all the previous assumptions; see \cite{MontArFer}.
Another example can be found for the case of the Heisenberg group
$\mathbb{H}^n$; see Section \ref{prelcar}. Indeed,  the Korany
norm, defined by
$$\|x\|_\varrho:=\varrho(x)=\sqrt[4]{|x\cc|^4+16t^2}\qquad(x=\exp(x\cc,t)\in\mathbb{H}^n)$$ is
homogeneous and
  $\cin$-smooth  out of the identity
$0\in\mathbb{H}^n$. Moreover, it turns out that $\varrho$
satisfies (ii) and (iii). The last example can easily be
generalized for those 2-step Carnot groups which satisfy the
following assumption on the structural constants:
$$C^\alpha\cc C^\beta\cd=-\mathbf{1}\ci\delta_{\alpha}^{\beta}\qquad (\alpha, \beta\in I\cd).$$Actually, in such
a case one can show that the homogeneous norm
$\|\cdot\|_\varrho$, defined by
$$\|x\|_\varrho:=\sqrt[4]{|x\cc|^4+16|x\cd|^2}\qquad (x=\exp(x\cc, x\cd)),$$satisfies
all the above requirements. We left to the reader the proof of
these claims which are based on elementary computations.
\end{es}
 Let now $R$ denote the radius of the $\varrho$-ball $B_\varrho(0,R)$, centered at the identity $0$ of $\GG$ and circumscribed
about $\UU$, and let us estimate the right-hand side of
\eqref{bla}. To this aim we stress that $\g\leq
|\x\cc|\leq\|x\|_\varrho$. We have
\begin{equation}\label{1pisoplin}(\DH-1)\per(\UU)\leq
R\Big\{\int_\UU\big(|\MS|+ |C\cc\nn|\big)\per+
\nis(\partial\UU)\Big\}.\end{equation}Obviously, if $\UU$ is
minimal, i.e. $\MS=0$, then it follows that

\begin{equation}\label{2isoplin}(\DH-1)\per(\UU)\leq
R\Big\{\int_\UU |C\cc\nn| \per+
\nis(\partial\UU)\Big\}.\end{equation}Furthermore, if
$\MS^0:=\max\{\MS(x)|x\in\UU\}$, one gets
\begin{equation}\label{3isoplin}\per(\UU)\big\{(\DH-1)-R\MS^0\big\}\leq
R\Big\{\int_\UU |C\cc\nn| \per+
\nis(\partial\UU)\Big\}.\end{equation}Equivalently, one
obtains\begin{equation}\label{4isoplin}R\geq\frac{(\DH-1)\per(\UU)}{\MS^0\per(\UU)+\big\{\int_\UU
|C\cc\nn| \per+ \nis(\partial\UU)\big\}}.
\end{equation}Clearly, by assuming $R\MS^0<\DH-1$, we also get that
\begin{equation}\label{5isoplin}\per(\UU)\leq
\frac{R\big\{\int_\UU |C\cc\nn| \per+
\nis(\partial\UU)\big\}}{(\DH-1)-R\MS^0}.
\end{equation}

We would like to remark that all the previous formulae have been
proved for the case of relatively compact non-characteristic
hypersurfaces with boundary. \\

 {\it Now we turn our attention
to the more general case of hypersurfaces having a possibly
non-empty characteristic set.} As in Remark \ref{cpoints},  let
$\UU_\epsilon\,(\epsilon>0)$ be a family of open subsets of $\UU$
with  piecewise $\cont^2$-smooth boundary such
that:\begin{itemize}\item[{\rm(i)}] $C_\UU\subset\UU_\epsilon$ for
every $\epsilon>0$; \item[{\rm(ii)}]
$\sigma^{n-1}\rr(\UU_\epsilon)\longrightarrow 0$ for
$\epsilon\rightarrow
0^+$;\item[{\rm(iii)}]$\int_{\UU_\epsilon}|\PH\nu|\sigma^{n-2}\rr\longrightarrow
0$ for $\epsilon\rightarrow 0^+$.\end{itemize}Thus, in particular,
$\per(\UU_\epsilon)\longrightarrow 0$ and
$\nis(\partial\UU_\epsilon)\longrightarrow 0$ for
$\epsilon\rightarrow 0^+$. Next, by applying Corollary \ref{IPPH0}
to $\UU\setminus\UU_\epsilon$ we get

\begin{equation}\label{scpas}\int_{\UU\setminus\UU_\epsilon}\big\{ (\DH-1)+ \g\MS +  \langle C\cc\nn,
\x\ss\rangle\big\}\per = \int_{\partial\UU}\langle
\x\cc,\eta\ss\rangle\,\nis+
 \int_{\partial\UU_\epsilon}
  \langle x\cc, \eta\ss\rangle\, \nis.\end{equation}Therefore, by
  letting $\epsilon\rightarrow 0^+$, {\it one
infers that \eqref{jkjlkloko} also holds true even for the case of
a non-empty characteristic set. Thus, it is clear that
\eqref{1pisoplin}, \eqref{2isoplin}, \eqref{3isoplin},
\eqref{4isoplin},
 \eqref{5isoplin} also hold for the
 characteristic case.}
Finally, we may consider the case of a smooth compact
 hypersurface $S$ without boundary. Note that, in this case, one cannot expect that $C_S=\emptyset$. Since the boundary term in
 \eqref{1pisoplin} vanishes, we still get that
 \begin{equation}\label{6isoplin}(\DH-1)\,\per(\UU)\leq
R\int_\UU\big(|\MS|+ |C\cc\nn|\big)\per.\end{equation}

\subsection{A weak monotonicity formula}\label{wmf}

As before, we shall set $\UU_t=\UU\cap {B_\varrho}(x,t)$. The
``natural'' monotonicity formula which can be deduced from the
linear inequality \eqref{1pisoplin} is contained in the next:

\begin{Prop}\label{in}The following inequality
holds\begin{equation}\label{diffin}-\frac{d}{dt}\frac{\per(\UU_t)}{t^{\DH-1}}\leq
\frac{1}{t^{\DH-1}}\bigg\{\int_{\UU_t}\big(|\MS|+|C\cc\nn|\big)\per
+ \nis(\partial\UU\cap B_\varrho(x,t))\bigg\}
\end{equation}for $\mathcal{L}^1$-a.e. $t>0$.
\end{Prop}\begin{proof}Since we are assuming that the homogeneous distance $\varrho$ is smooth (at least
piecewise $\cont^1$), by applying the classical Sard's Theorem we
get that $\UU_t$ is a $\cont^2$-smooth manifold with boundary for
$\mathcal{L}^1$-a.e. $t>0$. Equivalently, note that the previous
claim follows by intersecting $\UU$ with the boundary of a
$\varrho$-ball $B_{\varrho}(x,t)$ centered at $x$ and of radius
$t$. Indeed, by the regularity assumption on $\varrho$, it follows
that $\partial B_{\varrho}(x,t)$ is - at least - piecewise
$\cont^1$-smooth. Therefore, after noting that \eqref{jkjlkloko}
is invariant under left translations of the group, we may use them
for the set $\UU_t$\footnote{Remind that \eqref{jkjlkloko} can
also be used if $\UU_t$ has a non-empty characteristic set, as we
have proved in Section \ref{wlinisoineq}.}. So we have
\[(\DH-1)\per(\UU_t)\leq t\Big\{\int_{\UU_t}\big(|\MS|+
|C\cc\nn|\big)\per+ \nis(\partial\UU_t)\Big\},\]where $t$ is the
radius of a $\varrho$-ball centered at $x$ and intersecting $\UU$.
Since
$$\partial\UU_t=\{\partial\UU\cap B_\varrho(x,t)\}\cup\{\partial
B_\varrho(x,t)\cap \UU\}$$ we get
\begin{equation}\label{li}(\DH-1)\,\per(\UU_t)\leq
t\Big\{\underbrace{\int_{\UU_t}\big(|\MS|
+|C\cc\nn|\big)\per}_{:=\mathcal{A}(t)} +
\underbrace{\nis(\partial\UU\cap
B_\varrho(x,t))}_{:=\mathcal{B}(t)} +\nis(\partial
B_\varrho(x,t)\cap \UU)\Big\}.\end{equation}Now let us consider
the function $\psi(y):=\|y-x\|_\varrho\,(y\in \UU)$. By hypothesis
$\psi$ is a $\cont^1$-smooth function -at least piecewise-
satisfying $|\grad\cc\psi|\leq 1;$ see Remark \ref{iponhomnor}. So
we may apply the Coarea Formula to this function; see Section
\ref{COAR}. Since $|\qq\psi|\leq|\grad\cc\psi|$,  we easily get
that
\begin{eqnarray*}\per(\UU_{{t_1}})-\per(\UU_{t})&\geq&
\int_{\UU_{{t_1}}\setminus\UU_t}|\qq\psi|\per\\&=&\int_{t}^{{t_1}}\nis\{\psi^{-1}[s]\cap
\UU\}\,ds\\&=&\int_{t}^{t_1}\nis(\partial B_\varrho(x,s)\cap
\UU)\,ds.\end{eqnarray*}From the last inequality we infer that
$$\frac{d}{dt}\per(\UU_t)\geq\nis(\partial B_\varrho(x,t)\cap
\UU)$$ for every $t>0$ for $\mathcal{L}^1$-a.e. $t>0$. Hence, from
this inequality and \eqref{li}, we obtain
$$(\DH-1)\,\per(\UU_t)\leq t\Big\{ \mathcal{A}(t)+\mathcal{B}(t)+\frac{d}{dt}\per(\UU_t)\Big\},$$which is an equivalent form of
\eqref{diffin}.\end{proof}

We have to notice however that, in order to prove an intrinsic
isoperimetric inequality,
 the number $\DH-1$ in the previous differential inequality
 {\it is not the correct one}, which is $Q-1$.
 This fact motivates the results of Section \ref{mike}; see, in particular, Section \ref{wlineq}.

\subsection{Some weak inequalities on hypersurfaces}\label{wme}
With an elementary technique analogous to that used at Section
\ref{wlinisoineq}, we now state some local Poincar\'e-type
inequalities for smooth functions compactly supported on bounded
domains (i.e. open and connected) with ``small size''.
\begin{Defi}Let $S \subset\GG$ be a $\cont^2$-smooth hypersurface and let $\UU\subseteq S$ be open. We say that $\UU$ is {\bf
uniformly non-characteristic} if
$$\sup_{x\in\UU}|\varpi(x)|=\sup_{x\in\UU}\frac{|\P\vv\nu(x)|}{|\P\cc\nu(x)|}<+\infty.$$\end{Defi}
We stress that
\begin{eqnarray}\label{biotsav}|C\cc\nn|=\big|\sum_{\alpha\in I\vv}\omega_\alpha C^\alpha\cc\nn\big|\leq
\sum_{\alpha\in I\vv}|\omega_\alpha| \|C^\alpha\cc\|\ngr\leq
C|\varpi|,\end{eqnarray}where $C$ is a constant depending on the
structural constants of $\gg$. Moreover, let us set
$$R_\UU:=\frac{1}{2\big(\|\MS\|_{L^{\!\infty}(\UU)} + C\|\varpi\|_{L^{\!\infty}(\UU)}\big)
}.$$

\begin{teo}\label{0celafo}Let $\UU\subset S$ be a  $\cont^2$-smooth uniformly non-characteristic
open set with bounded horizontal mean curvature $\MS$. Then, for
all $x\in \UU$ and all $R\leq \min\{{\rm
dist}_\varrho(x,\partial\UU), R_\UU\},$ the following
holds\begin{eqnarray}\label{bibino}
\Big(\int_{\UU_R}|\psi|^p\per\Big)^{\frac{1}{p}}\leq C_p\,
R\Big(\int_{\UU_R}|\qq\psi|^p\per\Big)^{\frac{1}{p}} \qquad
p\in[1,+\infty[
\end{eqnarray}for every $\psi\in\cont_0^1(\UU_R)$. More generally, let $\widetilde{\UU}\subset\UU$
 be a bounded open subset of $\UU$ with smooth boundary such that ${\rm diam
}_\varrho(\widetilde{\UU})\leq 2 \min\{{\rm
dist}_\varrho(x,\partial\UU), R_\UU\}.$  Then
\begin{eqnarray}\label{bibino2}
\Big(\int_{\widetilde{\UU}}|\psi|^p\per\Big)^{\frac{1}{p}}\leq
C_p\, {\rm diam
}_\varrho(\widetilde{\UU})\Big(\int_{\widetilde{\UU}}|\qq\psi|^p\per\Big)^{\frac{1}{p}}
\qquad p\in[1,+\infty[
\end{eqnarray}for every $\psi\in\cont_0^1(\widetilde{\UU})$.
\end{teo}

In the above theorem one can take $C_p=\frac{2p }{2\DH-3}$.

\begin{proof}
 Let
us set
${\psi}_\varepsilon:=\sqrt{\varepsilon^2+\psi^2},\,(\varepsilon\geq
0)$. By applying Corollary \ref{IPPH} with $X={\psi}_\varepsilon
x\cc$ we get

\begin{equation*} \int_{\UU}\big\{{\psi}_\varepsilon\,\big( (\DH-1)+ \g\MS +  \langle C\cc\nn,
\x\ss\rangle\big)+
\langle\qq{\psi}_\varepsilon,\x\cc\rangle\big\}\per =
\int_{\partial\UU_R}{\psi}_\varepsilon\langle
x\cc,\eta\ss\rangle\nis,\end{equation*}and so
\begin{eqnarray*}(\DH-1)\int_{\UU_R}{\psi}_\varepsilon\per&\leq&
R\Big(\int_{\UU_R}\big\{{\psi}_\varepsilon\big(|\MS|+
|C\cc\nn|\big) +
|\qq{\psi}_\varepsilon|\big\}\per+\int_{\partial\UU_R}{\psi}_\varepsilon
\nis\Big)\\&\leq&R\bigg\{\big(\|\MS\|_\infty + C\|\varpi\|_\infty
\big)\int_{\UU_R}{\psi}_\varepsilon\per +
\Big(\int_{\UU_R}|\qq{\psi}_\varepsilon|\per
+\int_{\partial\UU_R}{\psi}_\varepsilon \nis\Big)\bigg\}.
\end{eqnarray*}By using Fatou's Lemma and the
estimate $R\leq R_\UU$ we get that

\begin{eqnarray*}(\DH-1)\int_{\UU_R}|\psi|\per&\leq&(\DH-1)\,
\liminf_{\varepsilon\rightarrow
0^+}\int_{\UU_R}{\psi}_\varepsilon\per
\\&\leq&\frac{1}{2}\lim_{\varepsilon\rightarrow
0^+}\int_{\UU_R}{\psi}_\varepsilon\per +
R\lim_{\varepsilon\rightarrow
0^+}\Big(\int_{\UU_R}|\qq{\psi}_\varepsilon|\per
+\int_{\partial\UU_R}{\psi}_\varepsilon \nis\Big).
\end{eqnarray*}Obviously $\psi_\varepsilon\longrightarrow|\psi|$ and
$|\qq\psi_\varepsilon|\longrightarrow |\qq\psi|$ as long as
$\varepsilon\rightarrow 0$; moreover $|\psi|=0$ along
${\partial\UU_R}$.  Now since, as it is well-known,
$|\qq|\psi||\leq|\qq\psi|,$ we easily  get the claim by Lebesgue's
Dominate Convergence Theorem. So we have shown that
$$\int_{\UU_R}|\psi|\per\leq \frac{2R}{2\DH-3}
\int_{\UU_R}|\qq\psi|\per$$ for every $\psi\in \cont_0^1(\UU_R)$.
Finally, the general case follows by H\"{o}lder's inequality. More
precisely, let us use the last inequality with $|\psi|$ replaced
by $|\psi|^p$. This implies
\begin{eqnarray*}\int_{\UU_R}|\psi|^p\per&\leq& \frac{2R}{(2\DH-3)}
\int_{\UU_R}p\,|\psi|^{p-1}|\qq\psi|\per\\&\leq&\frac{2pR}{(2\DH-3)}\,
\Big(\int_{\UU_R}|\psi|^{(p-1)q}\per\Big)^{\frac{1}{q}}\Big(\int_{\UU_R}|\qq\psi|^{p}\per\Big)^{\frac{1}{p}},
\end{eqnarray*}where $\frac{1}{p}+\frac{1}{q}=1$. This achieves the proof of \eqref{bibino}. Clearly, \eqref{bibino2} can be proved by repeating
the same arguments as above, just by replacing $R$ with ${\rm
diam}(\widetilde{\UU})$.

\end{proof}

With some extra hypotheses one can show that \eqref{bibino} still
holds on the characteristic set.

\begin{teo}\label{1celafo}Let $S\subset\GG$ be a  $\cont^2$-smooth hypersurface with bounded horizontal mean curvature $\MS$. Then, for
every $x\in S$ there exists $R_0\leq {\rm dist}_\varrho(x,\partial
S)$ such that
\begin{eqnarray}\label{bibino1}
\Big(\int_{S_R}|\psi|^p\per\Big)^{\frac{1}{p}}\leq C_p\,
R\Big(\int_{S_R}|\qq\psi|^p\per\Big)^{\frac{1}{p}} \qquad
p\in[1,+\infty[
\end{eqnarray}for every
$\psi\in\cont_0^1(S_R)$ and every $R\leq R_0$.
 More precisely, let $\UU_\epsilon$ denotes the family of open
neighborhoods of $C_S$ introduced in Remark \ref{cpoints}. Then,
for every $\epsilon>0$ there exists $R_0=R_0(\epsilon)$ such that
\eqref{bibino1} holds for every $\psi\in\cont_0^1(S_R)$ and every
$R\leq R_0$, where
$$R_0(\epsilon):=\min\bigg\{{\rm dist_\varrho}(x,\partial S),\,\,
\frac{1}{2\big[C\big(1+
\|\varpi\|_{L^{\!\infty}(S_R\setminus\UU_{\epsilon})} \big) +
\|\MS\|_{L^{\!\infty}(S_R)}\big]}\bigg\}.$$
\end{teo}

\begin{proof}Let us set
$\psi_\varepsilon:=\sqrt{\varepsilon^2+\psi^2}\,(0\leq\varepsilon<1).$
We shall first prove the theorem for $p=1$. The  general case will
follow by H\"{o}lder's inequality. As in Remark \ref{cpoints}, let
$\UU_\epsilon\,(\epsilon>0)$ be a family of open subsets of $S$
with  piecewise $\cont^2$-smooth boundary, such
that:\begin{itemize}\item[{\rm(i)}] $C_S\subset\UU_\epsilon$ for
every $\epsilon>0$; \item[{\rm(ii)}]
$\sigma^{n-1}\rr(\UU_\epsilon)\longrightarrow 0$ for
$\epsilon\rightarrow
0^+$;\item[{\rm(iii)}]$\int_{\UU_\epsilon}|\PH\nu|\sigma^{n-2}\rr\longrightarrow
0$ for $\epsilon\rightarrow 0^+$.\end{itemize}Fix $\epsilon_0>0$.
For every $\epsilon\leq \epsilon_0$ one has\footnote{Indeed
\eqref{biotsav} implies that $|C\cc\nn|\leq C\max_{\alpha\in
I\vv}|\varpi_\alpha|$. Moreover
$\int_B|\varpi_\alpha|\per=\int_B|\nu_\alpha|\sigma^{n-1}\rr\leq\sigma^{n-1}\rr(B)$
for every Borel set $B\subseteq S$.}
$$\int_{\UU_{\epsilon}}\psi_\varepsilon|C\cc\nn|\per\leq 2  C\,\|\psi\|_{L^{\!\infty}(\UU_{\epsilon_0})}\sigma^{n-1}\rr(\UU_\epsilon).$$
Furthermore (ii) implies that for every $\delta>0$ there exists
$\epsilon_\delta>0$ such that
$\sigma^{n-1}\rr(\UU_\epsilon)<\delta$ whenever
$\epsilon<\epsilon_\delta$. By putting $\widetilde{\delta}\leq
\frac{\int_{S_R}\psi_\varepsilon\per}{\|\psi\|_{L^{\!\infty}
(\UU_{\epsilon_{0}}) }},$ we therefore get that
$$\int_{\UU_{\epsilon}}\psi_\varepsilon|C\cc\nn|\per\leq
C\int_{S_R}\psi_\varepsilon\per$$ for every $\epsilon\leq \min\{
\epsilon_{\widetilde{\delta}}, \epsilon_0\}$, where $C$ only
depends on the structural constants of $\gg$. Moreover, for any
given $\epsilon\in]0,\min\{ \epsilon_{\widetilde{\delta}},
\epsilon_0\}[$, we have
 $$\int_{S_R\setminus\UU_{\epsilon}}\psi_\varepsilon|C\cc\nn|\per\leq C\|\varpi\|_{L^{\!\infty}
(S_R\setminus\UU_{\epsilon})}\int_{S_R}\psi_\varepsilon\per.$$
From what above, it follows that
$$\int_{S_R}\psi_\varepsilon|C\cc\nn|\per\leq C\big(1+ \|\varpi\|_{L^{\!\infty}
(S_R\setminus\UU_{\epsilon})}
\big)\int_{S_R}\psi_\varepsilon\per.$$ Since, by hypothesis, the
horizontal mean curvature is bounded along $S$ we have
$$\int_{S_R}\psi_\varepsilon|\MS|\per\leq \|\MS\|_{L^{\!\infty}
(S_R)}\int_{S_R}\psi_\varepsilon\per.$$ By applying Corollary
\ref{IPPH} with $X=\psi_\varepsilon x\cc$ one gets
\begin{eqnarray*}(\DH-1)\int_{S_R}{\psi}_\varepsilon\per&\leq&
R\Big(\int_{S_R}\big\{{\psi}_\varepsilon\big(|\MS|+ |C\cc\nn|\big)
+ |\qq{\psi}_\varepsilon|\big\}\per+\int_{\partial
S_R}{\psi}_\varepsilon \nis\Big)\\&\leq&R\big[C\big(1+
\|\varpi\|_{L^{\!\infty}(S_R\setminus\UU_{\epsilon})} \big) +
\|\MS\|_{L^{\!\infty}(S_R)}\big]\int_{S_R}\psi_\varepsilon\per \\
&+& R\Big(\int_{S_R}|\qq{\psi}_\varepsilon| \per+\int_{\partial
S_R}\psi_{\varepsilon}\nis\Big)\end{eqnarray*}and if $$R\leq
\min\Bigg\{{\rm dist_\varrho}(x,\partial S),\,\,
\frac{1}{2\big[C\big(1+
\|\varpi\|_{L^{\!\infty}(S_R\setminus\UU_{\epsilon})} \big) +
\|\MS\|_{L^{\!\infty}(S_R)}\big]}\Bigg\}$$one gets
$$\int_{S_R}{\psi}_\varepsilon\per\leq \frac{2R}{2\DH-3}\Big(\int_{S_R}|\qq\psi_\varepsilon|\per
+ \int_{\partial S_R}\psi_{\varepsilon}\nis\Big).$$ Clearly
$\psi_\varepsilon\longrightarrow|\psi|$ and
$|\qq\psi_\varepsilon|\longrightarrow |\qq\psi|$ as long as
$\varepsilon\rightarrow 0$ and  $|\psi|=0$ along ${\partial S_R}$.
Since $|\qq|\psi||\leq|\qq\psi|,$ the thesis follows by making use
of Fatou's Lemma and of Lebesgue's Dominate Convergence Theorem.
\end{proof}

\section{Isoperimetric Inequality on hypersurfaces}\label{mike}
Let $S\subset\GG$ be a $\cont^2$-smooth hypersurface with boundary
$\partial S$. As usual, $\varpi=\frac{\P\vv\nu}{|\PH\nu|}$.
Moreover, we shall consider the following
objects:$$\varpi\ci:=\P\ci\varpi=\sum_{\alpha\in
I\ci}\varpi_\alpha X_\alpha\quad i=2,...,k.$$ We shall also denote
by $\eta$ the unit normal vector $\eta$ along $\partial S$
\footnote{We stress that, at each point $x\in\partial S$,
$\eta(x)\in\TT_xS$}.

\begin{Defi}
In the sequel, we shall set
\begin{itemize}
\item[{\bf(i)}]$\chi:=\frac{\P\vs\eta}{|\P\ss\eta|}$;\item[{\bf(ii)}]$\chi\ciss:=\P\ciss\chi\qquad
(i=2,...,k);$
\end{itemize}see Remark \ref{indbun} for the very definitions of $\VS$ and $\HH_iS\,(i=2,...,k)$.\end{Defi}

\begin{Defi}\label{berlu}Let us fix a
point $x\in\GG$ and let us consider the Carnot homothety centered
at $x$, i.e. $\vartheta^x(t,y):=x\bullet\delta_t (x^{-1}\bullet
y)$. The variational vector field related to
$\vartheta^x_t(y):=\vartheta^x(t,y)$ at $t=1$ is given by
$$Z_x:=\frac{\partial \vartheta^x_t}{\partial
t}\Big|_{t=1}.$$\end{Defi}With no loss of generality, by using
group translations, we may put $x=0\in\GG$. In this case
$$\vartheta_t^0(y)=\esp(t\,y\cc,t^2y\cd,
t^3y\ctr,...,t^iy\ci,...,t^ky\ck)\qquad (t\in \R)$$ where
$y\ci=\sum_{j_i\in I\ci} y_{j_i}X_{j_i}$ and  $\esp$ denotes the
Carnot exponential mapping; see Section \ref{prelcar}. Thus it
turns out that $$Z_0=y\cc+ 2y\cd+...+ky\ck.$$\[\]

\begin{oss}\label{2iponhomnor1} As in Section \ref{wlinisoineq}, let $\varrho:\GG\times \GG\longrightarrow\R_+$
be a homogeneous distance such that:
\begin{itemize}\item[{\bf(i)}]$\varrho$ is - at least - piecewise $\cont^1$-smooth;\item[{\bf(ii)}]$|\grad\cc\varrho|\leq 1$ at each regular point of $\varrho$;
\item[{\bf(iii)}]$\frac{|x\cc|}{\varrho(x)}\leq 1$.
\end{itemize}Furthermore, we shall assume  that $c_i\in\R_+\,\,(i=2,...,k)$ are positive constants satisfying:
\begin{equation}\label{lis1am}|x\ci|\leq c_i
\varrho^i(x)\qquad(i=2,...,k),\end{equation}where we have set
$\varrho(x)=\varrho(0,x)=\|x\|_\varrho$.
\end{oss}

\[\]

\begin{Defi}Let $\GG$ be a $k$-step Carnot group and $S\subset\GG$ be a $\cont^2$-smooth
hypersurface with - at least - piecewise $\cont^1$-smooth boundary
$\partial S$. Let us set $S_R:=S\cap B_\varrho(x,R)$, where
$B_\varrho(x,R)$ is the open $\varrho$-ball centered at $x\in S$
and of radius $R>0$. We shall set
\begin{eqnarray*}\mathcal{A}_\infty(R)&:=&\int_{S_R}|\MS|\bigg(1+\sum_{i=2}^k
i\,c_i\varrho_x^{i-1}|\varpi\ci|\bigg)\per,\\\mathcal{B}(R)&:=&\int_{\partial
S\cap B_\varrho(x,R)}\frac{1}{\varrho_x}\Big|\Big\langle
Z_x,\frac{\eta}{|\P\ss\eta|}\Big\rangle\Big|\nis,\\\mathcal{B}_\infty(R)&:=&\int_{\partial
S\cap B_\varrho(x,R)}\Big(1+\sum_{i=2}^k
i\,c_i\varrho_x^{i-1}|\chi\ciss|\Big)\nis,\end{eqnarray*}where
$\varrho_x(y):=\varrho(x,y)\,(y\in S)$ denotes the
$\varrho$-distance from the ``fixed'' point $x\in S$. Furthermore,
we shall define the following integrals:
\begin{eqnarray*}\mathcal{A}_\infty(S)&:=&\int_{S}|\MS|\bigg(1+\sum_{i=2}^k
i\,c_i\varrho_S^{i-1}|\varpi\ci|\bigg)\per,\\\mathcal{B}(S)&:=&\int_{\partial
S}\frac{1}{\varrho_x}\Big|\Big\langle
Z_x,\frac{\eta}{|\P\ss\eta|}\Big\rangle\Big|\nis,\\\mathcal{B}_\infty(S)&:=&\int_{\partial
S}\Big(1+\sum_{i=2}^k i\,c_i\varrho_S^{i-1}|\chi\ciss|\Big)\nis,
\end{eqnarray*}where we have set$$\varrho_S:=\frac{\rm{diam}_\varrho(S)}{2}.$$
\end{Defi}
One immediately sees that $\mathcal{B}(R)\leq
\mathcal{B}_\infty(R)$ for every $R>0$ and that
$\mathcal{B}(S)\leq \mathcal{B}_\infty(S)$; see, for more details,
Section \ref{wlineq}.\\

We are now in a position to state the claimed isoperimetric
inequality.

\begin{teo}\label{isopineq}Let $\GG$ be a $k$-step Carnot group and $S\subset\GG$ be a $\cont^2$-smooth compact
hypersurface with $\cont^1$-smooth boundary $\partial S$.
boundary. Then\begin{equation}\label{isq0}\per(S)\leq
C_I\big\{\mathcal{A}_\infty(S)+
\mathcal{B}_\infty(S)\big\}^{\frac{Q-1}{Q-2}},
\end{equation}where $C_I$ is a constant independent of
$S$.
\end{teo}

The previous theorem generalizes to Carnot groups a well-known
result, proven by Michael and Simon \cite{MS} and, independently,
by Allard \cite{Allard}. The proof of this result will follow, as
much as possible,  the classical pattern which can be found, for
instance, in the nice book by Burago and Zalgaller \cite{BuZa}.
However, there are many non-trivial differences in doing this and
many Euclidean - and/or Riemannian - tools have been adapted to
our sub-Riemannian setting.

\begin{oss}\label{ghy}Let $S\subset\GG$ be a $\cont^2$-smooth hypersurface with $\cont^1$-smooth boundary
$\partial S$. Moreover, let us assume that $\dim\,C_{\partial S}<
n-2.$ This implies that $\sigma^{n-2}\rr(C_{\partial S})=0$. So
let $\mathcal{V}_\epsilon\subset\partial S$ be an open
neighborhood - in the relative topology of $\partial S$ - of
$C_{\partial S}$ such that
$\sigma^{n-2}\rr(\mathcal{V}_\epsilon)\longrightarrow 0$ for
$\epsilon\rightarrow 0^+$. Since
\[\chi=\sum_{i=2}^k\chi\ciss\in L^1(\partial S, \nis),\]it
follows that
$$\int_{\mathcal{V}_{\epsilon}}|\chi\ciss|\nis\leq
\nis(\partial S)$$ for all (small enough)
$\epsilon>0\,\,(i=2,...,k)$. Furthermore, one has
 $$\int_{\partial S\setminus\mathcal{V}_{\epsilon}}|\chi\ciss|\nis\leq \|\chi\|_{L^{\!\infty}
(\partial S\setminus\mathcal{V}_{\epsilon})}\,\nis(\partial
S)$$for all $\epsilon< \epsilon_{{\delta}}\,\,(i=2,...,k)$. Hence,
there exists a constant $C'>0$ such that
$$\mathcal{B}_\infty(S)=\int_{\partial S}\Big(1+\sum_{i=2}^k i\,c_i\varrho_S^{i-1}|\chi\ciss|\Big)\nis\leq
C'\,\nis({\partial S}).$$
\end{oss}

By applying Remark \ref{ghy} we immediately obtain the following:

\begin{corollario}\label{isopineqcaso2}
Let $\GG$ be a $k$-step Carnot group and $S\subset\GG$ be a
$\cont^2$-smooth compact hypersurface with $\cont^1$-smooth
boundary $\partial S$. Moreover, let us assume that
$\dim\,C_{\partial S}< n-2.$ Then there exists $C'_I>0$ such that
\begin{equation}\label{isq}\per(S)\leq
C'_I\big\{\mathcal{A}_\infty(S)+ \nis(\partial
S)\big\}^{\frac{Q-1}{Q-2}}.
\end{equation}\\
\end{corollario}

\subsection{
A strong linear inequality and its related monotonicity
formula}\label{wlineq} In this section we shall prove another
``linear'' isoperimetric inequality somewhat stronger than those
previously obtained in Section \ref{wlinisoineq}.

Some similar results have been independently obtained in a
interesting preprint by Danielli, Garofalo and Nhieu \cite{DGN3},
where a monotonicity estimate for the $\HH$-perimeter has been
proven for graphical strips in the Heisenberg group
$\mathbb{H}^1$.

Let $S\subset\GG$ be a $\cont^2$-smooth hypersurface  with
boundary $\partial{S}$ smooth enough for the validity of the
Riemannian Divergence Theorem.

We shall apply the 1st variation formula of $\per$, i.e. Theorem
\ref{1vg} (see also Corollary \ref{22vg}), with a ``special''
choice of the variational vector field. More precisely, let us fix
a point $x\in\GG$ and let us consider the Carnot homothety
centered at $x$, i.e. $\vartheta^x(t,y):=x\bullet\delta_t
(x^{-1}\bullet y)$. Without loss of generality, by using group
translations, we may choose $x=0\in\GG$. So
$$\vartheta^0(t,y):=\esp(ty\cc,t^2y\cd,
t^3y\ctr,...,t^iy\ci,...,t^ky\ck)\qquad (t\in \R)$$where
$y\ci=\sum_{j_i\in I\ci} y_{j_i}\ee_{j_i}$ and  $\esp$ denotes the
Carnot exponential mapping; see Section \ref{prelcar}. Thus the
variational vector field related to
$\vartheta^0_t(y):=\vartheta^0(t,y)=\delta_ty$, at $t=1$, is just
$$Z_0:=\frac{\partial \vartheta^0_t}{\partial
t}\Big|_{t=1}=\frac{\partial \delta_t}{\partial
t}\Big|_{t=1}=y\cc+ 2y\cd+...+ky\ck.$$ As it is well known, by
invariance of $\per$ under Carnot dilations, one
has$$\frac{d}{dt}\delta_t^\ast\per\Big|_{t=1}=(Q-1)\,\per({S}).$$Furthermore,
by using the 1st variation formula \eqref{fv} (see also Corollary
\ref{22vg}), one gets
\begin{equation}\label{prpas}(Q-1)\,\per({S})=-\int_{{S}}\MS\Big\langle Z_{0}, \frac{\nu}{|\PH\nn|}\Big\rangle\per +
\int_{\partial{S}} \Big\langle Z_{0},
\frac{\eta}{|\P\ss\eta|}\Big\rangle\,\underbrace{|\PH\nn|\cdot|\P\ss\eta|\sigma^{n-2}\rr}_{=\nis}.\end{equation}Note
that
$$\Big\langle Z_{0}, \frac{\nu}{|\PH\nn|}\Big\rangle=
\langle Z_{0}, (\nn+\varpi)\rangle=\langle
y\cc,\nn\rangle+\sum_{i=2}^k\langle
y\ci,\varpi\ci\rangle.$$Analogously, we have$$\Big\langle Z_{0},
\frac{\eta}{|\P\ss\eta|}\Big\rangle= \langle Z_{0},
(\eta\ss+\chi)\rangle=\langle
y\cc,\eta\ss\rangle+\sum_{i=2}^k\langle y\ci,\chi\ciss\rangle.$$
By Cauchy-Schwartz inequality we immediately get the following
estimates:\begin{itemize}\item $\Big|\Big\langle Z_{0},
\frac{\nu}{|\PH\nn|}\Big\rangle\Big|\leq |y\cc|+\sum_{i=2}^k
i\,|y\ci||\varpi\ci|$;\item$\Big|\Big\langle Z_{0},
\frac{\eta}{|\P\ss\eta|}\Big\rangle\Big|\leq |y\cc|+\sum_{i=2}^k
i\, |y\ci||\chi\ciss|.$ \end{itemize}Here above $|\cdot|$ denotes
the norm associated with the fixed Riemannian metric $\mathit{g}$
on $\GG$.\\According with Remark \ref{2iponhomnor1}, let
$c_i\in\R_+\,\,(i=1,...,k)$ be positive constants satisfying:
\begin{equation}\label{l1am}|y\ci|\leq c_i \varrho^i(y)\qquad(i=1,...,k).\end{equation}
The above argument implies that:\begin{itemize}\item
$\Big|\Big\langle Z_{0}, \frac{\nu}{|\PH\nn|}\Big\rangle\Big|\leq
\varrho\Big(\frac{|y\cc|}{\varrho}+\sum_{i=2}^k
i\,\frac{|y\ci|}{\varrho}|\varpi\ci|\Big)\leq
\varrho\Big(1+\sum_{i=2}^k i\,c_i\varrho^{i-1}|\varpi\ci|\Big)
$;\item $\Big|\Big\langle Z_{0},
\frac{\eta}{|\P\ss\eta|}\Big\rangle\Big|\leq
\varrho\Big(\frac{|y\cc|}{\varrho}+\sum_{i=2}^k
i\,\frac{|y\ci|}{\varrho}|\chi\ciss|\Big)\leq
\varrho\Big(1+\sum_{i=2}^k
i\,c_i\varrho^{i-1}|\chi\ciss|\Big).$\end{itemize}

\begin{Prop}\label{correctdimin}Let ${S}\subset\GG$ be a smooth compact hypersurface
 with  piecewise $\cont^1$-smooth boundary $\partial{S}$. Moreover, let
 $R$ denote the radius of a $\varrho$-ball centered at $x\in\GG$ and circumscribed about ${S}$. Then
\begin{equation}\label{1isopkklin}\per({S})\leq
\frac{R}{Q-1}\Bigg\{\int_{S}|\MS|\Big(1+\sum_{i=2}^k
i\,c_i\varrho_x^{i-1}|\varpi\ci|\Big)\per
+\int_{\partial{S}}\frac{1}{\varrho_x}\Big|\Big\langle Z_x,
\frac{\eta}{|\P\ss\eta|}\Big\rangle\Big|
 \nis\Bigg\}\end{equation}and
 \begin{equation}\label{1a}\per({S})\leq
\frac{R}{Q-1}\Bigg\{\int_{S}|\MS|\Big(1+\sum_{i=2}^k
i\,c_i\varrho_x^{i-1}|\varpi\ci|\Big)\per
+\int_{\partial{S}}\Big(1+\sum_{i=2}^k
i\,c_i\varrho_x^{i-1}|\chi\ciss|\Big)
\nis\Bigg\},\end{equation}where we have set
$\varrho_x(y)=\varrho(x,y)$.
\end{Prop}\begin{proof}Immediate by the previous discussion and by the invariance of $\per$ under left-translations.\end{proof}

In order to prove the claimed monotonicity inequality, we have to
 make some preliminary comments. Indeed, as will be clear in the
sequel, the monotonicity inequality (see Proposition \ref{rmonin})
could be proved, following a classical pattern, for instance by
assuming that:\begin{itemize}\item{\it there exists a smooth
homogeneous norm  $\varrho:\GG\times \GG\longrightarrow\R_+$ such
that:
\begin{equation}\label{0key}\frac{1}{\varrho_x}|\langle Z_x, \grad\ts\varrho_x\rangle|\leq 1
\end{equation}for every $x, y\in S$.}\end{itemize}

We stress that the above hypothesis, in a sense, would be
``natural'' in the Riemannian setting. In this regard, we would
like to quote the paper by Chung, Grigor'jan and Yau \cite{CGY},
where a similar hypothesis is one of the starting points of a
general theory about isoperimetric inequalities on weighted
Riemannian manifolds and graphs.

\begin{oss}Note that  \eqref{0key} turns out to be trivially true in the Euclidean setting. Indeed, in such a case, it
turns out that $Z_x(y)=y-x$ and
$\grad\,\varrho_x(y)=\frac{y-x}{|y-x|}$. Therefore,
\[\frac{1}{\varrho_x(y)}|\langle Z_x(y),
\grad\ts\varrho_x(y)\rangle|=1-\Big\langle\frac{y-x}{|y-x|},\textsl{n}_{\bf
e}\Big\rangle^2\leq 1,\]where $\textsl{n}_{\bf e}$ denotes the
Euclidean unit normal along $S$.
\end{oss}

\begin{es}Let us briefly discuss \eqref{0key} for
 the case of the Heisenberg group $\mathbb{H}^n$
endowed with the Korany norm; see Section \ref{prelcar} and
Example \ref{Kor}. In this case it turns out, after an explicit
computation, that
\begin{equation}\label{0kinawa}\grad\,\varrho_0=\frac{r^2\cc x\cc -4tx^\perp\cc +8t
T}{\varrho^3},\end{equation} and this implies that
$$\frac{1}{\varrho_0}\langle Z_0, \grad\,\varrho_0\rangle=1.$$Note
however that, from the last identity, we cannot conclude that
\eqref{0key} holds true independently of the hypersurface $S$. In
a sense \eqref{0key} is a kind of ``convexity condition'' about
the hypersurface $S$, w.r.t. the homogeneous metric $\varrho$. For
instance, it is very simple to show that \eqref{0key} holds true
in the following elementary cases:
\begin{itemize}\item $S\subset\mathbb{H}^n$ is any vertical hyperplane; \item
$S=\{\esp(x\cc,t)\in\mathbb{H}^n: t={\rm
const.}\}$.\end{itemize}\end{es}

We have to stress that the identity \eqref{0kinawa} in the
previous example, it is not  a mere coincidence, and it turns out
to be true for every (smooth enough) homogeneous distance
 on a Carnot group.

  Actually, we can prove the
following:

\begin{lemma}\label{cardcaz}Let $\GG$ be a $k$-step Carnot group and let
$\varrho:\GG\times\GG\longrightarrow\R_+$ be any $\cont^1$-smooth
homogeneous norm. Then
$$\frac{1}{\varrho_x}\langle Z_x, \grad\,\varrho_x\rangle=1$$for
every $x\in\GG$.

\end{lemma}\begin{proof}By homogeneity and left-invariance of $\varrho$. More
precisely, let us begin by noting that
$t\varrho(z)=\varrho(\delta_t z)$ for all $t>0$ and for every
$z\in\GG$. Setting $z:=x^{-1}\bullet y$, we get that
$$t\varrho(x, y)=t\varrho(z)=\varrho(\delta_t z)=\varrho(x, x\bullet\delta_t
z)$$for every $x, y \in\GG$ and for all $t>0$. Hence
\begin{eqnarray*}\varrho(x,y)=\frac{d}{dt}\varrho(x,
 x\bullet\delta_t x^{-1}\bullet y)\big|_{t=1}=\langle\grad\,\varrho_x(y),Z_x(y)
\rangle,
\end{eqnarray*}and this implies the claim.\end{proof}

Nevertheless, the assumption \eqref{0key} could be too strong in
some cases. To circumvent this inconvenience, by arguing locally,
we can prove the following

\begin{lemma}\label{00key} Let ${S}\subset\GG$ be a $\cont^2$-smooth
hypersurface and let $x\in S\setminus C_S$. Then there exists
$R(x)>0$ such that
\begin{equation}\label{000key}\int_{S_R}\frac{1}{\varrho_x}|\langle Z_x,
\grad\ts\varrho_x\rangle|\,\per\leq \per(S_R)
\end{equation}for every $R\leq R(x)$.
\end{lemma}\begin{proof}We may argue exactly as in Section \ref{blow-up}, Case (a),
for the case of non-characteristic points, by ``blowing-up'' the
hypersurface $S$ around $x\in S\setminus C_S$. By setting
$$\phi(y):=\frac{1}{\varrho_x(y)}|\langle Z_x(y),
\grad\ts\varrho_x(y)\rangle|$$and using the homogeneity of $\per$
under Carnot homothety $\vartheta^x_{R}(\cdot)$, we get that
\begin{eqnarray*}\int_{S_R}\phi\,\per=
R^{Q-1}\int_{\vartheta^x_{{1}/{R}}S\cap B_\varrho(x,
1)}\phi\circ\vartheta^x_R\,\per.
\end{eqnarray*}Now we remind that, at each
non-characteristic point $x\in S\setminus C_S$, one has
  $$\vartheta^x_{{1}/{R}}S\longrightarrow
\mathcal{I}(\nn(x))$$if $R\rightarrow 0^+$; see Section
\ref{blow-up}, Case (a). This implies that
\begin{itemize}\item $\langle Z_x,
\nu\rangle\longrightarrow \langle \PH(Z_x), \nn(x)\rangle$ \item
$\langle \grad\,\varrho_x, \nu\rangle\longrightarrow \langle
\PH(\grad\,\varrho_x), \nn(x)\rangle$
\end{itemize}at each point of
$\vartheta^x_{{1}/{R}}S\cap B_\varrho(x, 1)$ for $R\rightarrow
0^+.$ Therefore, it follows that $\frac{1}{\varrho_x}\langle Z_x,
\nu\rangle\leq 1$ at each point of $\vartheta^x_{{1}/{R}}S\cap
B_\varrho(x, 1)$ for $R\rightarrow 0^+$. Furthermore, by applying
Remark \ref{2iponhomnor1}, it follows that $\langle
\grad\,\varrho_x, \nu\rangle\leq 1$ at each point of
$\vartheta^x_{{1}/{R}}S\cap B_\varrho(x, 1)$ for $R\rightarrow
0^+$. Since, by Lemma \ref{cardcaz}, one has $\langle Z_x,
\grad\,\varrho_x \rangle=1$ everywhere, it follows that there
exists $R(x)>0$ such that $\phi\circ\vartheta^x_R\leq 1$ along
$\vartheta^x_{{1}/{R}}S\cap B_\varrho(x, 1)$ for every $R\leq
R(x)$. Hence \begin{eqnarray*}\int_{S_R}\phi\,\per=
R^{Q-1}\int_{\vartheta^x_{{1}/{R}}S\cap B_\varrho(x,
1)}\phi\circ\vartheta^x_R\,\per\leq
R^{Q-1}\int_{\vartheta^x_{{1}/{R}}S\cap B_\varrho(x,
1)}\per=\per(S_R)
\end{eqnarray*}for every $R\leq
R(x)$.
\end{proof}

 At this point, by means of
\eqref{1isopkklin} and of the previous Lemma \ref{00key}, we are
able to state and prove the claimed monotonicity formula for the
$\HH$-perimeter $\per$.

Henceforth, we shall set set ${S}_t:={S}\cap
{B_\varrho}(x,t)\,\,(t>0)$.

\begin{Prop}[Monotonicity of $\per$]\label{rmonin}  Let ${S}\subset\GG$ be a $\cont^2$-smooth
hypersurface. Then for every $x\in {\rm Int} S\setminus C_S$
there exists $R(x)>0$ such that the following differential
inequality
holds\begin{eqnarray}\label{rmytn}-\frac{d}{dt}\frac{\per({S}_t)}{t^{Q-1}}
\leq
\frac{1}{t^{Q-1}}\big\{\mathcal{A}_\infty(t)+\mathcal{B}_\infty(t)\big\}
\end{eqnarray}for $\mathcal{L}^1$-a.e. $t\in]0,R(x)[$.
\end{Prop}\begin{proof}We begin by arguing exactly as in the first
part of the proof of Proposition \ref{in}. More precisely, by
applying the classical Sard's Theorem we get that ${S}_t$ is a
$\cont^2$-smooth manifold with boundary for $\mathcal{L}^1$-a.e.
$t>0$. By using \eqref{1isopkklin} we obtain
\[(Q-1)\,\per({S}_t)\leq
t\Bigg\{\int_{{S}_t}|\MS|\bigg(1+\sum_{i=2}^k
i\,c_i\varrho_{x}^{i-1}|\varpi\ci|\bigg)\per+\int_{
\partial{S}_t} \frac{1}{\varrho_{x}}\Big|\Big\langle Z_{x}, \frac{\eta}{|\P\ss\eta|}\Big\rangle\Big|
 \nis\Bigg\}\]for $\mathcal{L}^1$-a.e.
$t>0$, where $t$ is the radius of a $\varrho$-ball centered at
$x\in {\rm Int}{S\setminus C_S}$. Since
$$\partial{S}_t=\{\partial{S}\cap B_\varrho(x,t)\}\cup\{\partial
B_\varrho(x,t)\cap {S}\}$$ we get that
\begin{eqnarray*}(Q-1)\per({S}_t)&\leq&
t\,\Bigg\{\underbrace{\int_{{S}_t}|\MS|\bigg(1+\sum_{i=2}^k
i\,c_i\varrho_{x}^{i-1}|\varpi\ci|\bigg)\per}_{:=\mathcal{A}_\infty(t)}
+ \underbrace{\int_{\partial{S}\cap B_\varrho(x,t)
}\frac{1}{\varrho_{x}}\Big|\Big\langle Z_{x},
\frac{\eta}{|\P\ss\eta|}\Big\rangle\Big|\nis
 }_{:=\mathcal{B}(t)}\\ \nonumber&&+\underbrace{\int_{\partial
B_\varrho(x,t)\cap {S}}\frac{1}{\varrho_{x}}\Big|\Big\langle
Z_{x},
\frac{\eta}{|\P\ss\eta|}\Big\rangle\Big|\nis}_{:=\widetilde{\mathcal{B}}(t)}\Bigg\}.\end{eqnarray*}Notice
that the term $\mathcal{B}(t)$ can be estimated as in the case of
the boundary term of the inequality \eqref{1a}; see Proposition
\ref{correctdimin}. In particular, one has $\mathcal{B}(t)\leq
\mathcal{B}_\infty(t)$. Now let us estimate the third integral
$\widetilde{\mathcal{B}}(t)$ by using Coarea Formula and  Lemma
\ref{00key}. We have

\begin{eqnarray*}\per({S}_{{t+h}})-\per({S}_{t})&\geq&
\int_{{S}_{{t+h}}\setminus{S}_t}\frac{1}{\varrho_{x}}|\langle
Z_{x},
\grad\ts\varrho_{x}\rangle|\per\\&=&\int_{t}^{{t+h}}ds\int_{\partial
B_\varrho(x,s)\cap {S}}\frac{1}{\varrho_{x}}\Big|\Big\langle
Z_{x},
\frac{\eta}{|\P\ss\eta|}\Big\rangle\Big|\nis,\end{eqnarray*}for
all (small enough) $h>0$. From the last inequality we infer that
$$\frac{d}{dt}\per({S}_t)\geq\int_{\partial
B_\varrho(x,s)\cap {S}}\frac{1}{\varrho_{x}}\Big|\Big\langle
Z_{x}, \frac{\eta}{|\P\ss\eta|}\Big\rangle\Big|\nis$$for
$\mathcal{L}^1$-a.e. $t\in]0, R(x)[$.  Hence
\[(Q-1)\,\per({S}_t)\leq t\Big(
\mathcal{A}_\infty(t)+\mathcal{B}_\infty(t)+\frac{d}{dt}\per({S}_t)\Big)\]that
is equivalent to the thesis.

\end{proof}

Both Lemma \ref{00key} and Proposition \ref{rmonin} could be
generalized also to the case where $x\in C_S$, i.e. $x$
 is a characteristic point. This could be done by using the same approach
 used for studying characteristic points at Section \ref{blow-up}, Case(b).
  Nevertheless, for brevity, we will not consider this problem in its general
 formulation, but only
for the case of the Heisenberg group endowed with the Korany norm.

\begin{lemma}\label{condcar77}Let $\mathbb{H}^n$ be the $2n+1$-dimensional Heisenberg group and let
$\varrho(x)=\sqrt[4]{|x\cc|^4 +16 t^2}$ be the Korany norm. Let
$S\subset\mathbb{H}^n$ be  a  $\cont^2$-smooth hypersurface having
non-empty characteristic set $C_S$. Then \eqref{000key} holds
locally around every $x\in {\rm Int} S\cap C_S$. Consequently, the
monotonicity inequality \eqref{rmytn} holds true.
\end{lemma}\begin{proof}Because of the invariance under left-translations, and without loss of generality,
 we may assume that $0\in  {\rm Int}
S\cap C_S$. Thus, locally around the identity $0\in\mathbb{H}^n$,
one easily infers that
\[S=\big\{x=\exp(x\cc,\psi(x\cc)): x\cc\in\R^{2n},\,\psi\in\cont^2(U_0),
\, \psi(0_{2n})=0,
\,\grad_{\R^{2n}}\psi(0_{2n})=0_{2n}\big\},\]for some open
neighborhood $U_0$ of $0_{2n}\in\R^{2n}$. Setting $A:={\rm
Hess}_{\R^{2n}}\psi(0_{2n})$ and using an ordinary 2nd order
Taylor's expansion for $\psi$ near $0_{2n}$, one gets, after
routine calculations, that $$\psi(x\cc)=\langle Ax\cc, x\cc\rangle
+ o(|x\cc|^2),\qquad(|x\cc|\rightarrow 0^+).$$ Therefore, using
the same notation and terminology as in Section \ref{blow-up},
Case(b), we may blow-up $S$ around $0\in\mathbb{H}^n$. It turns
out that the limit set $\Psi_\infty$ is given by
$$\Psi_\infty=\exp\{(x\cc, \langle Ax\cc, x\cc\rangle): x\cc\in
U_0\}.$$Hence, if we denote by $\nu$ the Riemannian unit normal
along $\Psi_\infty$, we find that
$$\nu(x)=\frac{\big(-Ax\cc+\frac{x^\perp\cc}{2}, 1\big)}{\sqrt{1+ |Ax\cc|^2 +
\frac{|x\cc|^2  }{4}}}.$$So we get that\[\langle Z_0,
\nu\rangle=\frac{\langle Ax\cc, x\cc\rangle}{\sqrt{1+ |Ax\cc|^2 +
\frac{|x\cc|^2  }{4}}}\]and that
\[\langle \grad\,\varrho, \nu\rangle=\frac{\langle Ax\cc,
x\cc\rangle(8+ o(|x\cc|))}{\varrho^3{\sqrt{1+ |Ax\cc|^2 +
\frac{|x\cc|^2  }{4}}}}\sim\frac{8}{\varrho^3}\langle Ax\cc,
x\cc\rangle\]for $|x\cc|\rightarrow 0^+$. From what above one
infers that
$$\frac{1}{\varrho_x}\langle Z_x, \grad\ts\varrho_x\rangle\big|_{\Psi_\infty}=1-\frac{8\langle A\xi\cc,
\xi\cc\rangle^2}{(1+16\langle A\xi\cc, \xi\cc\rangle^2)(1+
o(|x\cc|))}\leq1$$for $|x\cc|\rightarrow 0^+$, where we have set
$\xi\cc:=\frac{x\cc}{|x\cc|}$. By arguing as in Lemma \ref{00key},
the claim follows.
\end{proof}

\subsection{Proof of Theorem \ref{isopineq}}\label{isopineq1}

\begin{lemma}\label{lem}Let $x\in {\rm Int}\,S\setminus C_S$ and set \footnote{The quantity $R(x)$ has been introduced at Lemma \ref{rmonin}.}
$R_0(x):=\min\Big\{R(x),
2\Big\{\frac{\per({S})}{k_\varrho(\nn(x))}\Big\}^{{1}/{Q-1}}\Big\}$.
Then, for every $\lambda\geq 2$ there exists $R\in ]0, R_0(x)[$
such that\begin{eqnarray*}\label{condlem}\per({S}_{\lambda R})\leq
\lambda^{Q-1}R_0(x)\big\{\mathcal{A}_\infty(R)+\mathcal{B}_\infty(R)\big\}.\end{eqnarray*}
\end{lemma}

\begin{oss}\label{boundonmetricfactor}We remind that the metric-factor $k_\varrho(\nn)$
turns out to be constant, for instance, by assuming that $\varrho$
be symmetric on all layers; see, for instance, \cite{Mag3}.
Anyway, the metric factor is uniformly bounded by two positive
constants $k_1$ and $k_2$. This can be easily  deduced by making
use of the so-called {\bf ball-box metric}\footnote{By definition
one has
$${\rm Box}(x,R)=\big\{y=\exp\big(\sum_{i=1,...,k} y\ci\big)\in\GG\,|\,
\|y\ci-x\ci\|_\infty\leq R^i\big\},$$where $y\ci=\sum_{j_i\in
I\ci}y_{j_i}\ee_{j_i}$ and $\|y\ci\|_\infty$ is the sup-norm on
the $i$-th layer of $\gg$; see, for instance, \cite{Gr1},
\cite{Montgomery}.} and by a homogeneity argument. Indeed, for any
given $\varrho$-ball $B_\varrho(x,R)$, there exist two boxes ${\rm
Box}(x,R_1),\,{\rm Box}(x,R_2)\,\,\,(R_1\leq R\leq R_2)$ such that
$${\rm Box}(x,R_1)\subseteq B_\varrho(x,R)\subseteq{\rm
Box}(x,R_2).$$Remind that
$$k_\varrho(\nn(x))=\per(\mathcal{I}(\nn(x))\cap B_\varrho(x,1))=\Ar(\mathcal{I}(\nn(x))\cap B_\varrho(x,1)).$$
Now let us choose $R_1, R_2$ such that $0<R_1\leq 1\leq R_2$ and
${\rm Box}(x,R_1)\subseteq B_\varrho(x,1)\subseteq{\rm
Box}(x,R_2).$ Since $\delta_{t}{\rm Box}(x,1/2)={\rm Box}(x,t/2)$
for every $t\geq 0$, by a simple computation\footnote{The unit box
${\rm Box}(x,1/2)$ is the left-translated at the point $x$ of that
centered at the identity of $\GG$, i.e. ${\rm Box}(0,1/2)$. So, by
the left-invariance of the measure $\per$, the computation can be
done at $0\in\GG$. By observing that ${\rm Box}(0,1/2)$ (or, more
precisely, $\llog({\rm Box}(0,1/2))$) can be regarded just as the
unit hypercube of $\R^n\cong\gg$, it remains to show how we can
estimate the $\per$-measure (remind that on vertical hyperplanes,
one has $\per=\Ar$)  of the intersection of the unit box ${\rm
Box}(0,1/2)$ with a ``generic'' vertical hyperplane through the
origin $0\in\R^n\cong\gg$.  But this is a very simple matter.
Indeed, if $\mathcal{I}(X)$ denotes the vertical hyperplane
through the origin of $\Rn$ and orthogonal to the (unit)
horizontal vector $X\in\HH$, we get that
$$1\leq\Ar({\rm Box}(0,1/2)\cap\mathcal{I}(X))\leq\sqrt{n-1},$$where we notice that $\sqrt{n-1}$ is just the diameter of any
face of the unit hypercube of $\Rn$. Therefore
\begin{eqnarray*}\big(\delta_{2R_1}{\rm
Box}(0,1/2)\cap\mathcal{I}(X)\big)\subseteq
\big(B_\varrho(0,1)\cap\mathcal{I}(X)\big)\subseteq
\big(\delta_{2R_2}{\rm Box}(x,1/2)\cap\mathcal{I}(X)\big)
\end{eqnarray*}and so\begin{eqnarray*}{(2R_1)}^{Q-1}&\leq&{(2R_1)}^{Q-1}\Ar({\rm Box}(0,1/2)\cap\mathcal{I}(X))
\leq
\Ar(B_\varrho(0,1)\cap\mathcal{I}(X))\\&=&\kappa_\varrho(X)\leq{(2R_2)}^{Q-1}\Ar({\rm
Box}(0,1/2)\cap\mathcal{I}(X))\leq\sqrt{n-1}{(2R_2)}^{Q-1}.\end{eqnarray*}}we
get that
$$(2R_1)^{Q-1}\leq k_\varrho(\nn(x))\leq
\sqrt{n-1}\,(2R_2)^{Q-1}.$$In particular, we may put
$k_1=(2R_1)^{Q-1}$ and $k_2=\sqrt{n-1}\,{(2R_2)}^{Q-1}$.\end{oss}

We stress that, as an application of Remark
\ref{boundonmetricfactor}, for every homogeneous metric $\varrho$
, one can choose two positive constants $k_1,\,k_2$ such that
\begin{eqnarray}\label{emfac}k_1\leq \kappa_\varrho(\nn(x))\leq
k_2\end{eqnarray}for every $x\in {S}\setminus C_{S}$.

\begin{no}Henceforth, we shall set
$R_0({S}):=\sup_{x\in {S}\setminus C_{S}}R_0(x).$
\end{no}

\begin{proof}[Proof of Lemma \ref{lem}]Let $R\in]0,R_0(x)[$. Note that $\per({S}_t)$ is a monotone non-decreasing function of $t$ on
$[R,R_0]$. So let us write the identity
$$\per({S}_t)/t^{Q-1}=\big\{\per({S}_t)/t^{Q-1}-\per({S}_{R_0(x)})\big\}/t^{Q-1}+\per({S}_{R_0(x)})/t^{Q-1}.$$The first addend is an  increasing  function of $t$, while
the second one is an absolutely continuous function of $t$.
Therefore, by integrating the differential inequality
\eqref{rmytn}, we get that

\begin{equation}\label{ppp1}\frac{\per({S}_{R})}{R^{Q-1}}
\leq\frac{\per({S}_{R_0(x)})}{R_0(x)^{Q-1}}+\int_0^{R_0(x)}{\big\{\mathcal{A}_\infty(t)+\mathcal{B}_\infty(t)
\big\}}{t^{-(Q-1)}}dt.\end{equation}Therefore

\begin{eqnarray}\label{ppp2}\beta&:=&\sup_{R\in[0,R_0(x)]}\frac{\per({S}_{R})}{R^{Q-1}}
\leq\frac{\per({S}_{R_0(x)})}{R_0(x)^{Q-1}}+\int_0^{R_0(x)}{\big\{\big(\mathcal{A}_\infty(t)+\mathcal{B}_\infty(t)
\big\}}{t^{-(Q-1)}}dt.\end{eqnarray} Now we argue by
contradiction. If the lemma is false, it follows that for every
$R\in]0,R_0(x)[$\begin{equation}\label{negcondlem}\per({S}_{\lambda
R})>\lambda^{Q-1}R_0(x)\big\{\mathcal{A}_\infty(R)+\mathcal{B}_\infty(R)\big\}.\end{equation}From
the last inequality we infer that
\begin{eqnarray*}&&\int_0^{R_0(x)}\big\{\mathcal{A}_\infty(t)+\mathcal{B}_\infty(t)\big\}{t^{-(Q-1)}}dt\\&\leq&\frac{1}{\lambda^{Q-1}R_0(x)}\int_{0}^{R_0(x)}\per({S}_{\lambda
t})t^{-(Q-1)}dt\\&=&\frac{1}{\lambda R_0(x)}\int_{0}^{\lambda
R_0(x)}\per({S}_{s})s^{-(Q-1)}ds\\&=&\frac{1}{\lambda
R_0(x)}\int_{0}^{\lambda R_0(x)}\per({S}_{s})s^{-(Q-1)}ds+
\frac{1}{\lambda R_0(x)}\int_{R_0(x)}^{\lambda
R_0(x)}\per({S}_{s})s^{-(Q-1)}ds\\&\leq
&\frac{\beta}{\lambda}+\frac{\lambda-1}{\lambda}\frac{\per({S})}{R_0(x)^{Q-1}}.\end{eqnarray*}By
using \eqref{ppp1} we get that
$$\beta\leq\frac{\per({S}_{R_0(x)})}{R_0(x)^{Q-1}}+
 \frac{\beta}{\lambda}+\frac{\lambda-1}{\lambda}\frac{\per({S})}{R_0(x)^{Q-1}}$$and so
$$\frac{\lambda-1}{\lambda}\beta\leq\frac{2\lambda-1}{\lambda}\frac{\per({S})}{R_0(x)^{Q-1}}=\frac{2\lambda-1}
{\lambda}\frac{k_\varrho(\nn(x))}{2^{Q-1}}.$$By its own
definition, one has
$$k_\varrho(\nn(x))=\lim_{R\searrow
0^+}\frac{\per({S}_R)}{R^{Q-1}}\leq\beta.$$Furthermore,
since\footnote{Indeed, the first non-abelian Carnot group is the
Heisenberg group $\mathbb{H}^1$ for which one has $Q=4$. Moreover,
note that, since the theory of Carnot groups also contains, as a
special case, the theory of Euclidean spaces, in the above
argument we can use the estimate $Q-1\geq 2$ which is relative to
the case of a 2-dimensional surface in $\R^3$; indeed, in such a
case $Q=3$, since the homogeneous dimension coincides with the
topological one.} $Q-1\geq 3$, we get that$$\lambda-1\leq
\frac{2\lambda-1}{8},$$or equivalently $\lambda\leq\frac{7}{6}$,
which contradicts the hypothesis $\lambda\geq 2$.

\end{proof}

 The next covering lemma is well-known and can be found in \cite{BuZa}; see also \cite{FE}.
\begin{lemma}[Vitali's Covering Lemma]\label{cov}Let $(X,d)$ be a compact metric space and
let $A\subset X$. Moreover, let $\mathcal{C}$ be a covering of $A$
by closed $d$-balls with centers in $A$. We also assume that each
point $x$ of $A$ is the center of at least one closed $d$-ball
belonging to $\mathcal{C}$ and that the radii of the balls of the
covering are uniformly bounded by some positive constant. Then,
for every $\lambda> 2$ there exists a no more than countable
subset $\mathcal{C}'\subset\mathcal{C}$ of pairwise
non-intersecting closed balls $\overline{B}(x_k,R_k)$ such that
$$A\subset\bigcup_k{B}(x_k,\lambda R_k).$$

\end{lemma}

We are now in a position to prove our main result.

\begin{proof}[Proof of Theorem \ref{isopineq}] Fist we shall apply Lemma \ref{lem}. To this aim,  let
$\lambda>2$ be fixed and, for every $x\in{\rm Int}{S}\setminus
C_{{S}}$, let $R(x)\in]0, R_0(S)[$ be such
that$$\per({S}_{R(x)})\leq\lambda^{Q-1}R_0({S})
\big\{\mathcal{A}_\infty(R(x))+\mathcal{B}_\infty(R(x))\big\}.$$Now
let us consider the
 covering $\mathcal{C}=\{\overline{B_\varrho}(x,R(x))\}$ of the (relatively compact) set
 ${S}\setminus C_{{S}}\subset\GG$. By Lemma \ref{cov}, there
 exists a sub-covering $\mathcal{C}'\subset\mathcal{C}$ of pairwise
non-intersecting closed balls $\overline{B}_\varrho(x_k,R(x_k))$,
such that
$${S}\setminus C_{S}\subset\bigcup_k{B}_\varrho(x_k,\lambda R(x_k)).$$ This
implies that \begin{eqnarray*}\per({S})&\leq&\sum_{k}\per({S}\cap
B_\varrho(x_k,\lambda R(x_k)))\\&\leq&
\lambda^{Q-1}R_0({S})\sum_k\big\{\mathcal{A}_\infty(R(x_k))+
\mathcal{B}_\infty(R(x_k))\big\}\\&\leq&\lambda^{Q-1}R_0({S})\big\{\mathcal{A}_\infty({S})+\mathcal{B}_\infty({S})\big\}.
\end{eqnarray*}By letting $\lambda\searrow 2$ we get that
\begin{eqnarray}\label{gha}\per({S})\leq
2^{Q-1}R_0({S})\big\{\mathcal{A}_\infty({S})+\mathcal{B}_\infty({S})\big\}.\end{eqnarray}
Since$$2^{Q-1}R_0({S})\leq 2^{Q-1}\sup_{x\in {S}\setminus C_{S}}2
\Big\{\frac{\per({S})}{k_\varrho(\nn(x))}\Big\}^{\frac{1}{Q-1}}=2^Q\sup_{x\in
{S}\setminus
C_{S}}\frac{\{\per({S})\}^{\frac{1}{Q-1}}}{\{k_\varrho(\nn(x))\}^{\frac{1}{Q-1}}},$$
by making use of \eqref{emfac} one gets that
$$2^{Q-1}R_0({S})\leq 2^Q
\frac{\{\per({S})\}^{\frac{1}{Q-1}}}{k_1^{\frac{1}{Q-1}}},$$and so
\begin{eqnarray}\label{gjha}\{\per({S})\}^{\frac{Q-2}{Q-1}}\leq
\frac{2^{Q}}{k_1^{\frac{1}{Q-1}}}\big\{\mathcal{A}_\infty({S})+\mathcal{B}_\infty({S})\big\}.\end{eqnarray}Obviously,
the last inequality is equivalent to \eqref{isq} by setting
$C_I:=2^{\frac{Q\cdot Q-1}{Q-2}}\cdot k_1^{\frac{1}{2-Q}},$ where
$k_1$ has been defined in Remark \ref{boundonmetricfactor}. Note
that, by construction, $k_1$ only depends on the homogeneous
metric $\varrho$; see Remark \ref{boundonmetricfactor}. Moreover,
for a constant metric factor $\kappa_\varrho(\nn)$, one may choose
$k_1=\kappa_\varrho(\nn)$. Anyway the constant $C_I$ only depends
on the group $\GG$ and on the homogeneous metric $\varrho$. This
completes the proof of \eqref{isq0}.

\end{proof}

\[\]

\subsection{An application of the monotonicity formula: asymptotic
behavior of $\per$}\label{asintper}  The monotonicity formula
\eqref{rmytn} can equivalently be formulated as
follows:$$\frac{d}{dt}\Bigg\{\frac{\per({S}_t)}{t^{Q-1}}\,\exsp
\bigg(\int_0^t\frac{\mathcal{A}_\infty(s)+\mathcal{B}_\infty(s)}{\per({S}_s)}ds\bigg)\Bigg\}\geq
0$$ for $\mathcal{L}^1$-a.e. $t\in]0,R(x)[$ and for every $x\in
S\setminus C_S$. By Theorem \ref{BUP}, Case (a), we may pass to
the limit as $t\searrow 0^+$ in the previous inequality; see
Section \ref{blow-up}. Hence
\begin{eqnarray*}&&\per({S}_t)\geq
\kappa_{\varrho}(\nn(x))\,t^{Q-1}\exsp
\bigg(-\int_0^t\frac{\mathcal{A}_\infty(s)+\mathcal{B}_\infty(s)}{\per({S}_s)}ds\bigg),\end{eqnarray*}for
every $x\in{\rm Int}({S})$ satisfying the hypotheses of Theorem
\ref{BUP}.

\begin{corollario}\label{asynt}Under the above hypotheses and notation, let us assume that $$\partial{S}\cap
B_\varrho(x,t)={\emptyset}$$ and that $|\MS|\leq\MS^0<+\infty$.
Then, for every $x\in{S}\setminus C_{S}$, one has
\begin{eqnarray}\label{asintotica}
\per({S}_t)\geq \kappa_{\varrho}(\nn(x))\,t^{Q-1} e^{-t
\{\MS^0(1+O(t))\}}.\end{eqnarray}
\end{corollario}
\begin{proof}Since, by hypothesis, $\mathcal{B}_\infty(s)=0,$ we
just have to bound $\int_0^t
\frac{\mathcal{A}_\infty(s)}{\per({S}_s)}\,ds$ from above. To this
aim, let us note that, locally around $x\in S\setminus C_S$, one
has $|\varpi\ci|\leq\varepsilon_i$ for some positive constants
$\varepsilon_i\,(i=2,...,k)$. Moreover, one has
$\varrho_S(S_s)=s$. Therefore
$$\mathcal{A}_\infty(s)\leq \MS^0\bigg\{1+\sum_{i=2}^k
\varepsilon_i s^{i-1}\bigg\}\per(S_s)$$and the claim easily follows.\\
 \end{proof}

More generally, the previous result could be generalized, for
instance, to those characteristic points satisfying the hypotheses
 of Theorem \ref{BUP}. Later on, we shall analyze in detail only the case of the Heisenberg group
$\mathbb{H}^n$, because in such a case the monotonicity formula
also holds near characteristic points; see Lemma \ref{condcar77}.

Before turning our attention to the Heisenberg group case, let us
discuss the following:

\begin{oss}\label{r1sd}Let $S\subset\GG$ be a $\cont^2$-smooth hypersurface
and let $x\in {\rm Int} S$. As above, let us set $S_s=S\cap
B_{\varrho}(x, s)\, (s>0)$. Then, we claim that
\begin{eqnarray}\label{carosio}
\frac{\int_{{S}_s}|\varpi\ci|\per}{\per({S}_s)}\leq \varepsilon_i
\frac{s^{Q-i}}{{s^{Q-1}}},\end{eqnarray}for some constants
$\varepsilon_i>0,\,\,i=2,...,k.$ To see this, let us note that
$X_\alpha\LL\sigma^n\rr=\ast\omega_\alpha$, where $\ast$ denotes
the Hodge star operation on $\TT^\ast\GG$. Furthermore, one has
$\delta_t^\ast(\ast\omega_\alpha)=t^{Q-{\rm
ord}(\alpha)}(\ast\omega_\alpha)$ for every $t>0$. Therefore,
\eqref{carosio} easily follows since
$${\int_{{S}_s}|\varpi\ci|\per}\leq \sum_{{\rm ord}(\alpha)=i}\int_{S_s} |X_\alpha\LL\sigma^n\rr|=
\sum_{{\rm ord}(\alpha)=i}s^{Q-i}\int_{\vartheta^x_{{1}/{s}}S\cap
B_\varrho(x, 1)}(\ast\omega_\alpha)\circ\vartheta^x_s. $$
\end{oss}

\begin{corollario}\label{hasynt}Let $\mathbb{H}^n$ be the $2n+1$-dimensional Heisenberg group and let
$\varrho(x)=\sqrt[4]{|x\cc|^4 +16 t^2}$ be the Korany norm. Let
$S\subset\mathbb{H}^n$ be  a  $\cont^2$-smooth hypersurface having
non-empty characteristic set. Furthermore, let us assume that
$\partial{S}\cap B_\varrho(x,t)={\emptyset}$ and that
$|\MS|\leq\MS^0<+\infty$. Then, for every $x\in{\rm Int}{S}\cap
C_S$, there exists $\varepsilon_0>0$ such that
\begin{eqnarray}\label{2asintotica}
\perh({S}_t)\geq \kappa_{\varrho}(C_S(x))\,t^{Q-1} e^{-t \MS^0
\varepsilon_0},\end{eqnarray}where $\kappa_{\varrho}(C_S(x))$ has
been defined in Theorem \ref{BUP}, Case (b).\end{corollario}
\begin{proof}By applying Remark \ref{r1sd} we easily get that
\begin{eqnarray*}\frac{\mathcal{A}_\infty(s)}{\perh({S}_s)}\leq
\MS^0\Big\{1+\sum_{i=2}^k\varepsilon_i\Big\}.
\end{eqnarray*}Setting
$\varepsilon_0:=1+\sum_{i=2}^k\varepsilon_i$ now yields the
desired result.

\end{proof}

\section{ Sobolev-type inequalities on hypersurfaces}\label{sobineqg}

The isoperimetric inequality \eqref{isq0} is actually equivalent
to an inequality of Sobolev-type. The proof of this fact is
completely analogous to that of the well-known classical
equivalence between the Euclidean Isoperimetric Inequality and the
Sobolev one; see \cite{BuZa}, \cite{Ch2}.
\begin{teo}\label{sobolev}Let $S\subset\GG$ be a
$\cont^2$-smooth hypersurface without boundary. Then for every
$\psi\in\cont^1_0(S)$ the following inequality
holds\begin{eqnarray*}\bigg\{\int_S|\psi|^{\frac{Q-1}{Q-2}}\per\bigg\}^{\frac{Q-2}{Q-1}}&\leq&
C_I\bigg\{\int_{S}|\psi||\MS|\Big(1+\sum_{i=2}^k
i\,c_i\varrho_{S}^{i-1}|\varpi\ci|\Big)\per\\&&+\int_{S}\Big(|\qq\psi|+\sum_{i=2}^k
i\,c_i\varrho_{S}^{i-1}|\grad\ciss\psi|\Big)\per\bigg\}
\end{eqnarray*} where $C_I$ is a dimensional constant independent
of $S$.
\end{teo}
The constant $C_I$ is the same as in Theorem \ref{isopineq}.
\begin{proof}The proof follows a classical argument;
see \cite{FedererFleming}, \cite{MAZ}. Let us begin by noting
that, since $|\qq\psi|\leq|\qq|\psi||$, without loss of
generality, we may assume that $\psi\geq 0$. Moreover let us set
$$S_t:=\{x\in S: \psi(x)>t\}.$$Since $\psi$ has compact support,
the set $S_t$ is a bounded open subset of $S$ and, by applying
Sard's Lemma, one sees that its boundary $\partial S_t$ is smooth
for $\mathcal{L}^1$-a.e. $t\in\R_+\cup\{0\}$. Furthermore,
$S_t={\emptyset}$ for each (large enough) $t\in\R_+$. The main
tools that we are going to use are Cavalieri's principle and the
Coarea Formula; see Footnote \ref{CavPrin} and Theorem
\ref{1coar}. In particular, we stress that
$$\int_S|\psi|^{\frac{Q-1}{Q-2}}\per=\frac{Q-1}{Q-2}\int_0^{+\infty}t^{\frac{1}{Q-2}}\per\{S_t\}\,dt.$$We also
remind that, if
$\varphi:\R_+\cup\{0\}\longrightarrow\R_+\cup\{0\}$ is a positive
{\it decreasing} function and $\alpha\geq 1$, then the following
inequality
holds:$$\alpha\int_0^{+\infty}t^{\alpha-1}\varphi^\alpha\,dt\leq\Big(\int_0^{+\infty}\varphi(t)\,dt\Big)^\alpha.$$
So let us use this inequality with $\alpha=\frac{Q-1}{Q-2}$ and
$\varphi(t)=(\per\{S_t\})^{\frac{Q-2}{Q-1}}$. One has
\begin{eqnarray*}\int_S\psi^{\frac{Q-1}{Q-2}}\per&=&\frac{Q-1}{Q-2}\int_0^{+\infty}t^{\frac{Q-1}{Q-2}}(\per\{S_t\})
\,dt\\ &\leq&
\Big(\int_0^{+\infty}(\per\{S_t\})^{\frac{Q-2}{Q-1}}\,dt\Big)^{\frac{Q-1}{Q-2}}\\&\leq&
\bigg(\int_0^{+\infty}C_I^{\frac{Q-2}{Q-1}}\big\{\mathcal{A}_\infty(S_t)+\mathcal{B}_\infty(S_t)\big\}\,dt\bigg)
^{\frac{Q-1}{Q-2}}\\&=& C_I
\bigg(\int_0^{+\infty}\big\{\mathcal{A}_\infty(S_t)+\mathcal{B}_\infty(S_t)\big\}\,dt\bigg)
^{\frac{Q-1}{Q-2}}.\end{eqnarray*}At this point let us note
that\begin{eqnarray*}\int_0^{+\infty}\mathcal{A}_\infty(S_t)\,dt&=&\int_0^{+\infty}\bigg\{\int_{S_t}|\MS|\bigg(1+\sum_{i=2}^k
i\,c_i\varrho_{S_t}^{i-1}|\varpi\ci|\bigg)\per\bigg\}\,dt\\&\leq&\int_{S}|\psi||\MS|\bigg(1+\sum_{i=2}^k
i\,c_i\varrho_{S}^{i-1}|\varpi\ci|\bigg)\per.\end{eqnarray*}
Furthermore, one
has\begin{eqnarray*}\int_0^{+\infty}\mathcal{B}_\infty(S_t)\,dt&\leq&\int_0^{+\infty}\int_{\partial
S_t}\Big(1+\sum_{i=2}^k
i\,c_i\varrho_{S_t}^{i-1}|\chi\ci|\Big)\nis\,dt\\&\leq&\int_0^{+\infty}\int_{\partial
S_t}\Big(1+\sum_{i=2}^k
i\,c_i\varrho_{S}^{i-1}|\chi\ci|\Big)\nis\,dt\\&=&
\int_{S}|\qq\psi|\bigg(1+\sum_{i=2}^k
i\,c_i\varrho_{S}^{i-1}\frac{|\grad\ciss\psi|}{|\qq\psi|}\bigg)\per\qquad\mbox{(by
Coarea Formula)}
\\&=&
\int_{S}\Big(|\qq\psi|+\sum_{i=2}^k
i\,c_i\varrho_{S}^{i-1}|\grad\ciss\psi|\Big)\per\end{eqnarray*}and
the thesis easily follows.
 \end{proof}

 Now we are going to make a further assumption on the geometry
of $S$.
 The simplest hypothesis which can be made is, of course, that
  of the {\it global boundedness of the horizontal mean curvature}. The argument below is quite similar to the one used
  throughout the proof of Theorem \ref{1celafo}. As in Remark \ref{cpoints}, let ${S}_\epsilon\,(\epsilon>0)$ be a
family of open subsets of $S$ with  piecewise $\cont^2$-smooth
boundary such that:\begin{itemize}\item[{\rm(i)}]
$C_S\subset\UU_\epsilon$ for every $\epsilon>0$; \item[{\rm(ii)}]
$\sigma^{n-1}\rr(\UU_\epsilon)\longrightarrow 0$ for
$\epsilon\rightarrow
0^+$;\item[{\rm(iii)}]$\int_{\UU_\epsilon}|\PH\nu|\sigma^{n-2}\rr\longrightarrow
0$ for $\epsilon\rightarrow 0^+$.\end{itemize}Now let us fix
$\epsilon_0>0$. For every $\epsilon\leq \epsilon_0$ one has
$$\int_{\UU_{\epsilon}}|\psi||\varpi\ci|\per\leq\|\psi\|_{L^{\!\infty}(\UU_{\epsilon_0})}
\sigma^{n-1}\rr(\UU_\epsilon).$$ Then, {\rm(ii)} implies that for
every $\delta>0$ there exists $\epsilon_\delta>0$ such that
$\sigma^{n-1}\rr(\UU_\epsilon)<\delta$ whenever
$\epsilon<\epsilon_\delta$. By choosing $\widetilde{\delta}\leq
\frac{\int_{S}|\psi|\per}{\|\psi\|_{L^{\!\infty}
(\UU_{\epsilon_{0}}) }}$ one gets that
$$\int_{\UU_{\epsilon}}|\psi||\varpi\ci|\per\leq
\int_{S}|\psi|\per$$ for every $\epsilon\leq \min\{
\epsilon_{\widetilde{\delta}}, \epsilon_0\}$. Furthermore, one has
 $$\int_{S\setminus\UU_{\epsilon}}|\psi||\varpi\ci|\per\leq \|\varpi\|_{L^{\!\infty}
(S\setminus\UU_{\epsilon})}\int_{S}|\psi|\per$$for every
$\epsilon\leq \min\{ \epsilon_{\widetilde{\delta}}, \epsilon_0\}$.
Since, by hypothesis, the horizontal mean curvature $\MS$ is
globally bounded, one infers that:
\begin{itemize} \item {\it there exists $C_0>0$ such that}
$$\int_S|\psi||\MS|\bigg(1+\sum_{i=2}^k i\,c_i\varrho_S^{i-1}|\varpi\ci|\bigg)\per\leq
C_0\int_S|\psi|\per.$$\end{itemize}

By applying the above argument we get the following:
\begin{corollario}\label{issss4}Let $S\subset\GG$ be a
$\cont^2$-smooth hypersurface without boundary and assume that
$\MS$ be globally bounded along $S$. Then there exists $C_1>0$
such that
\begin{eqnarray}\label{isq4}\bigg\{\int_S|\psi|^{\frac{Q-1}{Q-2}}\per\bigg\}^{\frac{Q-2}{Q-1}}\leq
C_1\,\int_{S}\Big\{|\psi|+ |\qq\psi|+\sum_{i=2}^k
i\,c_i\varrho_{S}^{i-1}|\grad\ciss\psi|\Big\}\per.
\end{eqnarray}
for every $\psi\in\cont_0^1(S)$.
\end{corollario}\begin{proof}Immediate by the previous discussion.\end{proof}

\begin{itemize} \item[${\bf (H_1)}$] {\it Henceforth, let us assume that $\MS$ is globally bounded along
$S$.}\end{itemize}

\begin{no}As in the standard theory of Sobolev spaces, for any $p>0$, we shall set
 $$\frac{1}{p^\ast}=\frac{1}{p}-\frac{1}{Q-1}.$$Moreover, $p'$ will
 denote the H\"{o}lder conjugate of $p$, i.e. $\frac{1}{p}+\frac{1}{p'}=1$. In the sequel, any $L^p$
 norm  will be understood w.r.t. the $\per$-measure. \end{no}

\begin{corollario}\label{corsob1}Under the previous assumptions,
there exists $C_2>0$ such that
$$\|\psi\|_{L^{p^\ast}(S)}\leq C_2\bigg\{ \|\psi\|_{L^p(S)}+\|\qq\psi\|_{L^p(S)}+\sum_{i=2}^k
\|\grad\ciss\psi\|_{L^p(S)}\bigg\}$$ for every
$\psi\in\cont_0^1(S)$.
\end{corollario}

\begin{proof}Let us apply \eqref{isq4} with
$\psi$ replaced by $\psi|\psi|^{t-1}$, for some $t>0$. One has
\begin{eqnarray}\label{wfiha}\bigg\{\int_S|\psi|^{t\,\frac{Q-1}{Q-2}}\per\bigg\}^{\frac{Q-2}{Q-1}}\leq
C_1\int_{S}\bigg\{|\psi|^t+
t|\psi|^{t-1}\bigg[|\qq\psi|+\sum_{i=2}^k
|\grad\ciss\psi|\bigg]\bigg\}\per.
\end{eqnarray}If we put $t\,\frac{Q-1}{Q-2}=p'(t-1)$, it follows
that $p^\ast=t\,\frac{Q-2}{Q-1}$. Thus, by H\"{o}lder inequality
one gets
\begin{eqnarray*}\bigg\{\int_S|\psi|^{p^\ast}\per\bigg\}^{\frac{Q-2}{Q-1}}\leq
C_1\Big(\int_S|\psi|^{p\ast}\per\Big)^{\frac{1}{p'}}\bigg\{\|\psi\|_{L^p(S)}
+ t\Big[ \|\qq\psi\|_{L^p(S)}+\sum_{i=2}^k
\|\grad\ciss\psi\|_{L^p(S)}\Big]\bigg\},
\end{eqnarray*}which is equivalent to the thesis.
\end{proof}

\begin{corollario}\label{corsob12}Under the previous assumptions and notation, let
$p\in[1, Q-1[$. Then there exists $C_3>0$ such
that$$\|\psi\|_{L^{q}(S)}\leq C_3\bigg\{
\|\psi\|_{L^p(S)}+\|\qq\psi\|_{L^p(S)}+\sum_{i=2}^k
\|\grad\ciss\psi\|_{L^p(S)}\bigg\}\qquad \forall\,\,q\in[p,
p^\ast]$$ for every $\psi\in\cont_0^1(S)$.

\end{corollario}

\begin{proof}For any given $q\in[p, p^\ast]$ there exists $\alpha\in[0,1]$ such that
 $\frac{1}{q}=\frac{\alpha}{p}+\frac{1-\alpha}{p^\ast}.$ Hence$$\|\psi\|_{L^q(S)}\leq
 \|\psi\|^\alpha_{L^p(S)}\|\psi\|^{1-\alpha}_{L^{p^\ast}(S)}
 \leq\|\psi\|_{L^p(S)}+\|\psi\|_{L^{p^\ast}(S)},$$where we have
 used the usual {\it interpolation inequality} and Young's inequality.
 The thesis follows from Corollary \ref{corsob1}.\end{proof}

\begin{corollario}[Limit case: $p=Q-1$]\label{corsob13}Under the previous assumptions and notation, let
$p=Q-1$. Then there exists $C_4>0$ such
that$$\|\psi\|_{L^{q}(S)}\leq C_4
\bigg\{\|\psi\|_{L^p(S)}+\|\qq\psi\|_{L^p(S)}+\sum_{i=2}^k
\|\grad\ciss\psi\|_{L^p(S)}\bigg\}\qquad \forall\,\,q\in[Q-1,
+\infty[$$ for every $\psi\in\cont_0^1(S)$.
\end{corollario}

\begin{proof}Use \eqref{wfiha} and put $t\geq 1$.
By H\"{o}lder inequality we get
\begin{eqnarray*}&&\bigg\{\int_S|\psi|^{t\,\frac{Q-1}{Q-2}}\per\bigg\}^{\frac{Q-2}{Q-1}}\\&\leq&
C_1\bigg\{\|\psi\|^t_{L^t(S)}+ t\,\Big(
\int_S|\psi|^{p'\frac{Q-1}{Q-2}}\per\Big)^{\frac{1}{p'}}\Big[
\|\qq\psi\|_{L^p(S)}+ \sum_{i=2}^k
\|\grad\ciss\psi\|_{L^p(S)}\Big]\bigg\}
\end{eqnarray*}for every $\psi\in\cont_0^1(S)$. From this and Young's inequality
we easily get that
\begin{eqnarray*}\|\psi\|_{L^{t\,\frac{Q-1}{Q-2}}}&\leq&
\widetilde{C_1}\bigg\{[\|\psi\|^{1-\frac{1}{t}}_{L^{p'\frac{Q-1}{Q-2}}(S)}
\Big[ \|\qq\psi\|_{L^p(S)}+ \sum_{i=2}^k
\|\grad\ciss\psi\|_{L^p(S)}\Big]^{\frac{1}{t}}+\|\psi\|_{L^t(S)}\bigg\}\\&\leq&
\widetilde{C_2}\bigg\{\|\psi\|_{L^t(S)}+
\|\psi\|_{L^{p'\frac{Q-1}{Q-2}}(S)}+ \Big[ \|\qq\psi\|_{L^p(S)}+
\sum_{i=2}^k \|\grad\ciss\psi\|_{L^p(S)}\Big]\bigg\}.
\end{eqnarray*}By setting $p=Q-1$ in the last inequality one gets
\begin{eqnarray*}\|\psi\|_{L^{\frac{(Q-1)^2}{Q-2}}(S)}\leq
\widetilde{C_3}\bigg\{\|\psi\|_{L^{Q-1}(S)}+
\|\qq\psi\|_{L^{Q-1}(S)}+\sum_{i=2}^k
\|\grad\ciss\psi\|_{L^{Q-1}(S)}\bigg\}.
\end{eqnarray*}By reiterating this procedure for $t=Q, Q+1,...$
one can show that there exists $C_4$ such
that$$\|\psi\|_{L^q(S)}\leq C_4\bigg\{\|\psi\|_{L^{Q-1}(S)}+
\|\qq\psi\|_{L^{Q-1}(S)}+\sum_{i=2}^k
\|\grad\ciss\psi\|_{L^{Q-1}(S)}\bigg\}\qquad\forall\,\,q\geq
Q-1$$for every $\psi\in\cont_0^1(S)$.
\end{proof}

At this point, we would like to discuss another geometric
hypothesis and its consequences in getting ``horizontal''
Sobolev-type inequalities, i.e. only depending on the
$\HS$-component of the tangential gradient. As above, let
$S\subset\GG$ be a $\cont^2$-smooth hypersurface without boundary.
Moreover, we shall assume that:
\begin{itemize}\item[${\bf (H_2)}$]{\it For every $x\in S\setminus C_S$ and for
$\mathcal{L}^1$-a.e. $t\in ]0, R(x)[$ one has \[\dim\, C_{\partial
S_t}<n-2,\] where $S_t=S\cap B_\varrho(x, t)$.}\\
\end{itemize}

\begin{oss}Let $S\subset\GG$ be $\cont^2$-smooth and
non-characteristic and let us assume that the horizontal tangent
bundle $\HS$ be generic and horizontal, i.e. $(S, \HS)$ is a
CC-space. Then, by arguing exactly as in Remark \ref{pcar}, it
follows that ${\bf (H_2)}$ is automatically satisfied. More
generally, the validity of ${\bf(H_2)}$ follows by assuming, just
as in Remark \ref{xaz}, that:
\begin{itemize}\item[${\bf (H)}$]For every smooth
$(n-2)$-dimensional submanifold $N\subset S$ one has \[\dim\,
C_{N}<n-2.\]
\end{itemize}
\end{oss}

Let $\UU\subset S$ be any relatively compact open set with
$\cont^1$-smooth boundary $\partial \UU$, oriented by its unit
normal $\eta$.  By assuming either ${\bf (H_2)}$ or  ${\bf (H)}$,
one gets $\sigma^{n-2}\rr(C_{\partial \UU})=0$. So let
$\mathcal{V}_\epsilon\subset\partial\UU$ be an open neighborhood
of $C_{\partial\UU}$ such that
$\sigma^{n-2}\rr(\mathcal{V}_\epsilon)\longrightarrow 0$ for
$\epsilon\rightarrow 0^+$. Since
\[\chi=\sum_{i=2}^k\chi\ciss\in L^1(\partial \UU, \nis),\]  it
follows that
$$\int_{\mathcal{V}_{\epsilon}}|\chi\cidu|\nis\leq
\nis(\partial\UU)$$ for every (small enough)
$\epsilon>0\,\,(i=2,...,k)$. Furthermore, one has
 $$\int_{\partial\UU\setminus\mathcal{V}_{\epsilon}}|\chi\cidu|\nis\leq \|\chi\|_{L^{\!\infty}
(\partial\UU\setminus\mathcal{V}_{\epsilon})}\,\nis(\partial\UU)$$for
every $\epsilon>0\,\,(i=2,...,k)$. Therefore, we have shown that:
\begin{itemize} \item {\it there exist $C'_0>0$ such that}
$$\mathcal{B}_\infty(\UU)=\int_{\partial\UU}\Big(1+\sum_{i=2}^k i\,c_i\varrho_\UU^{i-1}|\chi\cidu|\Big)\nis\leq
C'_0\,\nis({\partial\UU}).$$\end{itemize}

The previous argument yields the next:

\begin{teo}\label{gsobin}Let $S\subset\GG$ be a
$\cont^2$-smooth hypersurface without boundary and let $\MS$ be
globally bounded along $S$. Moreover, let us assume either
$\bf(H_2)$ or $\bf(H)$ hold true. Then there exists $C'_1>0$  such
that
\begin{eqnarray}\label{isq56}\bigg\{\int_S|\psi|^{\frac{Q-1}{Q-2}}\per\bigg\}^{\frac{Q-2}{Q-1}}\leq
C'_1\,\int_{S}\big\{|\psi|+ |\qq\psi|\big\}\per.
\end{eqnarray}
for every $\psi\in\cont_0^1(S)$. Moreover, there exists $C'_2>0$
such that
\begin{eqnarray}\label{isq569}\|\psi\|_{L^{p^\ast}(S)}
\leq C'_2\big\{
\|\psi\|_{L^p(S)}+\|\qq\psi\|_{L^p(S)}\big\}\end{eqnarray}for
every $\psi\in\cont_0^1(S)$. If $p\in[1, Q-1[$, then there exists
$C'_3>0$ such
that\begin{eqnarray}\label{isq56s}\|\psi\|_{L^{q}(S)}\leq
C'_3\big\{ \|\psi\|_{L^p(S)}+\|\qq\psi\|_{L^p(S)}\big\}\qquad
\forall\,\,q\in[p, p^\ast]\end{eqnarray} for every
$\psi\in\cont_0^1(S)$. Finally, let $p=Q-1$. Then there exists
$C'_4>0$ such
that\begin{eqnarray}\label{isqsa69}\|\psi\|_{L^{q}(S)}\leq C'_4
\big\{\|\psi\|_{L^p(S)}+\|\qq\psi\|_{L^p(S)}\big\}\qquad
\forall\,\,q\in[Q-1, +\infty[\end{eqnarray}for every
$\psi\in\cont_0^1(S)$.
\end{teo}\begin{proof}The proof of \eqref{isq56} is almost identical to that of
Theorem \ref{sobolev}. The only difference is in the following
estimate:
\begin{eqnarray*}\int_0^{+\infty}\mathcal{B}_\infty(S_t)\,dt&\leq& C'_0\int_0^{+\infty}\nis(\partial S_t)\,dt
\\&=& C'_0\int_{S}|\qq\psi|\per\qquad\mbox{(by
Coarea Formula).}\end{eqnarray*}Then, inequality \eqref{isq56}
follows by arguing as in Corollary \ref{issss4}. The other
statements can be proved by arguing exactly as we have done
throughout  the proofs of Corollary \ref{corsob1}, Corollary
\ref{corsob12} and Corollary \ref{corsob13}.

\end{proof}

{\footnotesize \noindent Francescopaolo Montefalcone:
\\Dipartimento di Matematica, Universit\`a degli Studi di Trento,\\
  Povo (Trento)- Via Sommarive, 14 Italia
\\Dipartimento di Matematica, Universit\`a degli Studi di Bologna,\\
  Piazza di P.ta S.Donato, 5, 40126  Bologna, Italia\,
 \\ {\it E-mail address}:  {\textsf montefal@dm.unibo.it}}}

\end{document}